\documentclass[oneside, 11pt]{amsart}

\usepackage{amsmath, amssymb, amsthm, hyperref, graphicx,   verbatim, enumerate, xcolor, tikzsymbols, stmaryrd, epsfig, xypic}
\usepackage{mathrsfs}  
\usepackage{mathabx}
 \usepackage{times}
 
 \DeclareFontFamily{U}{mathc}{}
\DeclareFontShape{U}{mathc}{m}{it}%
{<->s*[1.03] mathc10}{}

\DeclareMathAlphabet{\mathscr}{U}{mathc}{m}{it}

\usepackage[paper=a4paper, marginpar=2.4cm]{geometry}
\usepackage[all]{xy}
\usepackage[utf8]{inputenc}

\usepackage[english]{babel}
 
\newcommand{\ee}{\mathbb{E}}
\newcommand{\pp}{\mathbb{P}}

\newcommand{\e}{\varepsilon}

 \newcommand{\cW}{\mathcal{W}}

 \newcommand{\disk}{\mathbb{D}}

\newcommand{\set}[1]{\left\{#1\right\}}
\newcommand{\norm}[1]{{\left\Vert#1\right\Vert}}
\newcommand{\abs}[1]{\left\vert#1\right\vert}

\newcommand{\pu}{{\mathbb{P}^1}}

\newcommand{\cX}{\mathcal{X}}
\newcommand{\rest}[1]{ \arrowvert_{#1}}

\newcommand{\unsur}[1]{\frac{1}{#1}}

\newcommand{\lrpar}[1]{\left(#1\right)}
\newcommand{\bra}[1]{\left\langle #1\right\rangle}

\newcommand{\inv}{^{-1}}

\DeclareMathOperator{\aut}{Aut}
\DeclareMathOperator{\supp}{Supp}

\DeclareMathOperator{\mult}{mult}

\DeclareMathOperator{\Crit}{Crit}
\DeclareMathOperator{\id}{id}

\DeclareMathOperator{\dist}{dist}

\DeclareMathOperator{\Hal}{Hal} 
\DeclareMathOperator{\dens}{dens} 
\DeclareMathOperator{\Stab}{Stab}

\def\C{\mathbf{C}}
\def\R{\mathbf{R}}
\def\Q{\mathbf{Q}}
\def\Z{\mathbf{Z}}
\def\N{\mathbf{N}}

\def\ii{{\mathsf{i}}}

\newcommand{\NS}{{\mathrm{NS}}}

\newcommand{\Sing}{{\mathrm{Sing}}}
\newcommand{\Tang}{{\mathrm{Tang}}}
\newcommand{\STang}{{\mathrm{STang}}}

\newcommand{\NT}{{\mathrm{NT}}}
\newcommand{\Lat}{{\mathrm{Lat}}}

\newcommand{\Ima}{{\mathsf{Im}}}
\newcommand{\Rea}{{\mathsf{Re}}}
\newcommand{\vol}{{\sf{vol}}}

\newcommand{\dil}{{\mathrm{Dil}}}

\def\Lip{\mathrm{Lip}}

 \def\Diff{{\mathsf{Diff}}}

\def\P{\mathbb{P}}

\def\Aut{\mathsf{Aut}}

 \def\PGL{{\sf{PGL}}}
\def\PSL{{\sf{PSL}}}

\def\GL{{\sf{GL}}}
\def\SL{{\sf{SL}}}

\newcommand{\m}{\mathscr{m}}

\newcommand{\X}{\mathscr{X}}
\newcommand{\cx}{\mathscr{x}}

\theoremstyle{plain}

\newtheorem{thm}{Theorem}[section]
\newtheorem{cor}[thm]{Corollary}
\newtheorem{pro}[thm]{Proposition}
\newtheorem{lem}[thm]{Lemma}

\theoremstyle{definition}
\newtheorem{defi}[thm]{Definition}

\newtheorem{eg}[thm]{Example}
\newtheorem{rem}[thm]{Remark}

\numberwithin{equation}{section}       

\addtocounter{section}{0}             
\numberwithin{equation}{section}       

\begin{document}

\setlength{\parskip}{.2em}
\setlength{\baselineskip}{1.26em}   


\begin{abstract} 
We study the hyperbolicity properties of the action 
of a non-elementary automorphism group  on a
compact complex surface,  with an emphasis on K3 and Enriques surfaces. A first result is that when such a group contains parabolic elements, Zariski diffuse 
invariant measures automatically have non-zero Lyapunov exponents. In combination with 
our previous work, this leads to simple   criteria for 
a uniform expansion property on the whole surface, for groups with and without parabolic elements. 
This, in turn, has strong 
consequences on the dynamics: description of orbit closures, equidistribution, ergodicity properties, etc. 

Along the way, we   provide a   reference discussion on uniform expansion of 
non-linear discrete group actions on compact (real) 
manifolds and the construction of Margulis functions under optimal moment conditions.  
 \end{abstract}

%
%
\title[Hyperbolicity for automorphism groups of surfaces]
{Hyperbolicity for large automorphism groups of projective surfaces}
\date{\today}

\author{Serge Cantat}
\address{Serge Cantat, IRMAR, Campus de Beaulieu,
b\^atiments 22-23
263 avenue du G\'en\'eral Leclerc, CS 74205
35042  RENNES C\'edex}
\email{serge.cantat@univ-rennes1.fr}
 \author{Romain Dujardin}
\address{Romain Dujardin,  Sorbonne Universit\'e, CNRS, Laboratoire de Probabilit\'es, Statistique  et Mod\'elisation  (LPSM), F-75005 Paris, France}
\email{romain.dujardin@sorbonne-universite.fr}

\thanks{ }
\maketitle
 
\setcounter{tocdepth}{1}
\tableofcontents

\section{Introduction}\label{sec:intro}

This article is a follow-up to  \cite{stiffness}, \cite{finite_orbits} and \cite{invariant}. 
Let $X$ be a compact complex 
surface and let  $\Aut(X)$ denote its group of 
automorphisms, i.e.\  of holomorphic diffeomorphisms. Let $\Gamma$ be a subgroup of $\Aut(X)$. 
We say that $\Gamma$ is \textbf{non-elementary} if the   subgroup 
 $\Gamma^\varstar\leq \GL(H^\varstar (X, \C))$ induced by the action of $\Gamma$ on the De Rham cohomology of $X$ contains a non-abelian free group; the existence of  
  a non-elementary subgroup of $\Aut(X)$ 
  implies that $X$ is projective (see \cite{Simons}). 
  Our purpose in this series of papers is to 
 study the dynamics  of such a non-elementary group $\Gamma$ on $X$, notably by 
 means of   random walk  techniques. 
 
 \subsection{Wehler examples} \label{subs:wehler}
 To understand the motivation behind our general results, 
 it is interesting to start  with the Wehler family $ \cW$ of  
 surfaces of degree $(2,2,2)$  in $\P^1\times \P^1\times \P^1$, which has been a 
 recurring example in our work (see e.g. \cite[\S 3]{stiffness}).  
This family $ \cW$ depends on $26$ parameters and is naturally parameterized by 
$\P^{26}(\C)$; we shall denote   by $ \cW_0\subset  \cW$ the  
Zariski open subset of smooth Wehler surfaces  which do not contain any fiber of the three coordinate projections $\P^1\times \P^1\times \P^1\to \P^1$. 
For $X\in  \cW_0$, the three natural projections $X\to \P^1\times \P^1$ are ramified covers of degree $2$; their deck transformations yield three holomorphic
 involutions $\sigma_1$, $\sigma_2$, and $\sigma_3$; 
  the group $\Gamma$ generated by these involutions  is non-elementary and isomorphic to 
  $\Z/2\Z\ast\Z/2\Z\ast\Z/2\Z$. 
   
Since every $X\in  \cW_0$ is a K3 surface, 
there is a canonical $\Aut(X)$-invariant volume form $\vol_X$ on $X(\C)$; furthermore, 
 when $X$ is defined over $\R$ there is a canonical area form $\vol_{X(\R)}$
 on $X(\R)$ which is   invariant  under the action of $\Aut(X_\R)$ (see Example~\ref{eg:K3_intro} below).  
Slightly abusing notation, we respectively denote by  $\vol_X$ and $\vol_{X(\R)}$ 
the    associated  measures on $X$ and $X(\R)$,  normalized to
have mass $1$. 

Our first main result is a complete description of orbit closures for 
most parameters $X\in  \cW_0$. 
Recall 
that a 2-dimensional real submanifold
 $Y\subset X$ is \emph{totally real} if for every $x\in Y$, $T_xY$ spans $T_xX$ as a 
complex vector space. 

\begin{thm}\label{thm:orbits_wehler}
 There exists a dense  Zariski open subset 
  $ \cW_{\mathrm{exp}}\subset  \cW_0(\C)$  
 such that for every $X\in  \cW_{\mathrm{exp}}$, there exists 
  a $\Gamma$-invariant finite set $F\subset X$ and 
   a $\Gamma$-invariant totally real analytic surface $Y\subset X$ (with possibly finitely many singular points)
   with the following property:  for   {\rm{every}} $x\in X$,
\begin{enumerate}[(a)]
\item either $x\in F$ (and its orbit is finite);
\item or $\overline{\Gamma(x)}$ is a union of connected components of $Y$;
\item or $\overline{\Gamma(x)} = X$.
\end{enumerate}
 \end{thm}   
   
In this statement both $F$ and $Y$ may be empty, depending on $X$. 
For instance, \cite[Thm A]{finite_orbits} says that $F$ is empty for a dense set of Wehler surfaces $X\in   \cW_0(\C)$. A typical situation for 
case~(b) is that $X$ is defined over $\R$ and $Y=X(\R)$. 

If we restrict to real parameters  in $ \cW$, we 
 also have a fairly complete understanding 
  of the asymptotic distribution of random orbits. 
  By this we mean the following.
 Let $\nu$ be  the probability measure  on $\Gamma$ defined by 
 $\nu = \frac{1}{3}(\delta_{\sigma_1}+ \delta_{\sigma_2}+\delta_{\sigma_3})$. For any $x$ in $X(\R)$, and for any sequence $(g_i)$ of   
automorphisms $g_i\in \Gamma$ chosen independently with distribution $\nu$, consider the trajectory $(g_n\cdots g_0(x))_{n\geq 0}$. 
Let $X'(\R)$ be a union of connected components of $X(\R)$. We say that these random trajectories, starting at $x$,  are  \textbf{equidistributed} in $X'(\R)$ if   for $\nu^\N$-almost every $(g_i)$, the empirical measures
 $\unsur{n} \sum_{k=0}^{n-1}  \delta_{g_k\cdots g_0(x)}$ converge
 to the normalized volume form induced by $\vol_{X(\R)}$ on $X'(\R)$ as $n\to\infty$. The appearance of $X'(\R)$ is due to the fact that $\Gamma$ may not act transitively on the 
components of $X(\R)$. 

\begin{thm}\label{thm:equidistribution_wehler}
There exists a dense Zariski open subset $ \cW_{\mathrm{exp}}(\R)\subset  \cW_0(\R)$  
such that for  {every} $X\in  \cW_{\mathrm{exp}}(\R)$, there exists 
a $\Gamma$-invariant finite set $F\subset X(\R)$ such that  for  {every} $x\in X(\R)$:
\begin{enumerate}[(a)]
\item either $x\in F$;
\item or the random trajectories starting at $x$ are equidistributed in a union of connected components of $X(\R)$. 
\end{enumerate}
\end{thm}

An interesting point in Theorems~\ref{thm:orbits_wehler} and~\ref{thm:equidistribution_wehler} is that their 
conclusions hold  for \emph{every} $x\in X$. Let us explain how these theorems fall within  
the progression of \cite{stiffness, finite_orbits, invariant} and what the last missing ingredient was until the present paper.  

First, the existence of the maximal finite invariant set $F$ follows from \cite[Thm C]{finite_orbits}. One key point here is that $\Gamma$ contains \emph{parabolic elements}, that is automorphisms whose 
action on $H^*(X;\C)$ is virtually unipotent and 
of infinite order (see Section~\ref{sec:preliminaries}). 

Now, the scheme of proof of Theorem~\ref{thm:equidistribution_wehler} is as follows. 
The random walk on $\Gamma$ induced by $\nu$ gives rise to a random dynamical system on $X$. 
We refer to~\cite{Kifer, br} for general references on this topic, 
and  to Sections~4 and~7 of \cite{stiffness} for our holomorphic context.
In particular, we shall use the notions of stationary and invariant measures 
$\mu$, of fibered entropy $h_\mu(X,\nu)$, etc.  
Fix $x\in X\setminus F$. By Breiman's ergodic theorem, every cluster value of the sequence of 
empirical measures 
$\unsur{n} \sum_{k=0}^{n-1} \delta_{g_k\cdots g_0(x)}$ is a $\nu$-stationary measure. We
proved in~\cite{stiffness} that every $\nu$-stationary measure is $\Gamma$-invariant, and 
in~\cite{invariant} we showed  that any invariant measure is either supported on $F$, or of the 
form $\vol_{X'(\R)}$, for some union of components  $X'(\R)$ of $X(\R)$. Therefore, any cluster value of 
$\unsur{n} \sum_{k=0}^{n-1}  \delta_{g_k\cdots g_0(x)}$ is a convex combination of point masses on $F$ and 
$\vol_{X'(\R)}$. Thus the last step is to show that if $\Gamma(x)$ is infinite, the limiting empirical 
 measures give no mass to $F$. 
 
 For Theorem~\ref{thm:orbits_wehler} the situation is similar: 
 most of the work was done in \cite[\S 8]{invariant}, except that there we could not exclude  
 that the accumulation locus of  an infinite orbit could be contained in a  finite invariant set. 
 
 These difficulties were    already addressed for homogeneous random dynamical systems in~\cite{benoist-quint3, Eskin-Lindenstrauss:RandomWalk} and in the  context of   non-linear actions on real 
 surfaces in~\cite{xliu, chung}. 
 The key is to show that if $X$ belongs to  the  dense open 
 set $ \cW_{\mathrm{exp}}$ of Theorem~\ref{thm:orbits_wehler}
  (resp. $ \cW_{\mathrm{exp}}(\R)$ of Theorem~\ref{thm:equidistribution_wehler}), the maximal finite invariant set 
   $F$ is \emph{repelling} for the random dynamics. Since we do not know the set $F$, nor its cardinality (examples of Wehler surfaces with large finite invariant sets were recently constructed 
  in~\cite{fuchs-litman-silverman-tran}), we make a large detour and prove a uniform 
  hyperbolicity property for 
  the dynamics on \emph{the whole of} $X$: this is the \emph{uniform expansion} property that we present in detail in \S~\ref{subs:UE_intro}. 
 Establishing this property relies on ergodic-theoretic arguments, the first of which is an 
 automatic hyperbolicity property that we describe in the next paragraph.

\subsection{Hyperbolicity of   invariant measures} \label{subs:hyperbolicity}

It is a fundamental  (and widely open) problem in conservative dynamics to show the typicality of non-zero Lyapunov exponents 
on a set of positive Lebesgue measure. In deterministic dynamics, a recent breakthrough is the work of 
Berger and Turaev \cite{berger-turaev}. Adding some randomness makes such a hyperbolicity result easier to obtain:
see \cite{blumenthal-xue-young} for random perturbation of the standard map, and 
\cite{barrientos-malicet, obata-poletti}
for random conservative diffeomorphisms on closed real surfaces. The  results  of Barrientos and Malicet \cite{barrientos-malicet} and of 
Obata and Poletti \cite{obata-poletti} are perturbative in nature, so they do not give explicit examples.  
In our context, the rigidity properties of holomorphic diffeomorphisms will 
enable us to exhibit explicit criteria ensuring  such a non-uniform hyperbolicity. 

 In~\cite{invariant} we have classified  
 invariant measures for non-elementary groups containing parabolic elements. We say
  that a measure $\mu$ on $X$ is
{\textbf{Zariski diffuse}}  
 if it gives zero mass to proper Zariski closed subsets. If $\mu$ is $\Gamma$-invariant and ergodic for some 
 $\Gamma\subset \Aut(X)$, this is equivalent to    its support $\supp(\mu)$ 
being  Zariski dense. Roughly speaking, our classification of invariant measures says that every Zariski diffuse, ergodic, invariant probability measure is
   given by an analytic $4$-form on  $X$ or by an analytic $2$-form on some invariant, real analytic subset $Y\subset X$ of dimension $2$.
 Here we proceed to a finer  study of the dynamical properties of these invariant measures. For this, we 
 fix a probability measure $\nu$ on $\Aut(X)$, satisfying  the moment condition 
\begin{equation}\label{eq:moment}
\tag{M}
\int \lrpar{\log\norm{f}_{C^1(X)}+ \log\norm{f\inv}_{C^1(X)}} \, d\nu (f) < + \infty,
\end{equation}
 and view any invariant measure $\mu$ 
 as a $\nu$-stationary measure, that is, 
 $\int f_\varstar\mu \, d\nu(f)  = \mu$. Then by~\eqref{eq:moment}, 
 the Lyapunov exponents of   $\mu$ are well defined.
We denote by  $\Gamma_\nu \subset \Aut(X)$   
the closed subgroup generated by  $\supp(\nu)$ (\footnote{Note that $\Aut(X)$ is discrete unless 
$X$ is a torus, see \cite[\S 3]{stiffness}
so in most cases $\Gamma_\nu = \langle \supp(\nu) \rangle$.}).

\begin{thm}\label{thm:hyperbolic}
Let $X$ be a compact complex surface and 
 $\Gamma$ be a non-elementary  subgroup of $\Aut(X)$ containing parabolic elements. 
Let $\mu$ be a Zariski diffuse  ergodic  $\Gamma$-invariant probability measure on $X$.
Let $\nu$ be any probability measure on $\Aut(X)$  
satisfying $\Gamma_\nu=\Gamma$ and   the moment condition~\eqref{eq:moment}. 

Then, viewed as a $\nu$-stationary measure, $\mu$ is hyperbolic and its  fiber entropy $h_\mu(X,\nu)$ is positive.
\end{thm}

 A variant of this result will also be obtained when $\Gamma_\nu$ contains a Kummer example 
instead of a parabolic element (see Theorem~\ref{thm:hyperbolic_kummer}). 

\begin{eg}\label{eg:K3_intro} When $X$ is a torus or a K3 surface, the canonical bundle $K_X$ is trivial and, up to multiplication by a complex number of modulus $1$, there is 
 a unique section $\Omega_X$ of $K_X$ that satisfies $\int_X\Omega_X\wedge \overline{\Omega_X}=1$. The volume form $\vol_X:=\Omega_X\wedge \overline{\Omega_X}$
 is $\Aut(X)$-invariant. Likewise, every Enriques surface $S$ inherits such an invariant volume form $\vol_S$ from its universal cover $X$ (a $2$-to-$1$ cover by a K3 surface). 
Under the assumptions of Theorem~\ref{thm:hyperbolic}, $\vol_X$ is $\Gamma$-ergodic, 
thus we conclude that {\textit{it is hyperbolic}}. Other examples are provided by  some rational surfaces (see the discussion on Coble surfaces in~\cite{Simons}). 

  In these situations the 2-form $\Omega_X$ also induces a natural measure $\vol_Y$ on any totally real surface $Y\subset X$ (see~\cite[Rmk 2.3]{invariant}). For instance, if $X$ is projective and defined over $\R$, $\Gamma$ is contained in $\Aut(X_\R)$, and $Y$ is a $\Gamma$-invariant connected component of $X(\R)$, Theorem~\ref{thm:hyperbolic} asserts that $\vol_Y$  is  hyperbolic.
\end{eg}
 
 \subsection{Uniform expansion}\label{subs:UE_intro}

Fix a riemannian metric on $X$. We say that the measure $\nu$ on $\Aut(X)$ 
  is \textbf{uniformly expanding} if there exists $c>0$ and an integer $n_0$ such that for \emph{every} $x\in X$ and \emph{every} $v\in T_xX\setminus\{ 0\}$, 
 \begin{equation}\label{eq:UE}
 \int_{\Aut(X)} \log\frac{\norm{D_xf (v)}}{\norm{v}} d\nu^{(n_0)}(f) \geq c;
 \end{equation}
here $\nu^{(n)}$ denotes the $n^\mathrm{th}$ convolution power of $\nu$. 
This notion is taken from 
\cite{chung, dolgopyat-krikorian, liu_thesis, zzhang} (see also \cite{Eskin-Lindenstrauss:RandomWalk, Prohaska-Sert:TAMS} for the linear context) and  has a number 
of strong ergodic and topological consequences on the action of $\Gamma_\nu$.  
So far, uniform expansion  has been verified only in the context of homogeneous dynamics, or for certain perturbative situations, or with the help of numerical methods. 
The geometric analysis of stationary measures developed in~\cite{stiffness} together with Theorem~\ref{thm:hyperbolic}
will be used to obtain the following result. 

\begin{thm}\label{thm:criterion_uniform_expansion_parabolic}
Let $X$ be a compact complex surface which is not rational. 
 Let $\nu$ be a  probability measure on $\Aut(X)$. Assume that:  
{\em{(i)}} $\nu$ satisfies the moment condition~\eqref{eq:moment}   and {\em{(ii)}} the group 
$\Gamma = \Gamma_\nu$  is non-elementary and  contains parabolic elements.  

Then $\nu$ is uniformly expanding if and only if the following two conditions hold:
\begin{enumerate}[ \em (1)]
\item every finite $\Gamma$-orbit is uniformly expanding;  
\item  there is no $\Gamma$-invariant algebraic curve.
\end{enumerate}
\end{thm}

Here, by definition, 
a finite orbit $F$ of $\Gamma$ is said uniformly expanding if condition~\eqref{eq:UE}  
holds for every $x\in F$.    This is the repulsion property alluded to at the end of \S~\ref{subs:wehler}. 

Checking condition~(2) is not hard in practice and boils down to cohomological computations (see \S \ref{subs:invariant_curves}). However, in a non-linear setting, 
Condition~(1) is  more delicate to verify. On the positive side, 
 if $X$ is not a torus and condition~(2) holds, then by   \cite{finite_orbits} 
 there are only finitely many finite $\Gamma$-orbits. Moreover, if $\nu$ is symmetric and satisfies a slightly stronger moment condition~\eqref{eq:moment+}, Theorem~\ref{thm:finite_UE} provides a  
 checkable necessary and sufficient condition for  a given finite 
 $\Gamma$-orbit to be uniformly expanding:
 it is equivalent to  the tangent action of $\Gamma$ being proximal and strongly irreducible. It  
 follows that when $\nu$ is symmetric (and $X$ is not rational)
  {\emph{the uniform expansion property
 depends only on $\Gamma$, and not on $\nu$}} (Corollary~\ref{cor:independence_nu}). 

On the negative side, so far there is no a priori bound  on the number of finite orbits. Nevertheless,   we show in Theorem~\ref{thm:effective} that under the assumptions of 
Theorem~\ref{thm:criterion_uniform_expansion_parabolic}, if there is no invariant algebraic curve, 
  \emph{there is a computable number $N = N(X, \Gamma)$ 
  such that  any finite orbit of length greater than $N$ is uniformly expanding} (see \S~\ref{subs:finitary} for   details on what we mean by computable here).
Consequently there is a simple algorithm for checking uniform expansion: verify  that finite orbits of length at most $N$ have non-elementary tangent action. 
 
An immediate consequence of Theorem~\ref{thm:criterion_uniform_expansion_parabolic} is the following. 
 
\begin{cor}\label{cor:no_algebraic_invariant_subset}
Under the assumptions of Theorem~\ref{thm:criterion_uniform_expansion_parabolic}, if there is no proper algebraic $\Gamma_\nu$-invariant subset, then 
$\nu$ is uniformly expanding. 
\end{cor}

Since uniform expansion is an open property in the $C^1$ topology, and since by~\cite[Thm A]{finite_orbits}
a dense set  of  Wehler examples  has no proper Zariski closed invariant set, this implies 
that uniform expansion is satisfied on an open and dense set of Wehler examples. 
To deduce Theorems~\ref{thm:orbits_wehler} and~\ref{thm:equidistribution_wehler}, we must further 
show that this   subset is open for the Zariski topology: 
for this, we use the effective result  described above, and prove
that \textit{in the Wehler family, the integer $N(X, \Gamma)$ is uniformly bounded} 
(here $\Gamma  =\bra{\sigma_1, \sigma_2, \sigma_3}$; see Theorem~\ref{thm:wehler_zariski} and Proposition~\ref{pro:N(X)_bounded}). 
In particular the equations defining $ \cW_{\mathrm{exp}}$ and $ \cW_{\mathrm{exp}}(\R)$ 
could in principle be written down explicitly.

\subsection{Ergodicity} \label{subs:intro_ergodicity}
Given an action of a general non-elementary group $\Gamma$ on a compact complex surface $X$, 
one may ask the following two basic questions: 
  does there exist a dense orbit?     is the action ergodic with respect to Lebesgue measure? 
(The latter makes   sense even when there is no invariant volume form.) 
If $\Gamma$ contains a parabolic element, by \cite{invariant} the answer to both questions is `yes', 
but without parabolic elements,  the answer is unknown.
A natural obstruction  to the existence of 
a dense orbit could be the presence   of a non-trivial Fatou 
component for $\Gamma$. No example of such a Fatou component is known so far; 
note   that  examples do exist for algebraic actions on \emph{affine} surfaces 
(see~\cite{Cantat:BHPS, rebelo-roeder}).

 As a matter of fact, the failure of ergodicity is    
associated to a lack of expansion:  indeed  a 
theorem of Dolgopyat and Krikorian~\cite{dolgopyat-krikorian} asserts 
 that a conservative uniformly expanding  action on a (real) surface must be ergodic. It is not difficult 
 to extend their argument to the complex setting (see Theorem~\ref{thm:dolgopyat-krikorian}). 
In Theorem~\ref{thm:criterion_uniform_expansion} we state  
  a   general criterion    (i.e. without parabolic elements) for   uniform expansion
which shows that under the conditions (1) and (2) of  
Theorem~\ref{thm:criterion_uniform_expansion_parabolic}, 
 the failure of uniform expansion,   is due to the existence of a 
$\Gamma$-invariant measure with exceptional properties (see  Theorems~\ref{thm:criterion_uniform_expansion} and~\ref{thm:entropy_every_f}). 
We expect it to be an extremely rare phenomenon. 
Incidentally, 
this shows that the question of ergodicity for general non-elementary groups (i.e. without parabolic elements)
 ultimately boils down to the 
classification of $\Gamma$-invariant measures. 
 
 Another  interesting consequence of our results, together with~\cite{dolgopyat-krikorian},   is that 
 a generic real Wehler example is \textbf{stably ergodic} among $C^2$ volume preserving actions, that is,  
 if $X$ belongs  to the open set $ \cW_{\mathrm{exp}}(\R)$ 
 of Theorem~\ref{thm:equidistribution_wehler} and 
  $\sigma'_1$, $\sigma'_2$, $\sigma'_3$ are  $C^2$
  volume preserving diffeomorphisms sufficiently close to  
  of $\sigma_1$, $\sigma_2$, $\sigma_3$ in the $C^1$ topology, then 
 $\Gamma' := \langle \sigma'_1,\sigma'_2,\sigma'_3\rangle $ is ergodic for $\vol_{X(\R)}$. 

 In an opposite direction, the examples  from \cite[\S 9]{invariant}  of  $\Aut(X_\R)$-invariant  
 domains  with boundary in $X(\R)$  (which admit an invariant curve)
 provide explicit  counterexamples to uniform expansion.

 \subsection{Organization of the paper}  
The first part of this paper (Sections~\ref{sec:uniform_generalities} to~\ref{sec:criterion}) is
devoted to a general study of the notion of uniform expansion  on compact (real) manifolds. 
 Much of this material is inspired from other sources; the novelty here is that we strive for 
 optimal moment  conditions.
 In Section~\ref{sec:uniform_generalities} we give several equivalent definitions of uniform expansion: 
this is inspired by   Liu~\cite{xliu} and Chung~\cite{chung}.
In Section~\ref{sec:inducing} we show that  uniform expansion is preserved 
when restricting to a finite index 
subgroup or taking a finite extension (Proposition~\ref{pro:induced_UE2}); this  
is  useful when dealing with invariant sets   made of  finitely many connected components.  Section~\ref{sec:margulis} deals with the construction of 
 \textbf{Margulis functions}. In a nutshell, a Margulis function 
 near a finite uniformly expanding invariant set $F$ is a function 
 $u:M\setminus F\to \R_+$ that tends to infinity at  $F$ and decreases on average along orbits.  
 The existence of such a function 
  guarantees that empirical measures of random orbits do not accumulate at $F$.  These functions have played an 
 important role in random dynamics since the work of Eskin and Margulis~\cite{eskin-margulis}. 
 Here, thanks to the work by Bénard and De Saxcé~\cite{benard-desaxce}, we construct such Margulis functions under optimal moment conditions (Theorem~\ref{thm:margulis}); 
 note that the usual average 
 decay property $\int u (f(x)) d\nu(f) \leq a u(x)+b$, $a<1$, is then replaced by $\int u (f(x)) d\nu(f) \leq   u(x) - \gamma$, $\gamma>0$.  This repulsion property does not hold if $F$ is an invariant submanifold (see Example~\ref{eg:margulis_submanifold}). However  in the holomorphic context,  Margulis functions 
 can be constructed for invariant 
 totally real manifolds of maximal dimension (Theorem~\ref{thm:margulis_totalement_reel}): a typical situation
  is that of $X(\R)\subset X$ for real projective manifolds. 
 In Section~\ref{sec:criterion}, we elaborate on an ergodic-theoretic criterion for uniform expansion borrowed 
 from~\cite{chung}.
 
In the second part of the paper (Sections~\ref{sec:preliminaries} to~\ref{sec:applications}),  we 
consider groups of automorphisms of projective surfaces. Theorem~\ref{thm:hyperbolic} 
is established  in Section~\ref{sec:hyperbolicity}. 
In Section~\ref{sec:characterization} we prove a general version of 
Theorem~\ref{thm:criterion_uniform_expansion_parabolic} and 
study   uniform expansion along periodic orbits; this makes essential use of the results of the first part. 
The focus in \S~\ref{subs:finitary}  is   on    finding 
algorithmically checkable  conditions for uniform expansion along finite orbits (Theorem~\ref{thm:effective}); 
 this leads to a    precise description of the locus of uniform expansion in the Wehler family (Theorem~\ref{thm:wehler_zariski}). 
 In \S~\ref{subs:thin}, we construct uniformly expanding actions by perturbing Kummer examples in the Wehler family; 
 in particular this work for ``thin'' subgroups of $\Aut(X)$ containing no parabolic element. 
 In Section~\ref{sec:applications} we study orbit closures and equidistribution   by proving 
 general versions of Theorems~\ref{thm:orbits_wehler} and~\ref{thm:equidistribution_wehler};
  we also explain the adaptation to the complex setting of the ergodicity 
 theorem of Dolgopyat and Krikorian~\cite{dolgopyat-krikorian}.
   
 The paper ends with an appendix on the rigidity of zero entropy measures. 


\subsection{Notes and comments} Theorem~\ref{thm:hyperbolic} was included in the first preprint version 
of~\cite{stiffness}.  We were informed of ongoing projects by Aaron Brown, Alex Eskin, 
Simion Filip and Federico Rodriguez Hertz, as well as Megan Roda, 
on the classification of stationary measures for uniformly expanding actions.
This should fit nicely with our work; indeed, parts of this article are written so as to to be easily combined with such a classification (see e.g. Theorem~\ref{thm:equidistribution_complex}), 

We are grateful  to Jean-François Quint for useful comments on Margulis functions. 



 \part{Uniform expansion for discrete group actions on manifolds}
 
\section{Generalities}\label{sec:uniform_generalities}
  
In this section, $M$ denotes  a compact manifold and $\nu$ 
is  a probability measure on the group $\Diff^1(M)$
of $C^1$ diffeomorphisms  of $M$. We fix a Riemannian metric on $M$. We denote by $\norm{\cdot}$ 
the norm induced by the metric on the tangent bundle $TM$, and by $T^1M$ the unit tangent bundle.

\subsection{Moment conditions} \label{subs:moment}
 If $f$ is a $C^1$-diffeomorphism of $M$, we denote by $f_\varstar $ its action on $TM$. Note that 
 if $v\in TM$ is a tangent vector based at $x$ (that is, $v\in T_xM$), then 
 $f_\varstar  v = D_xf(v)$ is based at $f(x)$. By definition, $\norm{f}_{C^1(X)}$ is the supremum of 
$ v\mapsto \norm{f_\varstar  v}$ on $T^1M$.
 For $f\in \Diff^1(M)$ we put 
 \begin{equation}\label{eq:subadditive_function_L}
 \mathrm L(f) =  \log\norm{f}_{C^1(X)}+ \log\norm{f\inv}_{C^1(X)};
 \end{equation}
this quantity  is subadditive: $\mathrm L(f\circ g)\leq \mathrm L(f)+\mathrm L(g)$.  
For $p\geq 1$ we   consider the moment conditions 
\begin{align}\label{eq:momentp}
\tag{M$_p$}
\int \mathrm L(f)^p \, d\nu (f)  <  + \infty,  \quad \quad \quad \quad\\
\label{eq:moment+}
\tag{M$_+$}
\exists p>1, \ \mathrm{(M}_p\mathrm{)}  \text{ holds},  \quad \quad \quad\quad \\
\label{eq:momentexp}
\tag{M$_{\exp}$}
\exists t>0, \  \int \norm{f}^t_{C^1(X)}+ \norm{f\inv}^t_{C^1(X)} \, d\nu (f)  < + \infty.
 \end{align}
When $p=1$, \eqref{eq:momentp} coïncides with the moment condition \eqref{eq:moment} from the introduction. For $p>1$,  \eqref{eq:momentp}
implies \eqref{eq:moment+} which implies \eqref{eq:moment}.
The subadditivity of $L$ and the convexity inequality 
$(r\inv \sum_{i=1}^r L_i)^p\leq r\inv  \sum_{i=1}^r L_i^p$ imply
\begin{equation}\label{eq:Mp_iteree}
\int  \mathrm L(f)^p \, d\nu^{(r)} (f) \leq r^{p} \int \mathrm L(f)^p \, d\nu (f) 
\end{equation}
for $p\in [1,+\infty[$ and $r\in \N^*$, where $\nu^{(r)}$ denotes the $r^\mathrm{th}$ convolution power of $\nu$.

\subsection{Notation for random compositions} \label{subs:notation}

Set $\Omega=\Diff^1(M)^\N$; its elements
 are sequences $\omega = (f_n)_{n\geq 0}$ of diffeomorphisms. We  use the
probabilistic notation $\ee(\cdot)$ and $\pp(\cdot)$ for the expectation and probability with respect to $\nu^\N$ 
on the probability space $\Omega$.  We let
$(\mathcal F_n)_{n\geq 1}$ be the increasing sequence of $\sigma$-algebras in $\Omega$
generated by cylinders of length $n$, so that  an event is $\mathcal F_n$-measurable if it depends only on 
the first $n$ terms $f_0, \ldots , f_{n-1}$ of $\omega=(f_n)_{n\geq 0}$. 
For $\omega = (f_n)_{n\geq 0} \in \Omega$ we put  $f^0_\omega=\id$ and 
\begin{equation}
f_\omega^n  = f_{n-1}\circ \cdots \circ f_0
\end{equation}
for $n\geq 1$; in particular $f_\omega^1 = f_0$. For $x$ in $M$ and $  v \in T_xM\setminus\{ 0\}$  we set
\begin{equation}
x_{\omega, n} = f_\omega^n(x) \quad {\text{and}} \quad
 v_{\omega, n} = \frac{ (f_\omega^n)_\varstar  (v)} {\norm{ (f_\omega^n)_\varstar  (v)}}\in T^1_{x_{\omega, n}}M.
\end{equation} 
For any sequence of integers 
$0= k_0< k_1< \cdots < k_p = n$   the chain rule gives
\begin{equation}\label{eq:chain_rule}
\log\norm{(f_\omega^n)_\varstar   v} = \sum_{j=0}^{p-1} \log \norm{ \lrpar{f^{k_{j+1}- k_j}_{\sigma^{k_j}\omega}}_\varstar   { v_{\omega, k_j}}}
\end{equation}

\subsection{Equivalent conditions for uniform expansion}

Recall that the probability measure $\nu$ on $\Diff^1(M)$ is   {\bf{uniformly expanding}}  
if there exists a real number $c>0$ and an integer $n_0\geq 1$ such that  
\begin{equation}\label{eq:UEbisrepetita}
\text{ for every } v\in T^1X, \quad 
 \int  \log   {\norm{ f_\varstar  ( v)}} \; d\nu^{(n_0)} (f) \geq c. 
 \end{equation}
Then, the cocycle relation for 
$ \log \frac{\norm{ f_\varstar  ( v)}}{\norm{ v}}$ implies that
 \begin{equation}\label{eq:UE_iteree}
 \int \log  {\norm{ f_\varstar  ( v)}}  \; d\nu^{(kn_0)}(f) \geq kc
 \end{equation}
 for every $k\geq 1$. Thus, $\nu$ is uniformly expanding if and only if 
$\nu^{(n)}$ is uniformly expanding for some (and hence for all) $n$. Moreover, the uniform expansion property 
 does not depend on our choice of a Riemannian metric on $M$. 
 
 \begin{rem} \label{rem:restriction_UE} 
If $\nu$ is uniformly expanding and the submanifold $N\subset M$ is invariant under every diffeomorphism in the support of $\nu$, 
then $\nu$ induces a uniformly expanding measure on $\Diff^1(N)$. 
 \end{rem}

\begin{lem}\label{lem:weak_UE}
Let $\nu$ be a  probability measure on $\Gamma$ satisfying  \eqref{eq:moment}. 
It is uniformly expanding if and only if 
 \begin{equation}\label{eq:weak_UE}
\forall  v\in T^1M, \ \exists n = n(v) \ \text{ such that  } \ 
 \int  \log \norm{f_\varstar  v } \; d\nu^{(n)}(f) >0.
 \end{equation}
\end{lem}

This is Lemma 4.3.1 of \cite{xliu}, but Liu assumes that the support of $\nu$ is compact; thus
 we briefly reproduce his  proof, assuming only \eqref{eq:moment}.  

\begin{proof} We have to show that~\eqref{eq:weak_UE} implies~\eqref{eq:UEbisrepetita}.
Since  $\abs{\log \norm{f_{\varstar}v}} \leq \mathrm L(f)$ for every $v\in T^1X$, the 
dominated convergence theorem implies that, for every $n$, 
\begin{equation}
 v\mapsto \int  \log \norm {f_{\varstar}( v)} \; d\nu^{(n)}(f)
\end{equation} is continuous. 
Thus by compactness, there exists a finite open cover $V_1, \ldots, V_p$ of $T^1M$, positive 
real numbers $c_i$, and   integers 
$n_i$ such that 
\begin{equation}
\int  \log \norm{f_\varstar  ( v) } \; d\nu^{(n_i)}(f) \geq c_i
\end{equation}
for every 
$ v\in V_i$.
Set $c_0 = \min(c_i)$ and $n_0  = \max(n_i)$. For $ v\in T^1X$ and 
$\omega\in \Omega$, define the  stopping time $\tau_1( v, \omega)$ to 
be the first integer $n\geq 1$  such that $\int  \log \norm{f_\varstar  v } \; d\nu^{(n)}(f) \geq c_0$, and then define  inductively 
\begin{equation}\tau_{k+1}( v, \omega) = \tau_k( v, \omega) + \tau_1( v_{\omega, k}, \sigma^k (\omega)).
\end{equation}
 By construction, $\tau_1$ does not depend on $\omega$ but $\tau_k$ does; in addition
    $\tau_k(v,\omega)\leq k n_0$ for all $k\geq 1$. For  $n\geq 1$, define
$K_n( v,\omega)$, or $K(n)$ for short, by  $K_n( v,\omega)=\max\set{k ; \;  \tau_k \leq n}$. Then  
 $K(n) \geq n/n_0$ and $n- K(n)\leq n_0-1$. 
 With the convention $\tau_0 = 0$, the chain rule~\eqref{eq:chain_rule} gives
\begin{align}\label{eq:stopping_bdd}
\notag\ee\lrpar{\log\norm{(f^n_\omega)_\varstar   v}} &= \ee\lrpar{\sum_{j = 0}^{K(n) -1} \log\norm{\lrpar{f_{\sigma^{\tau_j}\omega}^{\tau_{j+1} - \tau_j} }_\varstar   v_{\omega, \tau_j}}} + 
\ee\lrpar{ \lrpar{f_{\sigma^{\tau_{K(n)}}\omega}^{n- K(n)} }_\varstar   v_{\omega, \tau_{K(n)}}}\\
&\geq \frac{n}{n_0}c_0 - \max_{1\leq q\leq n_0} \ee\lrpar{ L (f_\omega^q) } \\
&\geq \frac{n}{n_0}c_0 - n_0 \int L(f)d\nu(f).
\end{align}
Thus, for $n\geq \frac{n_0}{2} + \frac{n_0^2}{c_0}\int L(f)d\nu(f)$, we have $\ee\lrpar{\log\norm{(f^n_\omega)_\varstar   v}} \geq \frac{c_0}{2}>0$ independently of $ v$, as was to be shown. 
\end{proof}

\begin{lem}\label{lem:weak_UE2}
Under the moment condition \eqref{eq:moment+}, $\nu$ is uniformly expanding if and only 
if   
\begin{equation}\label{eq:weak_UE2}
\forall v \in T^1X, \exists c>0 \ \text{ such that} \quad
\pp\lrpar{ \unsur{n}\log  \norm{ (f^n_\omega)_\varstar   v} \geq c} \underset{n\to\infty}{\longrightarrow} 1.
\end{equation}
Under the moment condition~\eqref{eq:moment}, Property~\eqref{eq:weak_UE2} implies uniform expansion.
\end{lem}

\begin{proof} 
 Let us first show that~\eqref{eq:weak_UE2} implies~\eqref{eq:weak_UE} under the assumption~\eqref{eq:moment}. 
 Fix $ v\in T^1X$,
 set $\Omega_n = \set{ \omega\in \Omega \; ; \;  \  \unsur{n}\log  \norm{ (f^n_\omega)_\varstar   v} \geq c}$, and split 
 $\ee\lrpar{\unsur{n}\log  \norm{ (f^n_\omega)_\varstar   v}}$ into the sum of an integral over $\Omega_n$ and an integral over $\Omega_n^\complement$. The first one is larger than $c \pp(\Omega_n)$, and $\pp(\Omega_n)$ tends to $1$ as $n$ goes to $+\infty$.  The second one satisfies 
  \begin{align}
  \abs{\ee\lrpar{\unsur{n}\log  \norm{ (f^n_\omega)_\varstar   v}  \mathbf 1_{\Omega_n^\complement} } } &\leq 
  \ee \lrpar{ \unsur{n} \mathrm L (f^n_\omega)  \mathbf 1_{\Omega_n^\complement} }  .
  \end{align}
The moment condition and Kingman's  subadditive ergodic theorem show that $ \unsur{n} \mathrm L(f^n_\omega)$  is uniformly integrable and converges almost surely to some finite constant; since  $ \mathbf \pp({\Omega_n^\complement}) $ converges to   $0$, we conclude that $ \ee\lrpar{\unsur{n}\log  \norm{ (f^n_\omega)_\varstar   v}}\geq c/2$ for large $n$.
  
For the converse implication we use a martingale convergence argument, as in~\cite[Lem. 4.3.5]{xliu} and~\cite[Prop. 2.2]{chung}(\footnote{Chung only assumes the moment condition \eqref{eq:moment} however it seems to us  that a stronger assumption is needed for the 
 control of the martingale differences.}). Choose $p>1$  such that 
 (M$_{p}$) holds. 
 For convenience, let us first replace $\nu$ by $\nu^{(n_0)}$, where $n_0$ is given by the espansion property~\eqref{eq:UEbisrepetita}. 
Define (for some fixed unit vector $v$)
\begin{equation}\label{eq:def_Xk}
X_k = \log\norm{  (f^1_{\sigma^k\omega})_\varstar   v_{\omega, k}} - 
\int \log\norm{   f_\varstar  ( v_{\omega, k})} d\nu(f).
\end{equation}
These increments $ X_k $ are uniformly bounded in $L^p$ because
\begin{equation}
\ee\lrpar{\abs{\log\norm{  (f^1_{\sigma^k\omega})_\varstar   v_{\omega, k}}}^p }^{1/p} \leq
 \ee\lrpar{ L(f^1_{\sigma^k\omega})^p}^{1/p} = \lrpar{\int L(f)^p d\nu(f)}^{1/p}
\end{equation}
and the second term in~\eqref{eq:def_Xk} is pointwise bounded by 
\begin{equation}
 \abs{\int \log\norm{ f_\varstar   v_{\omega, k}} d\nu(f)} \leq \int L(f) d\nu(f) \leq \lrpar{\int L(f)^p d\nu(f)}^{1/p}. 
\end{equation}
Thus, the sums $S_n = \sum_{k=0}^{n-1} X_k$ are all in $L^p$. Since $\ee({X_{n}\mid  \mathcal F_{n}})= 0$ and $S_n$ is $\mathcal{F}_{n-1}$-measurable,  $(S_n)$ is a martingale relative to the filtration $(\mathcal F_{n-1})$. 
 It follows from Theorem 2.22 in \cite[\S 2.7]{hall-heyde} that 
 $\frac{1}{n}S_n$ converges to $0$ in probability and in $L^p$. 
 Now, the chain rule gives
\begin{equation}
\frac{1}{n} S_n(\omega)= 
\unsur{n}\log \norm { (f^n_\omega)_\varstar   v} - \unsur{n} \int \log \norm { f_\varstar   v} d\nu^{(n)}(f), 
\end{equation}
and \eqref{eq:UE_iteree} asserts that 
$\int  \log {\norm{f_\varstar   v}} d\nu^{(n)}(f) \geq c n$, so we conclude that for any $c'<c$
\begin{equation}\label{eq:proba_expansion_goes_to_1}
\pp\lrpar{ \unsur{n}\log  \norm{ (f^n_\omega)_\varstar   v} \geq c'} \underset{n\to\infty}{\longrightarrow} 1,
\end{equation}
as desired. Recall however that we are working with $\nu^{(n_0)}$: 
 coming back to   $\nu$ this means that \eqref{eq:proba_expansion_goes_to_1} holds 
 along the subsequence $(nn_0)$. 
We then write $n = kn_0+ r$, with $0\leq  r \leq n_0-1$, so that 
\begin{equation}\label{eq:differential}
{(f^{n }_\omega)_\varstar   v} =  (f^{r}_{\sigma^{kn_0} \omega})_\varstar  
 (f^{kn_0}_\omega)_\varstar   v
\end{equation}
 and what we have to show is that applying $f^{r}_{\sigma^{kn_0} \omega}$ 
 does not affect 
 the linear growth of $ \log\norm{  (f^{kn_0}_\omega)_\varstar   v }$. 
 But    the   inequality~\eqref{eq:Mp_iteree}, applied with $p=1$, gives
 \begin{equation}
 \P \lrpar{\exists 0\leq r\leq n_0-1, \ \abs{\log\norm{  ( f^{r}_{\sigma^{kn_0}})_\varstar }}\geq  \e k}   
 \leq \sum_{r=0}^{n_0-1}
 \nu^{(r)}\lrpar{ L(f) \geq  \e k}
  \leq \frac{Cn_0^2}{\e k} ,
 \end{equation}
 and we are done.  
\end{proof}

\begin{rem} \label{rem:given_v} In the first part of the proof, 
the implication~\eqref{eq:weak_UE2}$\Rightarrow$\eqref{eq:weak_UE} is true for a given $v$, while the converse implication requires uniform expansion on the whole of $X$. 
\end{rem}

\begin{rem}  
This proof shows that if $\nu$ satisfies (M$_2$), then the convergence in probability in~\eqref{eq:weak_UE2} can be replaced by an almost sure convergence. (Indeed by Theorem 3 of \cite[p. 243]{Feller:vol2},  
$\frac{1}{n}S_n$ converges almost surely to $0$ when the $X_k$ are uniformly $L^2$.)
\end{rem}

\section{Inducing on a finite index subgroup} \label{sec:inducing}

\subsection{Hitting times and hitting measures  (see \cite[Chap. 5]{benoist-quint_book}) }\label{par:hitting_measure}

Let $\nu$ be a probability measure on $\Diff^1(M)$  and let $G$ be the
closed subsemigroup of $\Diff^1(M)$ generated by $\nu$. 
Let $H\subset G$ be a  closed finite  index 
subsemigroup; 
 this means that
there is a continuous and transitive action  $G\times F \to F$ on some finite set $F$ such that $H$ is the stabilizer of some element $x_0\in F$; the index of
$H$ is $[G:H]=\abs{F}$ and $F$ is the quotient space.

The hitting time $T_{H}$ of $H$ for the random walk 
induced by $\nu$  (starting from the neutral element) is  
 \begin{equation}
 T_H(\omega) = \min\set {n\geq 1, \ f^n_\omega \in H}.\end{equation}
Lemmas 5.4 and 5.5 in~\cite{benoist-quint_book} show that  $T_H$ is almost surely finite, 
admits an exponential moment, and satisfies $\ee(T_{H}) = [G:H]$.
By definition the {\bf{hitting  measure}} (or  induced measure) 
$\nu_H$ is  the probability measure on $H$ describing the distribution of $f^{T_{H}(\omega)}_\omega$.

Define   the {\bf{$k$-th   hitting time}} $T_{H,k}$ of $H$ by $T_{H,1}=T_H$ and the induction
 $T_{H,k+1}(\omega) = \min\set{n\geq T_{H,k}+1\; ; \;   f_\omega^{n} \in H}$. 
The convolution $\nu_H^{(k)}$ describes the distribution 
of  $f^{T_{H,k}(\omega)}_\omega$.
If $H$ is a finite index semigroup and  $g\in H$, $hg$ belongs to $H$ if and only if $h$ belongs to~$H$. Thus,
$T_{H, k+1}(\omega) - T_{H, k}(\omega)   = T_{H, 1}(\sigma^{T_{H, k}(\omega)}(\omega))$ and the Markov property implies that the random variables $(T_{H, k+1} - T_{H, k})$ are independent and identically distributed: each of them is distributed as $T_H$. Since their expectation equals $[G:H]$, 
the law of large numbers gives 
\begin{equation}\label{eq:limit_kth_hitting_time}
\lim_{k\to+\infty} \frac{1}{k}T_{H, k}(\omega)=[G:H]
\end{equation} $\nu^\N$-almost surely.  
  
\begin{thm}\label{thm:M$_2$} The hitting measure on a finite index subgroup satisfies the following properties 
\begin{enumerate}
\item if $\nu_H$ satisfies~\eqref{eq:momentp} for some $p\geq 1$, then so does $\nu$;  
\item if $\nu$ satisfies~\eqref{eq:momentp}, then $\nu_H$  satisfies (M$_{p'}$) for   any $1 < p' < p$;  
\item $\nu$ satisfies~\eqref{eq:moment}, or~\eqref{eq:moment+}, or~\eqref{eq:momentexp} if and only if $\nu_H$ does.
\end{enumerate}
Moreover, $\nu_H$ generates $H$ as a semigroup, which means that $H$ is the smallest closed subsemigroup of $G$ containing the support of $\nu_H$.
 \end{thm}

This result still holds if we substitute any subbaditive function to $ \log\norm{f}_{C^1(X)}$ in the definition of $L$ (see Equation~\ref{eq:subadditive_function_L}), with exactly the same proof. 
\begin{proof} 
Consider the finite quotient $F$ of $G$ by $H$ and denote the action of $G$ on $F$ by left translations by ($u\mapsto au$,   $a\in G$); by definition $H$ is the stabilizer of some $x_0\in F$. Set $K=\vert F\vert = [G: H]$.

For each $u\in F$, choose a sequence of measurable subsets $A_1(u)$, $A_2(u)$, $\ldots$, $A_k(u)$ in $G$, with $k=k(u)\leq K$ 
such that $\nu(A_i(u))>0$ for each $i$ and,  for all sequences $a_i\in A_i(u)$, $(a_k \cdots   a_1) u= x_0$ 
 while $(a_j  \cdots  a_1) u\neq x_0$ if $j<k$. Since $F$ is finite, there is   a real 
number $\e >0$ such that    $\nu(A_1(u))\cdots \nu(A_{k(u)}(u))\geq \e$ for all $u$. Shrinking the $A_i(u)$ if necessary, we may assume that 
$L(g)\leq C$ for some $C>0$ and all $g$ in $\bigcup_{u,i}A_i(u)$. 

We split the integral  of $L(f)^p$ as a finite sum
$\int L(f)^pd\nu(f) = \sum_{u\in F}\int_{\set{fx_0=u }}  L(f)^pd\nu(f) $.  If $(a_{k(u)}, \ldots, a_1)\in A_{k(u)}(u)\times \cdots \times A_1(u)$, 
then $L(f)\leq L(a_{k(u)} \cdots  a_1 f)+K C$  because $L$ is subadditive; thus,  
\begin{align}
\int_{\set{fx_0=u }}   L(f)^pd\nu(f) & \leq  \int_{\set{fx_0=u }}   (L(a_{k(u)} \cdots  a_1 f)+KC)^pd\nu(f).
\end{align}
By construction, the product $a_{k(u)} \cdots  a_1 f$ is a first return in $H$. Thus, 
integrating over the $A_i(u)$, the distribution of $a_{k(u)} \cdots  a_1 f$ contributes positively 
to $\nu_H$, and we get
\begin{align}
\e \int_{\set{fx_0=u }}   L(f)^pd\nu(f) & \leq   \int_H (L(g)+KC)^pd\nu_H(g).
\end{align}
Assertion (1) follows from this estimate.

For assertion (2), we must bound
$\int \mathrm L(f)^{p'} d\nu_H(f)= \ee ( L(f_\omega^{T_H(\omega)})^{p'} )$. By subadditivity of $L$ and convexity of $s\mapsto s^{p'}$, 
\begin{align}
 \ee  {\lrpar{L\lrpar{f_\omega^{T_H(\omega)}}^{p'} }} &\leq \ee \lrpar{\lrpar{\sum_{i=0}^{T_H(\omega)-1}  \mathrm L(f_i)}^{p'} }  \\
  \label{eq:int_T}&   \leq \ee\lrpar{  T_H(\omega)^{p'-1}  \sum_{i=0}^{T_H(\omega)-1}  {\mathrm L(f_i)}^{p'}}.
\end{align}
Let $r>1$ and $q>1$ satisfy $p =r p'$  
and   $\unsur{q}+\unsur{r}= 1$. Set $\alpha=\frac{1}{q} $.
We write $T_H^{p'-1} =T_H^{p'-1+\alpha} T_H^{-\alpha}$ and  apply Hölder's inequality to
bound~\eqref{eq:int_T} from above by
\begin{align}
&\leq \ee\lrpar{  T_H^{(p'-1+\alpha)q}}^{1/q}
\ee\lrpar{  T_H^{-\alpha r}(\omega) \lrpar{\sum_{i=0}^{T_H(\omega)-1} \mathrm L(f_i)^{p'}}^{r}}^{1/r} \\
&\leq \ee\lrpar{  T_H^{(p'-1+\alpha)q}}^{1/q}
 \ee\lrpar{  T_H^{-\alpha r}(\omega) T_H^{ r/q}(\omega) \sum_{i=0}^{T_H(\omega)-1} \mathrm L(f_i)^{p}}^{1/r}
\end{align}
 where in the last line we use the discrete Hölder  inequality, $ rp' = p$, and $r-1=r/q$. Since $\alpha =1/q$ and $r=p/p'$
  the last expression reduces to  
 \begin{equation}
 \ee\lrpar{  T_H^{\frac{(p-1)p'}{p - p'}}}^{ 1 - p'/{p}}  
 \ee\lrpar{ \sum_{i=0}^{T_H(\omega)-1} \mathrm L(f_i)^{p}}^{ {p'}/{p}} .
 \end{equation}
To conclude, we apply Lemma 5.4 of \cite{benoist-quint_book}, which says that  $\ee(\sum_{i=1}^{T_H(\omega)}\varphi \circ\sigma^i  )  = \ee(T_H)\ee(\varphi)$  for any integrable function $\varphi$,  and we arrive at the bound 
  \begin{equation}
   \int\mathrm L(f)^{p'} d\nu_1(f)   \leq  
 \ee\lrpar{  T_H^{\frac{(p-1)p'}{p - p'}}}^{ 1 - p'/{p}}  
  \ee\lrpar{  T_H }^{p'/p}
 \lrpar{  \int\mathrm L(f)^{p} d\nu(f)}^{ {p'}/p}.
\end{equation}
Since $\nu$ satisfies (M$_p$)  and 
the hitting time $T_H$ admits moments of all orders,
this last expression is finite, and the proof is complete. 

Assertion (3) follows from Assertions (1) and (2) and Corollary 5.6 in \cite{benoist-quint_book}. 

For the last assertion, fix an element $h$ of $H$ and an open neighborhood $U$ of $h$ in $H$. Since the
action of $G$ on $F=G/H$ is continuous, there is an open neighborhood $V$ of $h$ in $G$ such that 
every element of $G\cap V$ is in $U$. Since the support of $\nu$ generates a dense subsemigroup of 
$G$ the random walk induced by $\nu$ that starts at the neutral element visits $V$, hence the neighborhood 
$U$ of $h$.
Thus, $\nu_H$  generates $H$. \end{proof}

\subsection{Uniform expansion of the induced measure}

\begin{pro}\label{pro:induced_UE}
Let $\nu$ be a probability measure on $\Diff^1(M)$ satisfying \eqref{eq:moment}. Assume that $\nu$ is uniformly expanding and
let $n_0$ be as in~\eqref{eq:UEbisrepetita}.  Then, the measure induced by 
$\nu^{(n_0)}$ on $H$ is uniformly expanding. 
\end{pro}

In fact, the next proposition shows that, under  condition~\eqref{eq:moment+}, $\nu$ is uniformly expanding if
and only if $\nu_H$ is. The proof of Proposition~\ref{pro:induced_UE} is based on a simple martingale argument, while  
Proposition~\ref{pro:induced_UE2} relies on   the criterion of Lemma~\ref{lem:weak_UE2}. 
  
\begin{proof} We use ideas from \cite[\S 4.3]{xliu} and \cite[Prop. 2.2]{chung}. 
To ease notation we rename $\nu^{(n_0)}$ into $\nu$ so that~\eqref{eq:UEbisrepetita} holds with $n_0= 1$ and some $c>0$;  as above, we denote by $\nu_H$ the measure induced by $\nu$ (i.e. by $\nu^{(n_0)}$) on $H$. Fix 
$ v\in T^1X$, and  
define a  sequence  of random variables $(Y_k)_{k\geq 0}$ by 
\begin{equation}
Y_k(\omega) = \log\norm{  (f^1_{\sigma^k\omega})_\varstar    v_{\omega, k}}- c.
\end{equation}
 Then for all $k\geq 1$,
$\ee({Y_{k}\mid  \mathcal F_{k}})\geq 0$, so that the sequence $(S_n)_{n\geq 1}$
defined by $  S_n = \sum_{k=0}^{n-1} Y_k$ is a submartingale 
relative to the filtration $(\mathcal F_n)$: $\ee(   S_{n+1}\mid \mathcal F_n)\geq   S_{n}$.
The moment condition~\eqref{eq:moment} implies that 
$\ee(\abs{S_{n+1}-S_n}\mid \mathcal F_n) = \ee(\abs{Y_n}\mid \mathcal F_n)$ is uniformly bounded. 
Since the hitting time $T_H$ is integrable,  
we can apply  the optional stopping theorem \cite[Thm. 4.7.5]{durrett_book}, which  implies that  
$\ee(S_{T_H})\geq \ee(S_1)\geq 0$. Unwinding the definitions and applying the chain rule, we see  that 
\begin{equation}
\ee(S_{T_H}) = \int \log  \norm{ f_\varstar   v}\,  d\nu_H(f) - c  [G:H],
\end{equation}  
where we use $\ee (T_H) = [G:H]$. Therefore
$\int \log  \norm{f_\varstar   v}\,  d\nu_H(f)  \geq c  [G:H]>0$, and
$\nu_H$ is uniformly expanding. 
\end{proof}

\begin{pro}\label{pro:induced_UE2}
Let $\nu$ be a probability measure on $\Diff^1(M)$ satisfying  \eqref{eq:moment+}. Let $\nu_H$ be the measure induced on a closed
finite index subsemigroup.   Then $\nu$ is uniformly expanding if and only if  $\nu_H$ is uniformly expanding. 
\end{pro}

\begin{proof} 
Let us  show that if $\nu_H$ is uniformly expanding then $\nu$ is uniformly expanding.
The   converse implication is similar and is left to the reader   (in 
this direction, Proposition~\ref{pro:induced_UE} will actually be sufficient for our purposes). 
Fix $ v\in T^1M$. 
In view of Lemma~\ref{lem:weak_UE2}, 
we have to show that for some $c>0$, 
\begin{equation}\label{eq:expansion_n} 
\pp\lrpar{ \unsur{n}\log  \norm{ (f^n_\omega)_\varstar   v} \geq c} 
\underset{n\to\infty}{\longrightarrow} 1.
\end{equation}
Consider the sequence of hitting times $T_{H,k}$  defined in \S~\ref{par:hitting_measure} and 
denote it by $(T_k)$
for simplicity (hence $T_1=T_H$).
By Theorem~\ref{thm:M$_2$} $\nu_H$ satisfies \eqref{eq:moment+}, so we can apply
Lemma~\ref{lem:weak_UE2} to get a real number  $c>0$ such that 
\begin{equation}\label{eq:expansion_Tk}
 \pp\lrpar{ \unsur{k} \log\norm{\lrpar{f^{T_{k}(\omega)}}_\varstar   v} \geq c} \underset{k\to\infty}{\longrightarrow} 1. 
\end{equation}
To obtain~\eqref{eq:expansion_n} we denote by $K(n, \omega) := \max\set{k, \ T_{k}(\omega)\leq n}$ the number of
visits of $f^j_\omega$ in $H$ with $1\leq j\leq n$.
If $k\leq K(n,\omega)$, then 
 \begin{equation}\label{eq:differential2}
{(f^{n }_\omega)_\varstar   v} =  \lrpar{ f^{n-T_{k}(\omega)}_{\sigma^{T_{k}(\omega)}\omega}}_\varstar  
 \lrpar{f^{T_{k}(\omega)}_\omega}_\varstar   v \end{equation}
and as in Lemma~\ref{lem:weak_UE2} (see the lines following Equation~\eqref{eq:differential}), we have to 
show that applying $ (f^{n-T_{K(n,\omega)}(\omega)}_{\sigma^{T_{K(n,\omega)}(\omega) }\omega})_\varstar $ does not spoil 
the expansion obtained  in~\eqref{eq:expansion_Tk}. 

Recall from~\eqref{eq:limit_kth_hitting_time} that $\frac{1}{k}T_k(\omega)$
converges $\nu^\N$-almost surely to $\gamma:=[G:H]$.

\begin{lem}\label{lem:deviation_hitting_times}
There exists a constant $A$ and a sequence of measurable  subsets $\Omega^1_n\subset \Omega$ 
such that  $\pp(\Omega^1_n) \to 1$ as $n\to\infty$ and   for all $\omega \in \Omega^1_n$,
\begin{equation}\label{eq:K_n_omega}
\lfloor\frac{n}{\gamma} -   n^{3/4}\rfloor \leq K(n, \omega) 
\quad \quad \text{ and } \quad \quad \abs{n- T_{\lfloor\frac{n}{\gamma} -   n^{3/4}\rfloor}(\omega)}\leq A n^{3/4}.
\end{equation}
\end{lem}

We first conclude the proof of Proposition~\ref{pro:induced_UE2}. Let $\Omega_n^1$ be as in Lemma~\ref{lem:deviation_hitting_times}. Let $\Omega_n^2$ be the set of itineraries $\omega$ such that the lower bound of 
 \eqref{eq:expansion_Tk} holds for 
$k = \lfloor\frac{n}{\gamma} -   n^{3/4}\rfloor$. Then, $\pp(\Omega_n^1\cap \Omega_n^2)\to 1$ as $n\to +\infty$.
Fix $0 < \e < \gamma c$ and consider the set $\Omega_n^3\subset  \Omega_n^1\cap \Omega_n^2$ made of all $\omega$ such that  $\mathrm L\lrpar{ f^{n-T_{k}}_{\sigma^{T_{k} }\omega}} < \e n$ for $k=  \lfloor\frac{n}{\gamma} -   n^{3/4}\rfloor$. Then 
 \begin{equation}
\pp((\Omega^3_n)^\complement)\leq  \pp\lrpar{\mathrm L\lrpar{ f^{n-T_{k}}_{\sigma^{T_{k} }\omega}} \geq \e n}\leq 
\pp\lrpar{\max_{0\leq q\leq An^{3/4}} \sum_{i=0}^{q-1} L_i \geq \e n}
 \end{equation}
 where $(L_i)_{i\geq 0}$ is a sequence of independent random variables, each of them  distributed as $L(g)$ for $d\nu(g)$. Since the $L_i$ are non-negative, 
$\pp(\Omega^3_n)\leq\pp\lrpar{ \sum_{i=0}^{An^{3/4}} L_i \geq \e n}$. Now, the moment condition~\eqref{eq:moment} and the Markov inequality give 
 \begin{equation}
\pp\lrpar{ \sum_{i=0}^{An^{3/4}} L_i \geq \e n} \leq   \frac{An^{3/4} \norm{L(f)}_{L^1(\nu)}}{\e n} \leq C n^{-1/4} 
 \end{equation}
 for some $C>0$. Thus, $\pp(\Omega_n^1\cap \Omega_n^2\cap \Omega_n^3)\to 1$ as $n\to +\infty$, and the conclusion follows since  $\unsur{n}\log \norm{(f^n_\omega)_\varstar   v}\geq \gamma c-\e$ for all $\omega \in \Omega_n^1\cap \Omega_n^2\cap \Omega_n^3$.
 \end{proof}

\begin{proof}[Proof of Lemma~\ref{lem:deviation_hitting_times}]
The following moderate deviations estimate  follows from standard large deviations theory and the fact that $T_H$ has an exponential moment (see~\cite[3.7.1]{Dembo-Zeitouni}): 
\begin{equation}
\pp\lrpar{ \abs{\frac{1}{k}T_k - \gamma} \geq \unsur{k^{1/4}}}\leq C e^{-C k^{1/4}}.
\end{equation}
So if we set $\Omega^1_n=\set{\omega\; ; \;   \forall k\geq \sqrt{n}, \ \abs{T_k - k\gamma} \leq k^{1/4}}$, then 
$\pp(\Omega^1_n) \to 1$ as $n\to +\infty$. We claim that if $\omega\in \Omega^1_n$, 
there exists $A>0$ such that    
then~\eqref{eq:K_n_omega} holds for  all $n$ larger than some integer $n(A)$. Indeed let $n$ satisfy
$n/\gamma - n^{3/4} > \sqrt{n}$. Then for $m = {\lfloor\frac{n}{\gamma} -   n^{3/4}\rfloor}$ and $\omega\in \Omega_n^1$  we have 
$\gamma m- m^{3/4}\leq T_m\leq \gamma m+ m^{3/4} $; plugging in the value of $m$ 
we obtain the explicit bound  
\begin{equation}
 n  -  (\gamma +\gamma^{-3/4})  n^{3/4}\leq  T_{\lfloor\frac{n}{\gamma} -   n^{3/4}\rfloor} 
 \leq  n  + \frac{1}{\gamma^{3/4}} n^{3/4},
\end{equation}
and Property~\eqref{eq:K_n_omega} follows. 
\end{proof}

 \section{Margulis functions}\label{sec:margulis}

In this section we develop some tools for the proof of the equidistribution 
Theorem~\ref{thm:equidistribution_wehler}.  
Under appropriate assumptions, we show that the measures 
$\nu^n\ast \delta_x$ and $\unsur{n}\sum_{k=1}^n \delta_{f^k_\omega(x)}$ 
do not cluster at a $\Gamma$-periodic orbit, except when $\Gamma(x)$ is itself  finite. 
The basic tool is the  construction of a proper function, defined on the complement of such a periodic orbit, 
which essentially decreases along random trajectories. After \cite{eskin-margulis} it is often 
 referred to as a ``Margulis function'', even if  this strategy has a long history  in the Markov chain literature (see 
 \cite{meyn-tweedie}). Our presentation is greatly  influenced  by \cite{benoist-quint3} and 
 \cite{benard-desaxce}. 
 
\subsection{A general recurrence  criterion}
For concreteness, instead of general Markov chains, 
we consider the setting of  group actions. 
 
 \begin{thm}[Bénard-De Saxcé \cite{benard-desaxce}]\label{thm:margulis} Let $U$ 
  be a locally compact topological space. Let $\Gamma$ be a group of homeomorphisms of $U$, and $\nu$ be a probability measure on $\Gamma$.
 Assume that there exists a function $u:U\to \R_+$ satisfying the assumptions: 
 \begin{eqnarray}
 \label{eq:weak_margulis} 
 \exists A>0, \ \exists \gamma>0, \ \forall x\in U, \ u(x)\geq A \Rightarrow  
 \int u(f(x))d\nu(f) \leq u(x)  - \gamma \\
\label{eq:margulis_moment}
\exists B>0, \exists \eta>0, \forall x\in U, \ \int \abs{u(f(x)) - u(x) } ^{1+\eta} d\nu(f) \leq B.  
 \end{eqnarray}
Then for every $\e>0$ there exists $R>0$ such that for all $x$ in $U$,
 \begin{enumerate}
 \item  there exists $n_x\geq 0$, such that $(\nu^{n}\ast \delta_x )(\set{u\geq R}) \leq \e$ for all $n\geq n_x$; 
 \item   for $\nu^\N$-almost every $\omega$, 
 $$\limsup_{n\to\infty} \unsur{n} \# \set{k\in \set{1, \ldots, n} \; ; \;    u \lrpar{f^k_\omega(x) }\geq R} \leq \e.$$
 \end{enumerate}
 Furthermore the integer $n_x$ in (1) depends only on $u(x)$.
 \end{thm}

For the proof, see Proposition 1.2 in~\cite{benard-desaxce}, and the comments following it. More 
precisely, we refer to~\cite[Prop. 2.5]{benard-desaxce} for the conclusion~(1), 
including the uniformity statement on $n_x$, and to~\cite[Prop. 2.7]{benard-desaxce} for~(2). 
Even if this is not required in the proof, the function $u$ has to be understood as a proper function on 
$U$, in which case the conclusions (1) and (2) correspond to a ``non-escape of mass'' property. 

The original contraction  property  for the Margulis function $u$ in  
\cite{eskin-margulis, benoist-quint3} 
is    
 \begin{equation}
 \label{eq:margulis} 
 \exists 0<a<1, \ \exists b>0, \ \forall x\in U, \   \int u(f(x))d\nu(f) \leq au(x)  +b
 \end{equation}
instead of~\eqref{eq:weak_margulis}. One easily checks that if $u$ satisfies~\eqref{eq:weak_margulis} and the following strong integrability property 
\begin{equation}
 \label{eq:margulis_moment_exp}
\exists B>0,  \forall x\in U, \ \int \exp\lrpar{\abs{u(f(x)) - u(x) } }d\nu(f) \leq B,  
 \end{equation}
 then  $e^{\delta u}$ 
 satisfies~\eqref{eq:margulis} for small $\delta>0$. Under this assumption, Theorem~\ref{thm:margulis} was established in 
 \cite{benoist-quint3}. 
 
\subsection{Finite orbits of $C^2$ actions}
 Let $\nu$ be a probability measure on the group of $C^2$ diffeomorphisms of a compact Riemannian 
 manifold $M$ of dimension $d$. 
 As in \S~\ref{subs:moment} we  consider the moment conditions  
\begin{align}
\label{eq:moment2p}
\tag{M$_{2, p}$}
\int   \lrpar{ \log\norm{f}_{C^2 }+ \log\norm{f\inv}_{C^2 }} ^p \, d\nu (f)  <  + \infty,   \\
\label{eq:moment2+}
\tag{M$_{2, +}$}
\exists p>1, \ \mathrm{(M}_{2,p}\mathrm{)}  \text{ holds},  \quad \quad \quad\quad \quad \quad \quad\quad 
 \end{align}
\begin{rem}
For a holomorphic action on a compact 
 complex manifold, these conditions are equivalent to their respective
 $C^1$ analogues~\eqref{eq:momentp} and~\eqref{eq:moment+}, because a uniform control on the first derivatives provides a uniform control of higher  derivatives as well. 
 \end{rem}

\begin{thm}\label{thm:margulis_UE}
Let $\Gamma$ be a group of $C^2$ diffeomorphisms of a compact Riemannian 
manifold $M$, and $\nu$ be a measure on 
$\Gamma$ satisfying the moment condition (M$_{2,+}$). 
Let $F$ be a   finite orbit of $\Gamma$ such that $\nu$ is uniformly expanding on $F$.
Then  for every $x\in M\setminus F$, for every $\e>0$ 
there exists a compact set $K\Subset M\setminus F$ such that: 
\begin{enumerate}
 \item $(\nu^{n}\ast \delta_x) (K) \geq 1- \e$ for $n\geq n_x$, and 
 \item for $\nu^\N$-almost every $\omega$, 
 $$\limsup_{n\to\infty} \unsur{n} \# \set{k\in \set{1, \ldots, n}, \   f^k_\omega(x) \in K}  \geq 1-\e.$$
 \end{enumerate}
 Furthermore the integer $n_x$ in (1) is locally uniform in $M\setminus F$. 
\end{thm}

This result seems to be new: it appears under stronger (exponential) moment assumptions
in e.g. \cite{xliu, chung}. Note that such a result is not expected to hold under the 
first moment condition (M$_{2, 1}$), as explained in
Examples 1 and 2 of Section 2 in \cite{benard-desaxce}.

\begin{proof}
First,   the proof of Proposition 3.3 in~\cite{benard-desaxce} shows that if  
 the conclusions (1) and (2) hold for $\nu^{(n_0)}$, then they hold for $\nu$. So we can replace $\nu$ by $\nu^{(n_0)}$ and hence assume that    the uniform expansion property~\eqref{eq:UE} holds (on $F$) for $n_0=1$. 
 
Let $d(\cdot,\cdot)$ be the Riemannian distance on $M$. According to Theorem~\ref{thm:margulis}, we only need to show that  
$u\colon x\mapsto  - \log d(x, F)$  is a
proper function $M\setminus F\to \R_+$  satisfying   Properties~\eqref{eq:weak_margulis} 
and~\eqref{eq:margulis_moment}. 

\noindent{\bf{Preliminaries}}.-- We set 
$N(f)=\norm{f}_{C^2}+\norm{f\inv}_{C^2}$ and note that $N(f)\geq \Lip(f) + \Lip(f\inv)$ for every $f\in \Gamma$.
In particular, for every $x\in X$
\begin{equation}
\label{eq:lipschitz}
 \frac{1}{N(f)} \leq \frac{d (f(x), F)}{d(x,F)} \leq  N(f).
\end{equation}
For $R>0$,  set $\Gamma(R)=\{f\in \Gamma \; ; \; N(f) \leq R\}$.
We choose $\eta >0$ such that the moment condition~\eqref{eq:moment2p} is satisfied with $p=1+\eta$. Then, 
\begin{equation}
I_\eta:=\int_{\Gamma} \left(\log(N(f))\right)^{1+\eta} d\nu(f) 
\end{equation}
is a finite positive number. In what follows, we   choose $R>1$ such that 
\begin{equation}
\frac{2I_\eta}{(\log(R))^\eta}< \frac{c}{4}
\end{equation}
where $c$ is the expansion factor in Equation~\eqref{eq:UE} (along the finite orbit $F$). 

Take $s>0$ such that 
\begin{itemize}
\item $s$ is smaller than the injectivity radius of $M$ at $y$, for every $y\in F$;
\item the balls $B(y;s)$, for $y$ in $F$, are pairwise disjoint;
\item $C_0R^2 s < c/4$, where $c$ is the expansion factor as above, and $C_0$ is the constant   appearing
 below in the Taylor expansion (Equation~\eqref{eq:Taylor}). 
\end{itemize}
Then, define $V$ and $V'$ by 
\begin{equation}
V=\bigcup_{y\in F} B(y;s), \quad V' = \bigcup_{y\in F} B(y;s/R).
\end{equation}
By~\eqref{eq:lipschitz} we have $f(V')\subset V$ for every $f\in \Gamma(R)$. 

If $x$ belongs to $V$, we denote by $\pi(x)$ the unique point of $F$ at distance $\leq s$ from $x$, and we denote by $w_x$  the unique vector in $T_{\pi(x)}M$ such that $\exp_{\pi(x)}(w_x)=x$ and $\norm{w_x}=d(x,\pi(x))$.

\noindent{\bf{First estimate}}.-- For $f$ in $\Gamma(R)$ and $x\in V'$, Taylor's second order formula yields 
\begin{equation}\label{eq:Taylor}
\abs{d(f(x),f(\pi(x))) - \norm{f_\varstar (w_x)}}\leq C_0 N(f)d(x, \pi(x))^2,
\end{equation}
for some uniform constant $C_0$, that does not depend on $f$. This gives 
 \begin{equation}
 \abs{\frac{d(f(x),F)}{d(x, F)} - \frac{\norm{f_\varstar (w_x)}}{\norm{w_x}}} \leq C_0 N(f) d(x,F).
 \end{equation}
Now, using the Lipschitz estimate~\eqref{eq:lipschitz} and the fact that $\abs{\log(a) -\log(b)}\leq N\abs{a-b}$ when $a, b\in [N\inv,N]$, we obtain 
 \begin{equation}\label{eq:delta_1}
 \abs{\log\left(\frac{d(f(x),F)}{d(x, F)}\right) - \log\left(\frac{\norm{f_\varstar (w_x)}}{\norm{w_x}}\right)} \leq C_0 N(f)^2 d(x,F).
 \end{equation}
 By the definition of $\Gamma(R)$ and the requirements on $s$, we get 
  \begin{align}
\int_{f\in \Gamma(R)} \abs{\log\left(\frac{d(f(x),F)}{d(x, F)}\right) - \log\left(\frac{\norm{f_\varstar (w_x)}}{\norm{w_x}}\right)} d\nu(f)  & \leq C_0 R^2 d(x,F) \leq  \frac{c}{4},
 \end{align}
  because $d(x,F)\leq s$.

\noindent{\bf{Second estimate}}.-- Now, for any $f$ in $\Gamma$ we also  have 
\begin{equation}\label{eq:lipschitz_log_ratio_estimate}
 \abs{\log\left(\frac{d(f(x),F)}{d(x, F)}\right) - \log\left(\frac{\norm{f_\varstar (w_x)}}{\norm{w_x}}\right)} \leq 2\log(N(f))
\end{equation}
hence Markov's inequality  and our choice of $R$ give
\begin{align}
\int_{f\in \Gamma(R)^\complement} \abs{\log\left(\frac{d(f(x),F)}{d(x, F)}\right) - \log\left(\frac{\norm{f_\varstar (w_x)}}{\norm{w_x}}\right)} d\nu(f)& \leq \frac{2}{\log(R)^\eta}I_\eta  \leq \frac{c}{4}.
\end{align}

\noindent{\bf{Conclusion}}.--  Summing the integrals over $f$ in $\Gamma(R)$ and $\Gamma(R)^\complement$, we obtain 
\begin{equation}\label{eq:distanceF}
\int_{f\in \Gamma} \abs{\log\left(\frac{d(f(x),F)}{d(x, F)}\right) - \log\left(\frac{\norm{f_\varstar (w_x)}}{\norm{w_x}}\right)} d\nu(f) \leq \frac{c}{2}.
\end{equation}
Since  $w_x$ is a vector tangent to $M$ at $\pi(x)\in F$, 
 the uniform expansion along $F$ yields \begin{equation}\int  \log\left(\frac{\norm{f_\varstar (w_x)}}{\norm{w_x}}\right) d\nu(f) \geq c\end{equation} and 
 then~\eqref{eq:distanceF} implies that 
 \begin{equation}
 \int  - \log  {d(f(x),F)} d\nu(f) \leq  - \log(d(x, F)) - c/2.  
 \end{equation}
In other words, $u\colon x\mapsto -\log(d(x,F))$ satisfies Property~\eqref{eq:weak_margulis} (with $A=-\log(s)$). Property~\eqref{eq:margulis_moment} is obtained from~\eqref{eq:lipschitz} and the moment condition.
Thus, as announced above, $u$ satisfies the assumptions of Theorem~\ref{thm:margulis}, and we are done. \end{proof}

The local uniformity of $n_x$ in Theorem~\ref{thm:margulis_UE} has the following interesting consequence. 

\begin{pro}\label{pro:finite_mass}
Under the assumptions of Theorem~\ref{thm:margulis_UE}, any stationary Radon measure on $M\setminus F$ has finite mass. 
\end{pro}

\begin{proof}
Let $\mu$ be such a stationary measure. Fix $\e>0$, say $\e =1/2$ and 
let $K$ be as in Theorem~\ref{thm:margulis_UE}. The stationarity of $\mu$ implies that for every $n\geq 0$, 
\begin{equation}
(\nu^{(n)}\times\mu)\lrpar{ \set{(g, x), \ gx \in K}} = \mu(K), 
\end{equation}
hence for every borel set $B\subset M\setminus F$, 
\begin{equation}\label{eq:margulis_uniformity}
\int_B \nu^{(n)}\lrpar{\set{g, \ gx\in K}} d\mu(x) \leq \int_X \nu^{(n)}\lrpar{\set{g, \ gx\in K}} d\mu(x)  = \mu(K). 
\end{equation}
Now if $B$ is an arbitrary compact subset of $M\setminus F$,   the uniformity statement in 
Theorem~\ref{thm:margulis_UE}  implies that there exists $n = n_B$ such that 
for every $x\in B$, 
\begin{equation}
\nu^{(n_B)}\lrpar{\set{g, \ gx\in K}}\geq \frac12.
\end{equation}
Plugging this into~\eqref{eq:margulis_uniformity}, we obtain $\frac12 \mu(B) \leq \mu(K)$. 
Since $B$ is arbitrary, this implies that $\mu(M\setminus F)\leq 2\mu(K)$ and we are done. 
\end{proof}

\subsection{Totally real invariant manifolds}
We now consider a situation which is specific to the complex setting. 

\begin{thm}\label{thm:margulis_totalement_reel}
Let $X$ be a compact complex manifold of dimension $d$. Let $\Gamma$ be a group of holomorphic diffeomorphisms of $X$, 
endowed with a probability measure $\nu$ satisfying~\eqref{eq:moment+}.  Let $Y\subset X$  
be a $\Gamma$-invariant,
 analytic, totally real submanifold of maximal (real) dimension $d$, such that
  $\nu$ is uniformly expanding on $Y$.  
 Then for any 
$x\in X\setminus Y$ and any $\e>0$, there exists a compact subset $K\Subset X\setminus Y$ such that the 
conclusions (1) and (2) of Theorem~\ref{thm:margulis_UE} hold. 

The result also holds if $Y$ admits finitely many singular points, under the stronger
assumption that $\nu$ is finitely supported.
\end{thm}

By ``uniformly expanding along $Y$'' we mean   that the restriction of $\Gamma$ to $Y$ is 
uniformly expanding viewed as an action on $Y$, or equivalently 
that the uniform expansion condition~\eqref{eq:UE} holds in $X$ for every $x\in Y$; the equivalence between the two conditions comes from the fact that for every $x\in Y$, the 
complex span of  $T_xY$ is $T_xX$. When $Y$ is singular, we require that~\eqref{eq:UE} holds in $X$ along 
$\Sing(Y)$.  

Note also that this  statement is specific to totally real submanifolds and holomorphic actions.  In other words, 
there is no analogue of Theorem~\ref{thm:margulis}  when $F$ is replaced by an arbitrary  submanifold: 
 see Example~\ref{eg:margulis_submanifold} below.
 
\begin{proof}[Proof of Theorem~\ref{thm:margulis_totalement_reel} when $Y$ is smooth]
We suppose $Y$ smooth and show that there exists $n\geq 1$ such that 
 $x\mapsto  - \log d(x, Y)$ defines a 
Margulis function (i.e. satisfies~\eqref{eq:weak_margulis} and~\eqref{eq:margulis_moment}) 
 for $\nu^{(n)}$. Then, as explained before, \cite{benard-desaxce} shows that \eqref{eq:weak_margulis} and~\eqref{eq:margulis_moment} are automatically satisfied with $n=1$. As in Theorem~\ref{thm:margulis}, Property~\eqref{eq:margulis_moment} follows from the invariance of $Y$ and 
 the bilipschitz property; so we focus on~\eqref{eq:weak_margulis}.   

For every $x\in Y$ there exists a local chart in which 
the equation of $Y$ becomes $\Ima (z) = 0$, where  $\Ima(z) = \Ima (z_1, \ldots , z_d) = 
(\Ima (z_1), \ldots , \Ima (z_d))$ (see \cite[Prop. 1.3.8 and 1.3.11]{BER}). We fix   
a finite family $\phi_i: U_i\to \C^d$ of such charts,
covering a neighborhood of $Y$. The charts being bilipschitz,
there exists an absolute constant $D$ such that if $x\in U_i$, 
$\abs{\log d(\phi_i(x),  \phi_i(Y)) - \log d(x,Y)}\leq D$. Then  from~\eqref{eq:UE_iteree}, replacing 
 $\nu$ by $\nu^{(n)}$ we may assume that the  uniform expansion holds for $n=1$ and the expansion 
 constant $c$  is bigger than $10D$.   
 This argument shows that it is enough to prove uniform expansion for $-\log d(\cdot, Y)$ in the charts 
to infer the same property on  $X$.  

Let $d_{U_i}$ denote the euclidean distance in the $i$-th chart (pulled back by $\phi_i$). 
In $U_i$, write
$\phi_i(x) = z = (z_1, \ldots, z_d)$ and $\phi_i(Y)  = \{\Ima (z) = 0\}$. 
Let $\pi(\phi_i(x)) = (\Rea(z_1), \ldots , \Rea(z_d))$ be the projection of $\phi_i(x)$ on $Y$, 
so that 
\begin{equation} \label{eq:dui}
d_{U_i}(x,Y)  =  \norm{ \phi_i(x) - \pi(\phi_i(x))} = \norm{(\Ima (z_1), \ldots , \Ima (z_d))}  = \norm{\Ima(\phi_i(x))}. 
\end{equation}
As before let $\Gamma(R) = \set{f\in \Gamma \; ;  \  N(f)\leq R}$, where 
$N(f)  = \norm{f}_{C^2} + \norm{f\inv}_{C^2}$,
and fix $f\in \Gamma(R)$.  If $x$ is sufficiently close to $Y$, then so does $f(x)$, hence 
$f(x)$ belongs to some chart $U_j$ and working  in this chart we get
$d_{U_j}(f(x),Y) = \norm{\Ima(\phi_j(f(x)))}$. 
 Applying Taylor's formula to the coordinate expression $\tilde f$ of $f$, we obtain 
 \begin{align}
\notag \phi_j(f(x)) &=  \tilde f(\phi_i(x))  \\& \notag  = \tilde f (\pi( \phi_i(x))) + d \tilde f_{\pi( \phi_i(x))} (\phi_i(x) - \pi( \phi_i(x)))
+ O\lrpar{ \norm{ \phi_i(x) - \pi(\phi_i(x))}^2}.
 \end{align}
 Now, observe that the vector $d \tilde f_{\pi( \phi_i(x))} (\phi_i(x) - \pi( \phi_i(x)))$ is purely imaginary 
 because $\phi_i(x) - \pi( \phi_i(x))$ is purely imaginary and $d \tilde f_{\pi( \phi_i(x))}$ is real, 
 since it preserves $Y$. 
Thus,  taking imaginary parts and using~\eqref{eq:dui} yields
\begin{equation}
\abs{ \frac{d_{U_j}(f(x), Y)}{d_{U_i}(x, Y) }- \norm{df_{\pi( x)}(v_x) }}\leq C R d_{U_i}(x, Y),
\end{equation}
where
$v_x = \phi_i^\varstar\lrpar{\frac{\phi_i(x) - \pi( \phi_i(x))}{\norm{\phi_i(x) - \pi( \phi_i(x))}}}$,
$\pi(x) = \phi_i\inv \pi( \phi_i(x))$, and the constant $C$ depends only on the charts. 
Arguing as in~\eqref{eq:delta_1},   plugging in the bilipschitz estimate for the distance to $Y$, and increasing $C$ if necessary  we get  
\begin{equation}
\abs{\log \frac{d (f(x), Y)}{d (x, Y) }- \log \norm{df_{\pi( x)}(v_x) }}\leq C R^2  d(x, Y) + 2D.
 \end{equation}
Finally,   using the moment condition to deal 
with the contribution of $\Gamma\setminus \Gamma(R)$ as in Theorem~\ref{thm:margulis}, we obtain 
\begin{equation}
\int_\Gamma 
\abs{\log \frac{d(f(x), Y)}{d(x, Y) }- \log \norm{df_{\pi(x)}(v_x) }} d\nu(f) \leq C R^2 d(x, Y)  
 + 2D + \frac{C}{(\log R)^\eta},
 \end{equation}
and we conclude that $\log d(\cdot , Y)$ is a Margulis function by first fixing a large $R$ and then choosing
$x$ sufficiently close to $Y$,  as in Theorem~\ref{thm:margulis}. 
  \end{proof}
  \begin{proof}[Proof of Theorem~\ref{thm:margulis_totalement_reel} when $Y$ is singular]
When $Y$ has finitely many singularities  there is a priori no
control of the distortion of the charts near $\Sing(Y)$, so the argument must be modified.  

Fix $x_0\in X\setminus Y$ and $\e>0$. We seek a compact subset $K\Subset X\setminus Y$ such that the conclusions~(1) and~(2) of Theorem~\ref{thm:margulis_UE} are satisfied for $K$ and $x_0$.
Theorem~\ref{thm:margulis_UE} provides an open neighborhood $V_1$ of $\Sing(Y)$ such that 
these conclusions  
 hold for 
$K_1  = V_1^\complement$. Fix   a neighborhood $V'_1$  of 
$\Sing(Y)$  such that $\overline{V'_1}\subset V_1$; let 
$V$ be a small neighborhood of $Y$ and  set $V_2 = V\setminus \overline{V'_1}$. 
We will construct a proper Margulis function $u$ on $X\setminus (Y\cap V_2)$. Then,  
Theorem~\ref{thm:margulis}  provides a compact set $K_2\subset  X\setminus (Y\cap V_2)$ such that~(1) and~(2) hold for $K_2$. Therefore the desired conclusions hold for $K:=K_1\cap K_2$, with $2\e$ 
instead of $\e$.  

 To construct the desired  Margulis function on $X\setminus (Y\cap V_2)$, we put 
\begin{equation}
\begin{cases}
u(x)  =  - \log d(x, Y) \text{ for } x\in V_2\\
u(x)  = \min (u\rest{V_2}) - B \text{ for }x\notin V_2
\end{cases}
\end{equation}
where $B = \sup\set{\log N(f), \ f\in \supp(\nu)}$. 
Again, checking~\eqref{eq:margulis_moment} is immediate, so we focus on~\eqref{eq:weak_margulis}. 
Since $u$ is constant outside $V_2$, we have to show that for $x\in V_2$ sufficiently close to 
$Y$, 
\begin{equation}\label{eq:margulis_inequality}
\int_\Gamma \lrpar{ - u(f(x)) + u(x)   - \log \norm{df_{\pi(x)}(v_x} } d\nu(f) \geq -\frac{c}{2},
\end{equation}
where $c$ is the uniform expansion constant, and $\pi(x)$ and $v_x$ are as above.
As in the smooth case, the distortion between the distances in charts  
and the ambient distance is bounded on $V_2$  by a constant 
$D = D(V_2)$; we iterate to get $c>10D$.
We split~\eqref{eq:margulis_inequality}   into 
 \begin{equation}
\int_{\Gamma} = \int_{\set{f\in \Gamma, \ f(x)\in  V_2}}  +  \int_{\set{f\in \Gamma, \ f(x)\notin  V_2}} . 
\end{equation}
For the first integral we argue as in the smooth case to conclude that 
\begin{equation}
\int_{\set{f\in \Gamma, \ f(x)\in  V_2}} 
\abs{- u(f(x)) + u(x)   - \log \norm{df_{\pi(x)}(v_x}}  d\nu(f)\leq c/2
\end{equation}
when $d(x, Y)$ is small enough. For the second integral we simply use the fact that if 
$f(x)\notin V_2$, then  $u(x)- u(f(x)) \geq B$ so 
\begin{align}
\int_{\set{f\in \Gamma, \ f(x)\notin  V_2}}& \lrpar{ - u(f(x)) + u(x)   - \log \norm{df_{\pi(x)}(v_x} } d\nu(f) 
\\ \notag &
\geq \int_{\set{f\in \Gamma, \ f(x)\notin  V_2}}  \lrpar{  B  - \log \norm{df_{\pi(x)}(v_x} }d\nu(f),
\end{align}
and this last term is non-negative from our choice of  $B$. The proof is complete. 
\end{proof}

\begin{eg}\label{eg:margulis_submanifold}
\textit{There exists a group $\Gamma  = \bra{f, g}$ of diffeomorphisms of the 3-torus $\R^3/\Z^3$ and a finitely supported measure $\nu$ on $\Gamma$ with $\langle \supp(\nu)\rangle   = \Gamma$ such that:
\begin{itemize}
\item $\Gamma$ preserves $Y:=\R^2/\Z^2\times \set{0}$;
\item there exists a neighborhood $U$ of $Y$ on which the dynamics of $(\Gamma, \nu)$ is uniformly expanding;
\item for every $x\in U$ and almost every trajectory $\omega$, $f_\omega^n (x)$ converges to $Y$. 
\end{itemize}
}
\end{eg}

\begin{proof}
Let $0<\gamma<1$ and $\psi$ be a diffeomorphism of the circle $\R/\Z$, fixing 0, and conjugate to 
$t\mapsto \gamma t$ on $]-1/4, 1/4[\subset \R/\Z$  
by a diffeomorphism $\varphi: ]-1/4, 1/4[\to \R$  such that $\varphi(0)  = 0$,  
$\varphi(t) = t$  on $[-1/8, 1/8]$, and 
$\varphi([-1/8, 1/8]) = [-1/4, 1/4]$.  Note that  
$\psi(t)   = \gamma t$ on $[-1/8, 1/8]$ and $\psi$ preserves $[-1/4,1/4]$ (inducing a diffeomorphism of this interval). Pick $A, B\in \SL(2, \Z)$ generating a non-elementary subgroup, and 
define two diffeomorphisms $g$ and $h$  on  $\R^3/\Z^3$ by 
\begin{equation}
g(x,y,z) = (A(x,y)+(c_1z, c_2z), \psi(z))\;  \text{ and } \; h(x,y,z) = (B(x,y), \psi(z)). 
\end{equation}
We further assume that 
\begin{equation}\label{eq:eigenvalue}
(c_1, c_2)\neq (0, 0) \text{ and } \gamma \text{ is not an eigenvalue of }A.
\end{equation}
Let $\nu$ be a probability measure supported on $\set{g, h, g\inv, h\inv}$ such  that 
$0<\nu(g\inv)<\nu(g)$ and $0<\nu(h\inv)<\nu(h)$. Then there exists $\Omega_0\subset \Omega$ of full 
$\nu^\N$-measure such that for every $p = (x,y,z)\in \R^2/\Z^2 \times]-1/4, 1/4[$, 
 and $\omega\in \Omega_0$, $f_\omega^n(p)\to Y$. Indeed, writing $f_\omega^n(p) = (x_n, y_n, z_n)$, 
 we have: 
 \begin{enumerate}
 \item $z_n\in ]-1/4, 1/4[$ because $\psi$ preserves $]-1/4, 1/4[$;
 \item $\varphi(z_n ) = \gamma^{\sum_{i=1}^{n} \e_i} \varphi(z)$, where 
 $(\e_n)$ is a sequence of independent random variables with 
 $\pp(\e =1 )  = \nu(g)+\nu(h)$ and  $\pp(\e =-1 )  = \nu(g\inv)+\nu(h\inv)$. Since, $\nu(g\inv)+\nu(h\inv)< \nu(g)+\nu(h)$, $\phi(z_n)$ converges almost surely to $0$.
 \end{enumerate}

To conclude, we have to show that the dynamics of $(\Gamma, \nu)$  is uniformly expanding in 
$\R^2/\Z^2 \times ]-1/4, 1/4[$. Indeed, if $p\in \R^2/\Z^2 \times ]-1/4, 1/4[$ and  $\omega\in \Omega_0$, there 
is $n(\omega)$ such that  $f^n_\omega(p)\in \R^2/\Z^2 \times ]-1/8, 1/8[$ for $n\geq n(\omega)$. 
Now, in $\R^2/\Z^2 \times ]-1/8, 1/8[$ the dynamics is linear, and the tangent action is generated by 
\begin{equation}
\tilde g = \begin{pmatrix}A & \begin{pmatrix} c_1 \\ c_2 \end{pmatrix} \\ 0 & \gamma
\end{pmatrix}
\text{ and } 
\tilde h  = \begin{pmatrix}B& 0 \\ 0 & \gamma
\end{pmatrix}.
\end{equation}
We claim that the linear action of $(\tilde \Gamma, \tilde \nu)$ on $\R^3$ is uniformly expanding, 
where $\tilde \Gamma = \langle{\tilde g, \tilde h}\rangle$ and  $\tilde \nu$ is the measure naturally 
corresponding to $\nu$. Indeed, the action is uniformly expanding on $\R^2\times \set{0}$ and 
if it were  not uniformly expanding on $\R^3$, by Furstenberg-Kifer~\cite{furstenberg-kifer},
 there would exist a $\tilde \Gamma$-invariant  line transverse to $\R^2\times \set{0}$ along which the Lyapunov exponent would be non-positive. But the  hypotheses~\eqref{eq:eigenvalue} guarantee that such a line does not exist. 
From this,
 we deduce that there exists $c>0$ (any constant smaller than the Lyapunov exponent 
of the random product generated by $A$ and $B$ will do)
such that 
for every $p \in \R^2/\Z^2 \times ]-1/4, 1/4[$, every unit tangent vector $v$ at $p$ 
 and almost every $\omega$, $\unsur{n} \log \norm{(f_\omega^n)_\varstar v}\geq c$ if  $n$ is large enough.
Applying   Lemma~\ref{lem:weak_UE2} finishes the proof. 
\end{proof}

 \section{An ergodic-theoretic criterion for expansion}\label{sec:criterion}

 \subsection{Construction of stationary measures}
 \label{subs:stationary}

Let $M$ be a compact manifold 
endowed with a riemannian metric; let $T^1M$ denote its unit tangent bundle 
and  $\pi\colon T^1M\to M$ be the canonical projection. As in  Section~\ref{sec:uniform_generalities}, 
if $f$ is a diffeomorphism of $M$, we denote by $f_\varstar $ its action on $TM$. 
 Let $\nu$ be a probability measure on $\Diff^1(M)$
  satisfying the moment condition~\eqref{eq:moment}. 
We   apply a classical strategy  to get the
following theorem (see e.g.~\cite[Prop.  3.17]{chung}, and~\cite[Lem. 3.3]{hurtado:burnside}).

\begin{thm}\label{thm:limit_measure_top_lyapunov}
Assume that there exists an increasing sequence $(n_k)\in \N^\N$ and a 
sequence of unit tangent vectors $( u_k)\in (T^1M)^\N$ 
such  that 
\begin{equation}\label{eq:exponential_dilatation}
\lim_{k\to\infty} \frac{1}{n_k} \int \log \norm{f_\varstar  u_k} d\nu^{(n_k)}(f) = \chi_0. 
\end{equation}
Then, there exists a real number $\chi\geq \chi_0$, an ergodic $\nu$-stationary probability measure $\hat{\mu}$ on $T^1M$, 
and a $\nu$-almost surely invariant  sub-bundle $V\subset TM$ such that the top Lyapunov exponent of the  projected measure  $\mu:=\pi_\varstar \hat{\mu}$ in restriction to $V$ is equal to $\chi$. Likewise, there exists 
 a real number $\chi' \leq \chi_0$ that satisfies the same property for 
some pair $(\hat{\mu}', V')$. 
\end{thm}

Note that if $\hat\mu$ is a  probability measure  on $T^1M$ that is 
$\nu$-stationary  for the tangent action, then its projection $\mu $ on $M$ is $\nu$-stationary as well; and if $\hat\mu$ is ergodic, so is $\mu$. 
When $\chi >0$, one typically obtains $V=TM$.

\begin{proof}[Proof (see~\cite{chung, hurtado:burnside})] 
Consider the 
sequence of  measures $\hat\mu_k$ on $T^1M$ defined by 
\begin{equation}
\hat\mu_k=\frac{1}{n_k} \sum_{j=0}^{n_k-1} \nu^{(j)}\star \delta_{u_k} =\frac{1}{n_k} \sum_{j=0}^{n_k-1}\int \frac{f_\varstar   u_k}{\norm{f_\varstar   u_k}} d\nu^{(j)}(f),
\end{equation}
where $\nu^{(j)}\star \delta_{u_k}$ denotes the convolution for the action of $\Diff(M)$ on the unit tangent bundle. Since $T^1M$ is compact and the
$\hat\mu_k$ are probability measures, we can extract a subsequence (still denoted by $\hat\mu_k$ for simplicity) that converges weakly towards a
probability measure $\hat\mu_\infty$ on $T^1M$. By construction, this measure is $\nu$-stationary. 

The function $\dil(f, u):=\log\norm{f_\varstar   u}$ is continuous on 
$\Diff^1(M)\times T^1M$.  For $ u\in T^1M$ the chain rule gives 
\begin{align}\notag
\frac{1}{n} \int \log\norm{(f^n_\omega)_\varstar   u} d\nu^{\N}(\omega) & = \frac{1}{n} \sum_{j=0}^{n-1} \int \dil\left( f_j, \frac{(f^j_\omega)_\varstar   u}{\norm{(f^j_\omega)_\varstar   u}}\right) d\nu^{\N}(\omega) \\
& =  \int_{g\in \Diff(M)} \lrpar{\frac{1}{n} \sum_{j=0}^{n-1} \int \dil\left( g, \frac{h_\varstar   u}{\norm{h_\varstar   u}}\right)  d\nu^{(j)}(h) }  d\nu(g)
\end{align}
If we apply this equation to $n=n_k$ and $u=u_k$ the term between parentheses in 
the last integral is equal to 
$
\int \dil\lrpar{g,u} d\hat\mu_k(u),
$
so, letting $k$ go to $+\infty$, we conclude that 
\begin{equation}
\lim_{k\to\infty} \frac{1}{n_k} \int \log \norm{f_\varstar  u_k} d\nu^{(n_k)}(f) = \chi_0= 
 \int_{\Diff^1(M)}\int_{T^1M} \dil(g,u) d\hat\mu_\infty(u) d\nu(g)  
\end{equation}
Thus, there exists $\chi\geq \chi_0$ (resp. $\chi\leq \chi_0$) and 
 an ergodic component $\hat\mu$ of $\hat\mu_\infty$ 
such that
 \begin{equation}
 \int_{\Diff^1(M)}\int_{T^1M} \dil(g,u) d\hat\mu (u) d\nu(g) = \chi.
\end{equation}

As observed above,  $\mu = \pi_\varstar  \hat \mu$ is an  
ergodic $\nu$-stationary probability measure. Denote by $\hat\mu_x$ the conditional measures obtained    
by disintegration of $\hat\mu$ with respect to the fibers of $\pi$, that is,
$\hat\mu  = \int \hat\mu_x d\mu(x)$. For $\mu$-almost every $x$, let  $V(x)$ be the linear span of 
$\supp(\hat \mu(x))$. Since $\supp(\hat\mu)$ is $\nu$-almost invariant and 
$f_\varstar $ acts linearly along the fibers of $TM$, we infer that $V$ is a $\nu$-almost invariant 
measurable sub-bundle. The Furstenberg formula asserts that 
 the top Lyapunov exponent of $\mu$ in restriction 
to $V$ is equal to $\chi$. For completeness let us recall the argument:
the ergodic theorem shows that for $(\nu^\N\times \mu)$-almost every $(\omega, x)$ 
and $\hat\mu_x$-almost every $u\in T^1_xM$, 
\begin{equation}\label{eq:birkhoff}
\lim_{n\to + \infty} \frac{1}{n} \sum_{j=0}^{n-1} \dil\left(  f_j, (f^j_\omega)_\varstar   u \right) = \int_{\Diff^1(M)^\N} \int_{T^1M} \dil(f^1_\omega,  u) \, d\hat\mu( u) d\nu^{\N}(\omega) = \chi
\end{equation}
where as usual 
$\omega=(f_0, f_1, \ldots)$, $f^1_\omega=f_0$, and $f^j_\omega= f_{j-1}\circ \cdots \circ f_0$. 
 On the other hand the Oseledets theorem asserts that  for 
 $(\nu^\N\times \mu)$-almost every $(\omega, x)$, there exists a proper subspace 
 $W(\omega, x)\subset V(x)$ such that for $u\notin  W(\omega, x)$, 
 $\unsur{n} \log\norm{(f^n_\omega)_\varstar  u}$ converges to the top Lyapunov exponent $\chi^+(\mu, V)$ 
 of $\mu$ in restriction to $V$. Thus by~\eqref{eq:birkhoff}, $\chi^+(\mu, V)=\chi$, and the proof is complete. 
\end{proof}

\subsection{Application: Chung's criterion}

The following theorem, taken from \cite[Prop 3.17]{chung},
 plays an important role in this paper; a variant  of this result appears in \cite{brown-fisher-hurtado}. 
  It is stated in \cite{chung} for $C^2$ actions on surfaces but 
it holds in greater generality. The proof follows directly from the second assertion of Theorem~\ref{thm:limit_measure_top_lyapunov}.  
 
\begin{thm}[Chung]  \label{thm:criterion_chung}
Let $M$ be a compact manifold. 
Let $\nu$ be a probability measure on $\Diff^1(M)$
that satisfies \eqref{eq:moment}. If $\nu$ is not uniformly expanding there exists an ergodic $\nu$-stationary measure $\mu$ on $M$ and a $\mu$-measurable
subbundle $W\subset TM$ such that 
\begin{enumerate}[{(a)}]
\item $0<\dim(W)\leq \dim(M)$;
\item  $W$ is $\nu$-almost surely invariant;
\item in restriction to $W$, the top Lyapunov exponent of $\mu$ is non-positive. 
\end{enumerate}
Conversely, if such a  pair $(\mu, W)$ exists, then $\nu$ is not uniformly expanding.
\end{thm}

When $M$ is a  surface and  $\nu$ is supported by the group of diffeomorphisms preserving some fixed area form
the Lyapunov exponents of any ergodic stationary measure $\mu$ satisfy $\lambda^+(\mu)+\lambda^{-}(\mu)=0$. Thus,  
in Chung's theorem, either $\lambda^{-}(\mu)=\lambda^{+}(\mu)=0$ and we can take $W=TM$ or $\lambda^{-}(\mu)<0<\lambda^+(\mu)$ and $W$ coïncides with the stable line field 
provided by the 
Oseledets theorem; thus, $\mu$ is not hyperbolic or it is hyperbolic and its stable line field is non-random.

\part{Non-elementary actions on   complex surfaces}
 
 From now on we denote by $X$   a compact complex surface, endowed with a 
 group $\Gamma$ of  holomorphic 
 diffeomorphisms. Recall from~\cite{stiffness} that if $\Gamma$ is non-elementary, then 
 $X$ is necessarily projective and $\Gamma\subset \aut(X)$.

 \section{Preliminaries}\label{sec:preliminaries}
 
In this section we briefly recall some  results from  \cite{invariant} (see also \cite{cantat_groupes, Cantat:Milnor}). 

 \subsection{Parabolic automorphisms and their dynamics (see~\cite[\S 3]{invariant})} \label{subs:halphen}

Let $h$ be a parabolic automorphism of a compact projective surface $X$ (most of this discussion is valid for a 
compact Kähler surface).  Then, $h$  preserves a genus $1$ fibration 
$\pi_h\colon X\to B$, and every $h$-invariant holomorphic (singular) foliation -- in particular any invariant fibration -- coïncides with $\pi$. Let $h_B$ denote the automorphism of $B$ such that 
\begin{equation}
\pi\circ h=h_B\circ \pi.
\end{equation}
If $X$ is not a torus there is a positive integer $m$ such that $h^m$ preserves every fiber of $\pi$, i.e. $h_B^m=\id_B$. When $h_B = \id_B$ we say that $h$ is a {\bf{Halphen twist}}. The set of Halphen twists in a given subgroup $\Gamma\subset \Aut(X)$ is denoted by $\Hal(\Gamma)$. 

\begin{rem}\label{rem:many_parabolic}
If $\Gamma$ is non-elementary and contains a Halphen twist  (resp. a parabolic automorphism) $h$, then the conjugacy class of $h$ in $\Gamma$ contains Halphen twists (resp. parabolic automorphisms) associated with  infinitely many distinct invariant fibrations (see~\cite[\S 3.1]{finite_orbits}).
\end{rem}

Suppose now that $h$ is a Halphen twist. 
Then, $h$ acts by translation on every smooth fiber of $\pi$ (see~\cite[\S 3]{invariant}).
To be more precise, denote by $\Crit(\pi)\subset B$ the finite set of critical values of $\pi$ and set $B^\circ= B\setminus \Crit(\pi)$.
Fix some simply connected open subset $U\subset B^\circ$,  endowed  with a section $\sigma$ of 
 $\pi$ and a continuous choice of basis for $H_1(X_w, \Z)$.
Each fiber $X_w:=\pi^{-1}(w)$, $w\in U$, is an elliptic curve with zero $\sigma(w)$, and one can find a holomorphic function $\tau$ on $U$, with values in the upper half plane, such that $X_w$ is isomorphic to $\C/\Lat(w)$ for $\Lat(w)=\Z\oplus \Z\tau(w)$. On $X_w$, $h$ is  a translation $h_w(z)=z+t(w)$, for some holomorphic function $w\in U\mapsto t(w)\in \C/\Lat(w)$.  Moreover, Lemma~\ref{lem:halphen_flat}(4) says that $h$ behaves like a ``complex Dehn twist'', with a shearing property in the direction which is transversal to the fibers; thus shearing (or twisting) occurs along $X_w$ whenever $t$ and $\tau$ are ``transverse'' at $w$ 
(see \S~\ref{subs:finitary} for more details on the non-twisting locus). 

The points $w$ for which $h_w$ is periodic are characterized by the relation $t(w)\in \Q\oplus \Q\tau(w)$. If 
\begin{equation}
t(w)-(\alpha+\beta\tau(w))\in  \R \cdot (p+q\tau(w))
\end{equation}
for some $(\alpha,\beta)\in \Q^2$ and $(p,q)\in \Z^2$, the closure of $\Z t(w)$ in $\C/\Lat(w)$ is an abelian Lie group of dimension $1$, 
isomorphic to $\Z / k\Z \times \R /\Z$ for some $k>0$; then, the closure of each orbit of $h_w$ is a union of $k$ circles. This occurs along a countable union of analytic curves $\mathrm{R}^{\alpha, \beta}_{p, q}\subset U$. Otherwise, the orbits of $h_w$ are dense in $X_w$, and the unique $h_w$ invariant probability
measure is the Haar measure on $X_w$.

The following lemma summarizes this discussion. 

\begin{lem}\label{lem:halphen_flat} Let $h$ be a Halphen twist with invariant fibration 
$\pi:X\to B$. Then,
\begin{enumerate}[\em (1)]
\item $h$ acts  by translation on each fiber $X_w=\pi^{-1}(w)$, $w \in B^\circ$;
\item for $w$ in a dense countable subset of $B^\circ$, the orbits of $h_w$ are finite;
\item there is a dense, countable union of analytic curves $\mathrm R_j$ in $B^\circ$, 
such that:
\begin{enumerate}
\item for $w\notin\bigcup_j \mathrm R_j$, the action of $h$ in the fiber $X_w$ is a totally irrational 
translation (it is uniquely ergodic, and its orbits are dense in $X_w$); 
\item for $w\in \bigcup_j \mathrm R_j$ the orbits of $h_w$ are either finite or dense in a finite union of circles;
\end{enumerate}
\item there is a finite subset $\mathrm{NT}_h$ such that for  $x\notin \pi\inv\lrpar{
\mathrm{NT}_h}$
\[
\lim_{n\to \pm \infty} \norm{D_xh^n}\to +\infty 
\]
locally uniformly in $x$; more precisely for every 
 $v\in T_xX \setminus T_x X_{\pi(x)}$, $\norm{D_x h^{n}(v)}$ grows linearly while 
 $\frac{1}{n}\pi_\varstar (D_xh^{n}(v))$ converges to $0$. 
\end{enumerate}
If moreover  $h$ preserves a totally real $2$-dimensional real analytic subset $Y\subset X$, then:
\begin{enumerate}[\em (5)]
\item  the generic fibers of $\pi\rest{Y}$  are union of circles,
there exists an integer $m$ 
such that $h^m$ preserves each of these circles, and $h^m$ is uniquely ergodic along each of these circles,  
 except for countably many fibers.  
\end{enumerate} 
\end{lem}

Property~(4) is the above mentioned twisting property of $h$. Property~(5) occurs, for instance, when $X$ and $h$ are defined over $\R$
and $Y=X(\R)$ is the real part of $X$. There  are also examples of 
subgroups $\Gamma\subset \Aut(X)$ preserving a totally real surface $Y\subset X$ which is not the real part of $X$ for any real structure, see \cite[\S 9]{invariant}).

\subsection{Classification of invariant measures}

Recall from Example~\ref{eg:K3_intro} that if $X$ is a torus, a K3 surface, or an Enriques surface it admits a canonical $\Aut(X)$-invariant volume form $\vol_X$. The associated probability measure will also be denoted by $\vol_X$. 
Such an area form exists also on any totally real surface, by virtue of the following lemma. 

\begin{lem}[{see \cite[Remark 2.3]{invariant}}]\label{lem:volume_Y}
Let $X$ be an Abelian surface, or a K3 surface, or an Enriques surface with universal cover $\tilde{X}$. 
Let $Y\subset X$ be a totally real surface of class $C^1$, and  $\Aut(X;Y)$ be the 
subgroup of $\Aut(X)$ preserving $Y$. 
If $Y$ is totally real,  the canonical holomorphic 2-form 
 $\Omega_X$ (resp. $\Omega_{\tilde{X}}$) induces a smooth $\Aut(X;Y)$-invariant probability 
 measure $\vol_Y$ on $Y$. 
 \end{lem}

\begin{thm}[{see~\cite[Thm A]{invariant}}]\label{thm:classification_invariant}
 Let $X$ be a projective  surface. Let $\Gamma$ be a non-elementary subgroup of $\Aut(X)$ containing a parabolic element. Let $\mu$ be a $\Gamma$-invariant ergodic probability measure on $X$. 
Then, $\mu$ satisfies exactly one  of the following properties. 

\begin{enumerate}[\em (a)]
\item   $\mu$ is the average on a finite orbit of $\Gamma$;
\item  $\mu$ is non-atomic and  supported on a $\Gamma$-invariant algebraic curve $D\subset X$;
\item there is a $\Gamma$-invariant proper algebraic subset $Z$ of $X$, and  a $\Gamma$-invariant, 
totally real  analytic surface $Y$ of 
$X\setminus Z$ such that {\em {(1)}} $\mu(\overline{Y}) = 1$ and $\mu(Z)=0$; {\em {(2)}}  $Y$ has finitely many irreducible components;   
{\em {(3)}} the singular locus of $Y$ is locally finite in  $X\setminus Z$;   
{\em {(4)}} $\mu$ 
is absolutely continuous with respect to the Lebesgue measure  on $Y$;  and {\em {(5)}}  its density (with respect to any real analytic area form on the regular part of  $Y$)
is real analytic;
 \item  there is a $\Gamma$-invariant proper algebraic subset $Z$ of $X$ such that {\em {(1)}} $\mu(Z)=0$, {\em {(2)}} the support of $\mu$ is equal to $X$; {\em {(3)}} $\mu$ is absolutely continuous with respect to the Lebesgue measure on $X$; and {\em {(4)}} the density of $\mu$ with respect to any real analytic
volume form on $X$ is real analytic on $X\setminus Z$.
\end{enumerate}
If $X$ is not a rational surface, then in case~(c) (resp.~(d)) we can further conclude that the invariant measure is proportional to $\vol_Y$ (resp. $\vol_X$).  
\end{thm}

\subsection{Invariant curves}\label{subs:invariant_curves} 
By \cite[Lem. 2.12]{stiffness},  any  action of a non-elementary group $\Gamma$ on a 
projective  surface $X$ admits a maximal invariant curve $\mathrm D_\Gamma$, 
which can be easily detected from the action of $\Gamma$ on $H^2(X, \Z)$ 
since it corresponds to an invariant class. Bounds on the degrees of such invariant curves in terms of the action are given in~\cite[\S 3]{finite_orbits}. If in addition $\Gamma$ 
contains a parabolic element, $\mathrm D_\Gamma$ is the set of common components of the   singular fibers of all elliptic fibrations associated to 
parabolic elements in $\Gamma$ (see~\cite[\S 4.1]{invariant}).

\section{Hyperbolicity of invariant measures}\label{sec:hyperbolicity}

 Here, $X$ is a compact Kähler   surface. We fix a Kähler form $\kappa_0$ on $X$; norms of tangent vectors and differentials will be computed with respect to it.
 
\subsection{Ledrappier's invariance  principle and   invariant measures on $\P T\cX$} \label{subs:ledrappier_invariance_principle} 
In this paragraph we collect some  preliminary results for the proof of Theorems \ref{thm:hyperbolic} and \ref{thm:hyperbolic_kummer}. 
Our presentation is inspired by~\cite{barrientos-malicet}. It is similar in spirit
 to that of \cite{obata-poletti}, which relies on the ``pinching and twisting'' formalism of Avila and Viana 
(see \cite{viana} for an introduction\footnote{Beware that the word ``twisting'' has a different meaning there.}). 
Most of this discussion is valid for a random holomorphic  dynamical system on an arbitrary  complex surface (not necessarily compact), 
satisfying \eqref{eq:moment}.

We denote by $ \P TX$ the projectivized tangent bundle of $X$; if $f$ is an automorphism of $X$, we denote by $\P(Df)$ the induced action on $\P TX$. 

Let $\nu$ be a probability measure on $\Aut(X)$ that satisfies the moment condition~\eqref{eq:moment}. 
We endow $\Omega:=\Aut(X)^\N$ (resp. $\Sigma:= \Aut(X)^\Z$)
with the probability measure $\nu^\N$ (resp. $\nu^\Z$), and set $\cX_+=\Omega \times X$ (resp. $\cX=\Omega \times X$); $\sigma$ will  denote the shift (on $\Omega$ or $\Sigma$).
For $\omega=(f_i)_{i\geq 0}\in \Omega$, we keep the notation  $f_\omega^{n}$ from 
 \S~\ref{subs:notation}.   Then, we define $F_+\colon \cX_+\to \cX_+$ by 
$F_+(\omega,x)=(\sigma(\omega), f^1_\omega(x))$; $F\colon \cX\to \cX$ is defined by the same formula. For further standard notations, we refer to \cite[\S 7]{stiffness}.

Let $\mu$ be an ergodic $\nu$-stationary measure on $X$.  
We introduce the projectivized tangent bundles $\P T \cX_+  = \Omega \times \P TX$ and $\P T\cX = \Sigma\times \P TX$. The  bundles $TX$ and $\P TX$ admit measurable trivializations over a set of full measure. 
Consider any probability measure $\hat\mu$ on  $\P TX$ that is stationary under the random dynamical system induced by $(X,\nu)$ on $\P TX$ 
and whose projection on $X$ coincides with $\mu$, i.e.
$\pi_\varstar \hat\mu= \mu $ where $\pi\colon \P TX\to X$ is the natural projection. Such measures always exist:
indeed 
the set of probability measures on $\P TX$  projecting to $\mu$ is compact and convex, and it is non-empty since it contains
the  measures $\int \delta_{[v(x)]} d\mu(x)$ for any measurable section
$x\mapsto [v(x)]$ of $\P TX$. Thus, the operator $\int \P (Df) \,d\nu(f)$ has a fixed point on that set. The stationarity of $\hat\mu$ is equivalent to the invariance of $\nu^\N\times \hat\mu$
 under  the transformation ${\widehat{F}_+}\colon \Omega\times \P TX\to \Omega\times \P TX$ defined by 
\begin{equation}
{\widehat{F}_+}(\omega, x, [v])= (\sigma(\omega), f^1_\omega(x), \P(D_xf^1_\omega)[v])
\end{equation}
for any non-zero tangent vector $v\in T_xX$. We denote by $\hat\mu_x$ the family of probability measures 
on the fibers $\P T_x X$  of $\pi$ 
given by the disintegration of $\hat\mu$  with respect to $\pi$.
The conditional measures of $\nu^\N\times\hat\mu$ with respect to the projection 
$\P T\cX_+\to X$ are given by $\hat\mu_{\omega,x} = \nu^\N\times \hat\mu_x$.
 
\begin{rem}
Even when  $\mu$ is $\Gamma_\nu$-invariant, this construction only provides a stationary measure on 
$\P TX$.  This is exactly what    happens for non-elementary subgroups with a parabolic automorphism:
 indeed, we will show in \S~\ref{subs:proof_hyperbolic} that projectively invariant measures do not exist in this case. 
\end{rem}

 The tangent action of our random dynamical system  gives rise to a stationary product of 
matrices in $\GL(2, \C)$. To see this, fix a measurable trivialization $P\colon TX\to X\times \C^2$, given 
by linear isomorphisms  $P_x\colon T_xX\to \C^2$. It conjugates the action of $DF_+$ to that of a linear
  cocycle $A:\cX_+ \times \C^2\to \cX_+ \times \C^2$ over $(\cX_+, F_+, \nu^\N\times \mu)$. 
 In this context, 
Ledrappier establishes  in \cite{ledrappier_stationary} the following ``invariance principle''.

\begin{thm}\label{thm:ledrappier_invariance_principle}
If $\lambda^{-}(\mu) = \lambda^{+}(\mu)$, then for any  stationary measure $\hat \mu$ on $\P TX$ projecting   to $\mu$, we have $\P(D_xf)_\varstar \hat\mu_x = \hat\mu_{f(x)}$
for $\mu$-almost every $x$ and $\nu$-almost every $f$.
\end{thm}

The second ingredient in the proof of Theorem~\ref{thm:hyperbolic} is 
a description of such projectively invariant measures; this is where we follow \cite{barrientos-malicet}. 
To explain this result a bit of notation is required. Let $V$ and $W$ be   hermitian vector spaces of dimension $2$;
we fix  two  isometric isomorphisms $\iota_V\colon V\to \C^2$  and $\iota_W\colon W\to \C^2$ to the standard
hermitian space $\C^2$, and we endow the projective lines $\P(V)$ and $\P(W)$ with their respective Fubini-Study metrics. If $g\colon V\to W$ 
is a linear isomorphism, we set  
\begin{equation}
\llbracket g \rrbracket= \norm{ \P (g) }_{C^1}
\end{equation}
where  $\P(g) \colon \P(V)\to \P(W)$
is the projective linear map induced by $g$ and $\norm{\cdot}_{C^1}$
is the maximum of the norms of $D_z\P(g) \colon T_z\P(V)\to T_{\P (g)(z)}\P(W)$ with respect to the Fubini-Study metrics. 
 If   $\iota_W\circ g\circ \iota_V^{-1} = k_1ak_2$ is the KAK decomposition of
$\iota_W\circ g\circ \iota_V^{-1}$ in $\PSL(2, \C)$, we get 
$
 \llbracket g \rrbracket  = \norm{a}^2   =  \norm{\iota_W\circ g\circ \iota_V^{-1}}^2
$
  where 
$\norm{\cdot}$ is the  matrix norm in $\PSL_2(\C)=\SL_2(\C)/\langle \pm \id\rangle$ associated to the Hermitian norm of $\C^2$.
In particular,
\begin{itemize}
\item[(a)] $\llbracket g  \rrbracket=1$ if and only if $\P (g)$ is an isometry from $\P(V)$ to $\P(W)$;
\item[(b)] for a sequence $(g_n)$ of linear maps $V\to W$, 
$\llbracket g_n\rrbracket$ tends to $+\infty$ with $n$   if and only if $\P(\iota_W\circ g\circ \iota_V^{-1})$ 
diverges   to infinity in  $\PSL_2(\C)$.
\end{itemize}
If $f$ is an automorphism of $X$ and $x$ is a point of $X$, then $\kappa_0$ endows $T_xX$ and $T_{f(x)}X$ with 
hermitian structures, and we can apply this discussion to $D_xf\colon T_xX\to T_{f(x)}X$.
We are now ready to state the classification of projectively invariant measures. 

\begin{thm}\label{thm:classification_proj_invariant}
Let  $(X, \nu)$ be a random dynamical system on a complex 
surface and let $\mu$ be an ergodic stationary measure. Let
$\hat\mu$ be a stationary measure on $\P TX$ such that $\pi_\varstar {\hat{\mu}}=\mu$ and $\P(D_xf)_\varstar \hat\mu_x = \hat\mu_{f(x)}$
for $\mu$-almost every $x$ and $\nu$-almost every $f$. Then, exactly one of the following two properties is satisfied: 
 \begin{enumerate}[\em (1)] 
 \item For $(\nu^\N\times \mu)$-almost every
 $(\omega, x)$, the sequence $\llbracket D_xf^n_\omega\rrbracket$  
is unbounded  and then:
 \begin{itemize}
 \item[(1.a)] either there exists a measurable $\Gamma_\nu$-invariant family of lines  
  $E(x)\subset T_xX$ such that   $\hat\mu_x=\delta_{[E(x)]}$ for $\mu$-almost every $x$;
  \item[(1.b)] or there exists a measurable $\Gamma_\nu$-invariant family of pairs of  lines  
  $E_1(x), E_2(x) \subset T_xX$ and positive  numbers $ \lambda_1, \lambda_2$ with 
  $\lambda_1+\lambda_2 =1$  such that    $\hat\mu_x=\lambda_1\delta_{[E_1(x)]} + \lambda_2\delta_{[E_2(x)]}$  for $\mu$-almost every $x$.
 \end{itemize}
\item The projectivized tangent action of $\Gamma_\nu$ is reducible to a compact  group, that is 
there exists a measurable   trivialization  
of the tangent bundle $(P_x: T_xX \to \C^2)_{x\in X}$,    such that for  almost
 every $f\in \Gamma_\nu$ and every 
$x$, $\P\lrpar{P_{f(x)}\circ D_xf\circ P_x\inv}$ belongs to the unitary group $\mathsf{PU}_2(\C)$. 
\end{enumerate}
\end{thm}

In assertion (1.b), the pair is not   ordered:  there is no natural distinction of $E_1$ and $E_2$, 
the elements of $\Gamma_\nu$ may a priori permute these lines. 
The proof is obtained by adapting the 
arguments of  \cite{barrientos-malicet}  to the complex case; the details are given in \S~\ref{subs:barrientos_malicet}.

\subsection{Proof of Theorem \ref{thm:hyperbolic}}\label{subs:proof_hyperbolic}  

 By Theorem~\ref{thm:classification_invariant}, 
$\mu$ is either equivalent  to the Lebesgue measure on $X$, or to the $2$-dimensional Lebesgue measure on some components of an invariant totally real surface $Y\subset X$.  

\subsubsection{Proof of the hyperbolicity of $\mu$} 
Let us assume, by way of contradiction, that $\mu$ is not hyperbolic. Hence its Lyapunov exponents vanish, 
and by Theorem~\ref{thm:ledrappier_invariance_principle} and Theorem \ref{thm:classification_proj_invariant}, there 
 is a measurable set $X'\subset X$ 
 with $\mu(X') = 1$ such that one of the following properties is satisfied along $X'$:   \label{alternative_abc}
 \begin{itemize}
\item[(a)]  there is a measurable $\Gamma_\nu$-invariant line field $E(x)$; 
\item[(b)]  there exists  a measurable $\Gamma_\nu$-invariant splitting $E(x)\oplus E'(x)=T_xX$ of the tangent bundle; 
here, the invariance should be taken in the following weak sense: an element $f$ of $\Gamma_\nu$ maps $E(x)$ to  
$E(f(x))$ or $E'(f(x))$;
\item[(c)] there exists a measurable trivialization $P_x\colon T_xX\to \C^2$  such that in the corresponding coordinates
 the projectivized 
differential $\P(Df_x)$  takes its values in $\mathsf {PU}_2(\C)$ for all $f\in \Gamma_\nu$ and $\mu$-almost all $x\in X'$. 
\end{itemize}

Fix a small $\e>0$. By Lusin's  
 theorem, there is a compact set $K_\e $  with  $\mu(K_\e)>1-\e$ 
 such that the data   $x\mapsto E(x)$, or  $x\mapsto (E(x), E'(x))$ or $x\mapsto P_x$ in the respective 
 cases (a), (b), and (c)  
 are continuous on $K_\e$. 
 In particular, in case (c), the norms of $P_x$ and $P_x^{-1}$ are bounded by some uniform constant $C(\e)$ 
on $K_\e$; hence, if $g\in \Gamma_\nu$ and $x$ and $g(x)$ belong to 
$K_\e$,   $\llbracket Dg_x \rrbracket$ is bounded by $C(\e)^2$. 

Fix a pair of parabolic elements  $g$ and $h\in \Gamma_\nu$ with distinct invariant fibrations $\pi_g\colon X\to B_g$ and $\pi_h\colon X\to B_h$
respectively (see Remark~\ref{rem:many_parabolic}). These two fibrations are tangent along some curve $\Tang(\pi_g, \pi_h)$ in $X$. 

$\bullet$ In a first stage we assume that $X$ is not a torus. According to Section~\ref{subs:halphen}, there is an integer $N>0$ such that  $g^N$ and $h^N$  
 preserve every fiber of their respective invariant fibrations. From now on, we replace $g$ by $g^N$ and $h$ by $h^N$.

First assume that $\mu$ is absolutely continuous with respect to the Lebesgue measure on $X$, with a positive real analytic density
on the complement of some invariant, proper, Zariski closed subset. We apply Lemma~\ref{lem:halphen_flat} to $h$ and remark that 
$(\pi_h)_\varstar \mu$ can not charge the union of the curves $\mathrm R_j$. Then, we disintegrate $\mu$ with respect to $\pi_h$ to obtain
conditional measures $\mu_b$, for $b\in B_h$; since $\pi_h$ is holomorphic, the measures $\mu_b$ are
 absolutely continuous with respect to the Haar measure on almost every fiber $\pi_h^{-1}(b)$. By  Lemma~\ref{lem:halphen_flat},
 there exists a fiber $\pi_h\inv(b)$  such that   (1) the Haar measure of 
$K_\varepsilon \cap \pi_h\inv(b)$ is positive,  
(2) $b \notin \mathrm{NT}_h$ and (3) 
the dynamics of $h$ in $\pi_h\inv(b)$ is uniquely ergodic. These properties hold for $b=\pi_h(z)$, for $\mu$-almost all $z$ in $K_\e$. 
Then we can pick $x\in \pi_h\inv(b)$ such that $(h^{k}(x))_{k\geq 0}$ visits $K_\e$ infinitely many times (\footnote{Note that we use the invariance of $\mu$ here, not   mere $\nu$-stationarity.}). The fifth assertion of Lemma~\ref{lem:halphen_flat}
rules out case (c) because the twisting property implies that the projectivized derivative $\llbracket Dh^n_x \rrbracket$ tends to infinity, while it should be bounded 
by $C(\e)^2$ when $h^n(x)\in K_\e$.
Case (b)  is also excluded: under the action of $h^n$, tangent vectors  projectively 
converge to the tangent space of the fibers, so the only possible invariant subspace of dimension $1$  is $\ker(D\pi_h)$.
Thus we are in case (a) and moreover  $E(x) = \ker D_x\pi_h$  for $\mu$-almost every $x$. But then, 
using $g$ instead of $h$ and the fact that $\mu$ does not charge the  curve  $\Tang(\pi_g, \pi_h)$, we get a contradiction. 
This shows that the last alternative (a) does not hold either, and this contradiction proves that $\mu$ is hyperbolic.

If $\mu$ is supported by a $2$-dimensional real analytic subset $Y\subset X$, the same proof applies, 
except that we disintegrate
$\mu$ along the  singular  foliation of $Y$  by circles induced by $\pi_h$ and we use the fact that a generic leaf is a 
circle along which
  $h$ is   uniquely ergodic (see Lemma~\ref{lem:halphen_flat}.{(4)}). 

$\bullet$ If $X$ is a torus its tangent bundle is trivial and the differential of an automorphism is 
constant. In an appropriate basis, the differential of a Halphen twist $h$  is of the form 
\begin{equation}
\begin{pmatrix} 1 & \alpha \\ 0 & 1 \end{pmatrix} \text{ with } \alpha\neq 0.
\end{equation}
 Thus we are in case (a) with $E(x) = \ker D_x\pi_h$ for $\mu$-almost every $x$. Using another 
twist $g$ transverse to $h$ we get a contradiction as before. 

\subsubsection{Proof of the positivity of the fiber entropy} \label{subs:proof_positivity_entropy}
This follows from classical arguments. Since $\mu$ is invariant the  measure $\m=\nu^\Z\times \mu$ on $\cX$ is 
$F$-invariant. In both cases $\mu\ll \vol_X$ and  $\mu\ll \vol_Y$, respectively. 
  the absolute continuity of the foliation by  local Pesin  unstable manifolds implies that 
the unstable conditionals of $\m$ cannot be atomic, see e.g.~\cite{ledrappier-young3}. 
Since the unstable conditionals of 
a zero entropy stationary measure are automatically atomic (see \cite[Cor.~7.14]{stiffness}),
we conclude that 
$\mu$ has positive fiber entropy. 

This concludes the proof of Theorem~\ref{thm:hyperbolic}\qed

 \subsection{A variant of Theorem~\ref{thm:hyperbolic}} 
Let us first recall the definition of classical Kummer examples 
(see~\cite[\S 4]{finite_orbits} for a thorough treatment)).
Let $A = \C^2/\Lambda$ be a complex torus and let $\eta$ be the involution given by $\eta(z_1, z_2) = (-z_1, -z_2)$; it has $16$ 
fixed points. Then $A/\langle\eta\rangle$ is a  
 surface with $16$ singular points, and resolving  these singularities (each of them requires a single blow-up)
yields a  Kummer  surface $X$. 
Let $f_A$ be a loxodromic  automorphism of $A$ which is induced by a linear transformation of $\C^2$ preserving $\Lambda$;
then $f_A$ commutes to $\eta$ and goes down to an automorphism $f$ of~$X$; such automorphisms will be referred to as 
loxodromic, classical, Kummer examples. They preserve the canonical volume $\vol_X$. The Kummer surface $X$
 also supports automorphisms which do not come  from automorphisms of $A$ (see~\cite{Keum-Kondo} and~\cite{Dolgachev-Keum} for instance). 

In the following   statement  we do not assume that $\Gamma_\nu$ contains a parabolic element.

\begin{thm} \label{thm:hyperbolic_kummer} 
Let $(X, \nu)$ be a non-elementary  random dynamical system on a Kummer  K3 surface satisfying 
\eqref{eq:moment} and such that  $\Gamma_\nu$ contains a loxodromic classical Kummer example. Then   
any ergodic $\Gamma_\nu$-invariant 
measure giving no mass to proper Zariski closed subsets of $X$ is hyperbolic. 
\end{thm}


\begin{proof}
  The proof is similar to that of Theorem \ref{thm:hyperbolic} so we only sketch it. Assume by contradiction that $\mu$ is 
not hyperbolic;  since $X$ is a K3 surface, the volume invariance
shows that the sum of the Lyapunov exponents of $\mu$ vanishes (see \cite[\S 7.3]{stiffness}); thus, each of them is equal to $0$, and one of the alternatives of Theorem \ref{thm:classification_proj_invariant} holds, referred to as (a), (b), (c) as in Section~\ref{subs:proof_hyperbolic}, page \pageref{alternative_abc}.  

By assumption, $\Gamma_\nu$ contains a  loxodromic, classical Kummer example $f$ associated to a linear automorphism $f_A$ of a torus $A$. This automorphism $f$ is uniformly
hyperbolic in  some dense Zariski open subset $U$, which is thus of full $\mu$-measure: its complement is given by the sixteen rational curves
coming from the resolution of the singularities of $A/\eta$.  
 We denote by $x\mapsto E^u_f(x)\oplus E^s_f(x)$ 
the associated splitting of $TX\rest{U}$. The line field $E^u_f$ (resp. $E^s_f$) is everywhere tangent 
to an $f$-invariant (singular) holomorphic foliation $\mathcal F^u$ (resp. $\mathcal F^s$) coming from the $f_A$ invariant linear unstable (resp. stable) foliation on~$A$.
Since $f$ is uniformly expanding/contracting on $E^{u/s}_f$, Alternative (c) is not possible. 

If Alternative (a) holds, then $E(x)$ being $f$-invariant on a set of full 
measure, it must coincide with $E^u_f$ or $E^s_f$, say with $E^u_f$.  
By continuity any $g\in \Gamma_\nu$ preserves 
$E^u_f$ pointwise on $\supp(\mu)$.  
Since in addition $\mu$ is Zariski diffuse,  $g$ 
preserves $E^u_f$ everywhere on $X$, so it 
preserves also the unstable holomorphic foliation $\mathcal F^u$. From this, we shall contradict 
the fact that $\Gamma_\nu$ is non-elementary. We use a dynamical argument, based on basic constructions which are surveyed in~\cite{Cantat:Milnor}; one can also derive a contradiction from~\cite{Cantat-Favre}.

Every leaf of $\mathcal F^u$, except a finite number of them, is parametrized by an injective entire holomorphic curve $\varphi\colon \C\to X$, the image of which is Zariski dense. 
Fix a Kähler form $\kappa$ on $X$ and consider the positive currents defined by
\begin{equation}
\alpha\mapsto  \left(\int_0^R\int_{\disk(0;t)} \varphi^*\kappa \frac{dt}{t}\right)^{-1} \int_0^R\int_{\disk(0;t)} \varphi^*\alpha \frac{dt}{t}
\end{equation}
for any smooth $(1,1)$-form $\alpha$. As $R$ goes to $+\infty$, it is known that this sequence of currents converges to a closed positive current $T^+_f$
that does not depend on the parametrization $\varphi$ of the leaf, nor on the leaf itself (provided the leaf is Zariski dense). 
This current is uniquely determined by $\mathcal F^u$ and the normalization $\langle T^+_f\vert \kappa \rangle= 1$. Dynamically, it is
the unique closed positive current $T^+_f$ that satisfies $\langle T^+_f\vert \kappa \rangle= 1$ and $f^*T^+_f=\lambda(f)T^+_f$ for some $\lambda(f)>1$.
Its cohomology class $[T^+_f]$ is a non-zero element of $H^{1,1}(X;\R)$ of self-intersection $0$. 

Now, pick any element $g$ of $\Gamma_\nu$. Since $g$ preserves $\mathcal F^u$, it permutes its leaces and preserves the ray $\R_+[T^+_f]$.
Thus, $\Gamma_\nu$ preserves an isotropic line for the intersection form in $H^{1,1}(X;\R)$, and this contradicts  the non-elementarity assumption (see~\cite[\S 2.3]{stiffness}). 

Finally, if alternative (b) holds, any $g\in \Gamma_\nu$ preserves 
$ \{E^u_f(x), E^s_f(x)\}$ on a set of full measure so,  
since $\mu$ is Zariski diffuse,
it must either preserve or 
swap these directions. Passing to an index $2$ subgroup  both directions are preserved, and we again contradict the non-elementary assumption, as in case (a). 
\end{proof} 

\subsection{Proof of Theorem \ref{thm:classification_proj_invariant}} \label{subs:barrientos_malicet}

Let us consider a random dynamical system $(X, \nu)$ and an ergodic stationary measure $\mu$, 
as in Theorem \ref{thm:classification_proj_invariant}. We keep the notation from \S \ref{subs:ledrappier_invariance_principle}. In particular, we fix a Kähler form $\kappa_0$ on $X$ and compute norms with respect to it.

We say that a sequence of real numbers $(u_n)_{n\geq 0}$  {\bf{almost converges towards $+\infty$}} if for every 
$K\in \R$, 
 the set $L_K=\set{ n\in \N\; ; \; u_n\leq K }$ has an asymptotic lower density
\begin{equation}
\underline\dens(L_K):=\liminf_{n\to +\infty}\left( \frac{\sharp (L_K\cap [0,n])}{n+1}\right)
\end{equation}
which is  equal to $0$: $\underline \dens(L_K)=0$ for all $K$.

\begin{lem}\label{lem:alternative_bdd}
The set   of points $\cx=(\omega,x)$ in $\X_+$ such that $\llbracket D_xf^n_\omega\rrbracket$ 
almost converges towards $+\infty$ on $\P (T_xM)$ is $F_+$-invariant. In particular, by ergodicity, 
\begin{enumerate}[\em (a)]
\item either $\llbracket D_xf^n_\omega\rrbracket$  almost converges towards $+\infty$ for $(\nu^\N\times \mu)$-almost every $(\omega, x)$;
\item or, for $(\nu^\N\times \mu)$-almost every $(\omega, x)$, there is a sequence $(n_i)$ with positive lower density along which 
$\llbracket  D_xf^{n_i}_\omega \rrbracket$ is  bounded.
\end{enumerate}
\end{lem}

The proof is straightforward and left to the reader (see~\cite{barrientos-malicet}). 
We are now ready for the proof of Theorem \ref{thm:classification_proj_invariant}. 
Let us first  emphasize one delicate issue: in Conclusion~(1) of the theorem,  
it is important that the directions $E$ (resp. $E_1$ and $E_2$) only depend on $x\in X$ (and not on 
$\cx = (x,\omega)\in \cX_+$). Likewise in Conclusion~(2), the trivialization $P_x$ should depend  only on $x$. 
This justifies the inclusion of a detailed proof of Theorem \ref{thm:classification_proj_invariant}, since in
 the slightly different setting of       
\cite{barrientos-malicet}, the authors did not have to check this property.

We fix a measurable trivialization $P\colon TX\to X\times \C^2$, given 
by linear isometries  $P_x\colon T_xX\to \C^2$, where $T_xX$ is endowed with the hermitian form $(\kappa_0)_x$,  and 
$\C^2$ with its standard hermitian form. This trivialization conjugates the action of $DF_+$ to that of a 
cocycle $A\colon \cX_+ \times \C^2\to \cX_+ \times \C^2$ over $F_+$.  
We denote by $A_\cx\colon \set{\cx}\times \C^2\to \set{F_+(\cx)}\times \C^2$ the induced linear map; since $A_\cx=P_{f^1_\omega(x)}\circ (Df^1_\omega)_x \circ P_x^{-1} $, we see that 
$A_\cx = A_{(\omega, x)}$ depends only on $x$ and on the first coordinate $f_\omega^1 = f_0$ of $\omega$. 
Using  $P$ we transport the measure $\hat\mu$ to a measure,  
still denoted by $\hat\mu$, on the product space $X\times \P^1(\C)$. By our invariance assumption, its
disintegrations $\hat\mu_\cx=\hat\mu_x$ satisfy $(\P A_\cx)_\varstar \hat\mu_\cx=\hat\mu_{F_+(\cx)}=\hat\mu_{f^1_\omega(x)}$.

\smallskip

\textbf{The essentially bounded case.}-- In this paragraph we show
  that in case (b) of Lemma \ref{lem:alternative_bdd}, 
Conclusion~(2) of Theorem \ref{thm:classification_proj_invariant} holds.  We streamline the argument of    
\cite[Prop. 4.7]{barrientos-malicet} which deals with  the more general case of $\GL(d, \R)$-cocycles 
 (see also~\cite{Arnold-Nguyen-Oseledets, Zimmer:Israel1980}). 

Set $G =\PGL(2, \C)$, and  define the $G$-extension   ${\widetilde{F}}_+$ of
$F_+$  on $\X_+\times G$ by 
\begin{equation}
{\widetilde{F}}_+(\cx, g)=(F_+(\cx), \P(A_\cx)  g)=((\sigma(\omega), f^1_\omega(x)), \P(A_{(\omega, x)})  g)
\end{equation}
for every $\cx=(\omega,x)$ in $\X_+$ and $g$ in $G$; thus ${\widetilde{F}}_+$  is given by $F_+$ on $\X_+$ 
and is the multiplication by  $\P(A_\cx)$ on $G$.
Since $\P(A_{(\omega, x)})$ depends on $\omega $ only through its first coordinate, 
${\widetilde{F}}_+$ is the skew product 
map associated to the random dynamical system $(f(x), P_{f^1_\omega(x)}\circ (Df^1_\omega)_x \circ P_x^{-1} )$ on $X\times G$. 
Denote by $\mathcal P$ the convolution operator associated to this random dynamical system. 
Let $\mathrm{Prob}_\mu(X\times G)$ be
the set of probability measures 
on $X\times G$ projecting to $\mu$ under the natural map $X\times G \to X$. Since $\mu$ is stationary, 
$\mathcal P$ maps $\mathrm{Prob}_\mu(X\times G)$ to itself.

By assumption 
there is a set $E$ of positive measure in $\X_+$, 
a compact subset $K_G$ of $G$,  and a positive real number $\varepsilon_0$ such that 
\begin{equation}
\underline{\dens}\set{ n\; ; \; \P(A^{(n)}_\cx)\in K_G }\geq \varepsilon_0
\end{equation}
for all $\cx$ in $E$.  

\begin{lem} There exists  an ergodic, stationary, Borel probability measure $\widetilde{\mu}_G$ on $X\times G$ 
with marginal measure $\mu$ on $X$.
\end{lem}

\begin{proof} (See \cite[Prop. 4.13]{barrientos-malicet} for details).
Let $\widetilde{\mu}_G$ be any cluster value of the sequence of probability measures  
$ \unsur{N}\sum_{i=0}^{N-1} \mathcal P^i (\mu\times \delta_{1_G})$. By the boundedness assumption, $\widetilde{\mu}_G$ has mass $M\geq \e_0$ 
(the possible escape of mass to $\infty$ in $G$ is not total)
and is stationary (i.e. $\mathcal P$-invariant). Its projection on the 
first factor is equal to $M\mu$.  We renormalize it 
to get a probability measure. Then, using  the ergodic decomposition and  the ergodicity of $\mu$,  we may replace 
it by an ergodic stationary measure in $\mathrm{Prob}_\mu(X\times G)$.
\end{proof}

Denote by $\widetilde \m_G = \nu^\N\times \widetilde \mu_G$ the  ${\widetilde{F}}_+$-invariant measure associated to 
$\widetilde{\mu}_G$. 
The action of ${\widetilde{F}}_+$ on     $\cX_+\times G$ (resp. of the induced random dynamical system on 
$X\times G$) commutes to the action of $G$ 
by right multiplication, i.e. to the diffeomorphisms $R_h$ defined by
\begin{equation}
R_h(\cx, g)= (\cx, g  h)
\end{equation}
for $h\in G$.
Slightly abusing notation we also denote by $R_h$ the analogous map on $X\times G$. 
The next lemma combines classical arguments due to Furstenberg and Zimmer. 

\begin{lem}
Let $\widetilde\mu_G$ be a Borel stationary measure on $X\times G$ with marginal $\mu$ on $X$.  
Set 
\[
H =\set{ h \in G \; ; \; (R_h)_\varstar {\widetilde{\mu}}_G  = \widetilde \mu_G}
= \set{ h \in G \; ; \; (R_h)_\varstar {\widetilde{\m}}_G  =\widetilde \m_G }.
\]
Then $H$ is a compact subgroup of $G$ and there is a measurable function $Q\colon X \to G$ 
such that the cocycle $B_\cx = Q_{f^1_\omega(x)}^{-1}\cdot \P(A_\cx)\cdot Q_x$ takes its values in $H$
for $(\nu^\N\times \mu)$-almost every $\cx$. 
\end{lem}

\begin{proof}  
Clearly, $H$ is a closed subgroup of $G$. 
If $H$ were not
bounded then, given any compact subset $C$ of $G$, we could find a sequence $(h_n)$ of elements of $H$ 
such that the subsets $R_{h_n}(C)$ are pairwise disjoint. Choosing  $C$ 
such that $X\times C$ has positive $\widetilde\mu_G$-measure, 
we would get a contradiction with  the finiteness of $\widetilde\mu_G$. So $H$ is a compact subgroup of $G$. 

We say that a point $(x,g)$ in $X\times G$ is generic  
if for $\nu^N$-almost every $\omega$, 
\begin{equation}\label{eq:generic_G}
\frac{1}{N}\sum_{n=0}^{N-1}  \varphi\lrpar{{\widetilde{F}}^n_+(\omega, x, g) } \underset{N\to\infty}\longrightarrow 
\int_{\cX_+\times G} \varphi \; d  \widetilde\m _G 
\end{equation}
for every compactly supported continuous function $\varphi$ on $\cX_+\times G$.
The Birkhoff ergodic theorem 
provides a Borel set $\mathcal E$ of generic points of full $\widetilde \mu_G$-measure. 
Now  if   $(x, g_1)$ and $(x, g_2)$ belong to $\mathcal E$,  
writing $g_2=g_1  h=R_h(g_1)$ for $h=g_1^{-1}  g_2$, we get that   $h$ is an 
element of $H$.  

Given $g\in G$, define ${\mathcal{E}}_x\subset G$ to be the set of elements $g\in G$ such that $(x,g)$ is generic. Then there 
exists a   measurable section
 $X\ni x\mapsto Q_x  \in  G$ such that $Q_x\in {\mathcal{E}}_x$ for almost all~$x$. By definition of $\mathcal{E}_x$, 
$(\omega, x, Q_x)$ satisfies \eqref{eq:generic_G} for $\nu^N$-almost every~$\omega$.  
The $\widetilde F_+$-invariance of the set of Birkhoff generic points implies that 
$(f_\omega^1(x),  \P(A_\cx) Q_x)$ belongs to $\mathcal E$ for  $\nu$-almost every $f_0 = f_\omega^1$. 
Since $(f_\omega^1(x), Q_{f_\omega^1(x)})$ belongs to $\mathcal E$ as well, it follows that $Q_{f_\omega^1(x)}^{-1}  \P(A_\cx)  Q_x$ is in $H$. 
We conclude that the cocycle 
$B_\cx = Q_{f^1_\omega(x)}^{-1}\cdot \P(A_\cx)\cdot Q_x$ takes its values in $H$ for almost all $\cx$, as claimed. 
\end{proof}

Note that the map $x\mapsto Q_x$ lifts to a measurable map $x\mapsto Q'_x\in \GL_2(\C)$.  
Conjugating $H$ to a subgroup of $\mathsf{PU}_2$ by some element $g_0\in G$, the two previous lemmas give: {\textit{ if
$\llbracket  D_xf^{n}_\omega \rrbracket$ is essentially bounded, then Conclusion~(2) of Theorem~\ref{thm:classification_proj_invariant} holds}}
(the $P_x$ are obtained by composing the $Q'_x$ with a lift of $g_0$ to $\GL_2(\C)$).

\smallskip

\textbf{The unbounded case. --}  Now, we suppose that $\llbracket D_xf^n_\omega\rrbracket$ is essentially unbounded (alternative (a) of Lemma \ref{lem:alternative_bdd}) and we adapt the results of \cite[\S 4.1]{barrientos-malicet} to the complex setting to arrive at one of the Conclusions~(1.a) or (1.b)
of Theorem~\ref{thm:classification_proj_invariant}.
 The main step of the proof is the following lemma.

\begin{lem}\label{lem:1/2}
Let $A$ be a measurable $\GL(2, \C)$-cocycle over $(\cX_+, F_+, \nu^\N\times \mu)$
 admitting a projectively  invariant family of probability 
measures  $\lrpar{\hat \mu_x}_{x\in X}$ such that almost surely $\llbracket A_{\cx}^{(n)}\rrbracket$ almost converges to 
infinity. Then for almost every $x$, $\hat\mu_x$ possesses an atom of mass at least $1/2$; more precisely: 
\begin{itemize}
\item either $\hat\mu_x$ has a unique atom $[w(x)]$ of mass $\geq 1/2$, that depends measurably on $x\in X$;
\item or $\hat\mu_x$ has a unique pair of atoms   
 of mass $1/2$, and this (unordered) pair depends measurably on $x\in X$. 
\end{itemize}
\end{lem}

For the moment, we take  this result for granted and proceed with the proof. 
By ergodicity of $\mu$, the number of atoms of $\hat\mu_x$  and the list of their masses  are constant on a set
of full measure.  A first possibility is that $\hat\mu_x$ is almost surely the single point mass $\delta_{[w(x)]}$; this corresponds to (1.a).
A second possibility is that $\hat\mu_x$ is the sum of two point masses of mass $1/2$; this corresponds to~(1.b). 
In the 
remaining cases, 
 there is exactly one atom of mass $1/2\leq \alpha <1$ at a point $[w(x)]$. 
Changing the trivialization $P_x$, we can suppose that $[w(x)]=[w]=[1:0]$. Then we
write $\hat\mu_x=\alpha \delta_{[1:0]}+\hat\mu_x'$, and apply Lemma \ref{lem:1/2} to the family of measures  
$\hat\mu_\cx'$ (after normalization to get a probability measure). 
 We deduce that almost surely $\hat\mu_x'$ admits  an atom  of mass $\geq (1-\alpha)/2$. 
Two cases may occur:  
\begin{itemize}
\item $\hat\mu_x'$ has a unique atom of mass $\beta \geq (1-\alpha)/2$,
\item $\hat\mu_x'$ has two atoms of mass $(1-\alpha)/2$.
\end{itemize}
The second one is  impossible, because changing the trivialization, we would have
$\hat\mu_x=\alpha \delta_{[1:0]}+\frac{1-\alpha}{2} (\delta_{[-1:1]}+\delta_{[1:1]})$, and the invariance 
of the finite set $\set{ [1:0], [-1:1], [1:1] }$ would imply that the cocycle $\P (A_\cx)$ stays in a finite subgroup of $\PGL_2(\C)$, contradicting 
the unboundedness assumption. 

If  $\hat\mu_\cx'$ has a unique atom of mass $\beta \geq (1-\alpha)/2$, we change $P_x$ to put it at $[0:1]$
(the trivialization $P_x$ is not an isometry anymore).
We repeat the argument
with  $\hat\mu_x=\alpha \delta_{[1:0]}+\beta \delta_{[0:1]}+\hat\mu_x''$. 
If $\beta = 1-\alpha$, i.e. $\hat\mu_x''=0$, then we are done. Otherwise 
  $\hat\mu_x''$ has one or two atoms of mass $\gamma \geq (1-\alpha-\beta)/2$, and we change
$P_x$ to assume that one of them is $[1:1]$ and the second  one --provided it exists-- 
is $[\tau(x):1]$; here, $x\mapsto \tau(x)$ is a
complex valued measurable function. 
Endow the projective line $\P^1(\C)$ with the coordinate $[z:1]$; then   $\P(A_\cx)$ is of the 
form $z\mapsto a(\cx) z$. Since $\P (A_\cx) \lrpar{\set{1, \tau(x)}} = \lrpar{\set{1, \tau(F_+(\cx))}}$, we infer that:
\begin{itemize}
\item either $a(\cx)1 = 1$ and $\P(A_\cx)$ is the identity;
\item  or $a(\cx)1 = \tau(\pi_X(F_+(\cx))) $ and  $a(\cx)\tau(x) =1$ in which case  $\tau(\pi_X(F_+(\cx))) = \tau(x)\inv$. 
\end{itemize}
Thus we see that along the  orbit of $\cx$, $a(F^n_+(\cx))$ takes at most two values $\tau(\pi_X(F^n_+(\cx)))^{\pm 1}$, and $\llbracket A^{(n)}_\cx\rrbracket$ is bounded, which is contradictory. This concludes the proof. \qed

\begin{proof}[Proof of Lemma \ref{lem:1/2}]
Let $r$ and $\e$ be small positive real numbers.
Let $\mathrm{Prob}_{r, \e}(\P^1(\C))$ be the set of probability measures $m$
on $\P^{1}(\C)$ such that $\sup_{x\in \pu} m(B(x, r))\leq 1/2-\e$, 
where the ball is with respect to some fixed Fubini-Study metric. 
This is a compact subset of the space of probability measures on $\pu$. 
The set 
\begin{equation}
G_{r, \e} = \set{\gamma\in \PGL(2, \C),\ \exists m_1, m_2 \in \mathrm{Prob}_{r, \e}(\P^1(\C)), \ \gamma_\varstar m_1 =m_2}
\end{equation}
is a bounded subset of $\PGL(2, \C)$. Indeed otherwise there would be an unbounded sequence $\gamma_n$ together 
with sequences $(m_{1, n})$ and $(m_{2, n})$ in $\mathrm{Prob}_{r, \e}(\P^1(\C))$ such that 
$(\gamma_n)_\varstar m_{1, n} =  m_{2, n}$. Denote by $\gamma_n=k_n a_n k'_n$  the KAK decomposition of $\gamma_n$ in $\PGL(2, \C)$,
with $k_n$ and $k'_n$ two isometries for the Fubini-Study metric; since $\gamma_n$ is unbounded, we can extract a
subsequence such that the measures $(k'_n)_\varstar m_{1, n}$ and $(k_n^{-1})_\varstar m_{2, n}$ converge in 
$\mathrm{Prob}_{r, \e}(\P^1(\C))$ to two measures $m_1$ and $m_2$,
while the diagonal transformations $a_n$ converge locally uniformly 
on $\P^1(\C)\setminus\set{[0:1]}$ to the  constant map $\gamma: \P^1(\C) \setminus\set{[0:1]} \mapsto \set{[1:0]}$. Then
\begin{equation}
\gamma_\varstar \lrpar{m_{1\vert \P^1(\C)\setminus\set{[0:1]}}} =  m_1(\P^1(\C)\setminus\set{[0:1]}) \delta_a \leq m_2;
\end{equation}
since $m_1$  belongs to  $\mathrm{Prob}_{r, \e}(\P^1(\C))$, $m_1(\P^1(\C)\setminus\set{[0:1]}) \geq 1/2+\e$, hence 
$m_2\geq (1/2+\e) \delta_a$, in contradiction with $m_2\in \mathrm{Prob}_{r, \e}(\P^1(\C))$. This proves that $G_{r,\e}$ is bounded.

To prove the lemma, let us consider the ergodic dynamical system $\P DF_+$, and the family of conditional 
probability measures $\hat\mu_\cx$ for the projection $(\omega, x,v)\mapsto \cx=(\omega,x)$.
If there exist  $r, \e>0$ such that $\hat \mu_\cx $ belongs to $\mathrm{Prob}_{r, \e}(\P^1(\C))$ 
for $\cx$ in some positive measure subset $B$ then, by ergodicity, for almost every $\cx\in \X_+$ there exists a set of integers $L(\cx)$ of 
 positive density such that for $n\in L(\cx)$,  $F_+^n(\cx)$ belongs to $B$, hence $A^{(n)}_{\cx} $ belongs to $G_{r, \e}$ (\footnote{We are slightly abusing here when the Fubini-Study metric depends on $x$, for instance when $P_x$ is not an isometry; however restricting to subsets of large positive measure the metric $(P_x)_\varstar (\kappa_0)_x$ is uniformly comparable to a fixed Fubini-Study metric.}). 
From the above claim we deduce that   $\llbracket A^{(n)}_\cx \rrbracket$ is uniformly bounded for $n\in L(\cx)$, a contradiction. 
Therefore for every $r, \e>0$, the measure of $\set{\cx, \ \hat \mu_\cx \in \mathrm{Prob}_{r, \e}(\P^1(\C))}$ is equal to $0$; it  follows that 
for almost every $\cx$, $\hat\mu_\cx$ possesses an atom of mass at least $1/2$.

If there is a unique atom of mass $\geq 1/2$, this atom determines a measurable map $\cx\mapsto [w(\cx)]\in \P T_xX$; since $\hat \mu_\cx$ does not depend on $\omega$, $[w(\cx)]$ depends only on $x$, not on $\omega$. If there are generically two atoms of mass $\geq 1/2$, then 
both of them has mass $1/2$, and the pair of points determined by these atoms depends only on $x$. 
 \end{proof}

\section{Characterization of uniform expansion}\label{sec:characterization}

In this section we build on the previous results, in conjunction with the measure rigidity
 results from our previous work~\cite{stiffness},
to  find sufficient conditions for as well as obstructions to
 uniform expansion for a non-elementary action on a compact complex surface. 
 
\subsection{Proof of Theorem~\ref{thm:criterion_uniform_expansion_parabolic} and related results}
 
%
\subsubsection{Applying Chung's criterion}

 \begin{defi}\label{defi:non_expanding_measure}
Let $\nu$ be a probability measure on $\Aut(X)$. A 
 $\nu$-stationary measure $\mu$ on $X$ is said \textbf{non-expanding}   if 
 every ergodic component $\mu'$ of $\mu$ satisfies:
\begin{enumerate}[(i)]
\item either both Lyapunov exponents of $\mu'$  are non-positive,
\item or $\mu'$ is hyperbolic and its field of  Oseledets stable directions is  non-random.
\end{enumerate}
\end{defi}

Theorem~\ref{thm:criterion_chung} asserts that the existence 
of  non-expanding $\nu$-stationary measures is the obstruction to uniform expansion of $\nu$:

\begin{cor}[of Theorem~\ref{thm:criterion_chung}] \label{cor:criterion_chung}
Let $X$ be a compact complex surface and  $\nu$ be a probability measure on the group $\Aut(X)$, satisfying the moment condition \eqref{eq:moment}. 
Then $\nu$ is uniformly expanding if 
and only if non-expanding $\nu$-stationary measures do not exist, hence 
 if and only if every ergodic $\nu$-stationary  measure $\mu$ on $X$ satisfies one of the following properties:
 \begin{itemize}
 \item $\mu$ has a positive Lyapunov exponent and its stable distribution  depends non-trivially on the itinerary;
 \item the two Lyapunov exponents of $\mu$ are strictly positive. 
 \end{itemize}  \end{cor}

\subsubsection{Groups with invariant curves}

\begin{pro}\label{pro:invariant_curve}
Let $X$ be a compact complex surface. Let
$\Gamma$  be a subgroup of $\Aut(X)$ that preserves a complex curve $C\subset X$. If $\nu$ is a probability measure on $\Gamma$ satisfying~\eqref{eq:moment}, then $\nu$ is not uniformly expanding. 
\end{pro}

\begin{rem}
We leave the reader check that the proof adapts to the real case in the following sense: if $X$, $\Gamma$ and $C$ are defined over $\R$ and $C(\R)$ is of dimension 1 (that is, neither empty nor a finite set), then $\nu$ is not uniformly expanding in restriction to $C(\R)$. 
\end{rem}

\begin{lem}  \label{lem:criterion_chung_curves}
Let $C$ be  a compact Riemann surface. Then, $\Aut(C)$ does not support any uniformly expanding probability measure. 
\end{lem}

\begin{proof}
Let $\kappa$ be a Kähler form on $C$  that satisfies $\int_C\kappa = 1$. For every $f\in \Aut(C)$, 
$\int_C f^*\kappa = 1= \int_C \norm{D_xf}^2 \kappa= 1$, so by the Jensen inequality 
  $\int_C \log \norm{D_xf }  \kappa\leq 0$. Now,  if $\nu$ is any probability measure on $\Aut(C)$, then 
  $$\int_C \int_{\Aut(C)} \log \norm{D_xf } d\nu(f) \kappa\leq 0,$$
   hence Property~\eqref{eq:UEbisrepetita} cannot be satisfied
  by $\nu$ (for any $n_0\geq 1$). 
\end{proof}

Note that the same argument applies to conformal diffeomorphisms, in particular for
$\mathsf{Diff}^1(\mathbb S^1)$. Lemma~\ref{lem:criterion_chung_curves} and Remark~\ref{rem:restriction_UE} 
imply Proposition~\ref{pro:invariant_curve} when $C$ is smooth; we now prove Proposition~\ref{pro:invariant_curve} in full generality. 

\begin{proof}[Proof of Proposition~\ref{pro:invariant_curve}]
Arguing by contradiction, we assume that $\nu$ is uniformly expanding.  Let $\Gamma_1\leq \Gamma$ 
be the finite index subgroup fixing each  component of $C$,
and each of its branches at each of its singular points; let $\nu_1$ be 
the hitting measure on $\Gamma_1$ associated to 
$\nu^{(n_0)}$, where $n_0$ is as in Equation~\eqref{eq:UEbisrepetita}. By Proposition~\ref{pro:induced_UE}, $\nu_1$ is uniformly 
expanding, so by replacing $\nu$ by $\nu_1$ and $C$ by one of its components
 we  assume now  that $C$ is irreducible and all branches at its
  singular points are fixed by $\Gamma$. To get a 
 contradiction we will construct a stationary measure $\mu$ supported on $C$ 
such that the tangential Lyapunov exponent along $TC$ is non-positive. 

By Lemma~\ref{lem:criterion_chung_curves} we may assume that 
 the singular set $\Sing(C)$ is non-empty.  
If the genus of $C$ is $\geq 0$,  the invariance of $\Sing(C)$ forces $\Gamma\rest{C}$ to be finite, in contradiction with the 
uniform expansion of $\nu$. Thus, $C$ is a rational curve; 
let $\pi:\hat C\to C$ be its normalization and $\hat \Gamma\subset \Aut(\hat{C})\simeq \PGL_2(\C)$ be the group  induced by $\Gamma$; the measure $\nu$ induces a  
measure $\hat \nu$ on $\hat \Gamma$.
Fix  $\hat p\in \hat C$  such that $p:=\pi(\hat p)$ is singular. The germ of curve given by $\hat C$ at $\hat p$ determines one of the branches of $C$ at $p$; our assumptions imply that 
 $\hat p$ is fixed by $\hat\Gamma$. There are local coordinates  $t\in (\C,0)$ for $(\hat C, \hat p)$ and $(z,w)\in (\C^2, (0,0))$ for $(X, p)$ in which
$\pi$ is expressed as a Puiseux expansion
\begin{equation}
t\mapsto (\pi_1(t), \pi_2(t)) =  (\alpha t^q , \beta t^r) \ \ {\text{modulo higher order terms}}
\end{equation} 
where $1\leq q<r$; if $q=1$ the branch is smooth at $p$. In these coordinates, the tangent direction to $C$ at $p$ corresponding to the branch determined by $\hat p$ is  
 given by $(1,0)\in \C^2$. 
 Let $\lambda_{(\hat C, \hat{p})}$ be the 
Lyapunov exponent of $\hat \nu$ at $\hat p$, and $\lambda_{(C, p)}$ be the Lyapunov exponent of $\nu$
in the tangent direction of this branch. 


\begin{lem}\label{lem:tangent_singular}
With notation as above  
$\lambda_{(C, p)} = q  \lambda_{(\hat C, \hat{p})}$. In particular $\lambda_{(C, p)}$ and $\lambda_{(\hat C, \hat{p})}$ 
have the same sign.
  \end{lem}

\begin{proof}
 Pick $f\in \Gamma$, write 
$f(z,w)  =  (f_1(z,w), f_2(z,w))$ in the local coordinates $(z,w)$, and expand $f_1$ in power series: $f_1(z,w) = \sum_{i,j} a_{i,j}z^iw^j$. 
Since the branch determined by $\lambda_{(\hat C, \hat{p})}$ is $f$-invariant, we have $D_{p}f(1,0)   = (a_{1,0}, 0)$ with $a_{1,0}\neq 0$. Thus, 
\begin{equation} f_1(\pi(t))  
= \sum_{i,j=0}^\infty a_{i,j}\alpha^i\beta^j t^{qi+rj}  
 = a_{1, 0} \alpha t^q \mod(t^{q+1}).
\end{equation} 
Now,  $f$ lifts to an automorphism $\hat f$ of $\hat C$ fixing $\hat p$. Writing
 $\hat f(t) = \lambda t \mod(t^2)$, we get   $\pi_1(\hat f(t)) = \alpha \lambda^q t^q \mod(t^{q+1})$. Then,  
 the semiconjugacy  $f_1(\pi(t)) = \pi_1(\hat f(t))$  gives  $  \lambda^q  = a_{1, 0}$, and we are done.
\end{proof}

We resume the proof of Proposition~\ref{pro:invariant_curve}. We fix an affine coordinate $s$ on    $\hat C\simeq \P^1(\C)$ such that $\hat p = \infty$. Then, every lift  $\hat g \in \hat\Gamma$ can be written as an affine map  $\hat g(s)   = a_g s+b_g$. 
    
\begin{lem}\label{lem:integrability}
The functions $\log \abs{a_g}$ and $\log^+\abs{b_g}$ are $\nu$-integrable and $\ee( \log \abs{a_g})<0$. 
\end{lem}

\begin{proof}
For the spherical metric, the derivative of $\hat g$ at $ \infty$ in $\hat C$ is $1/a_g$. The computations 
of Lemma~\ref{lem:tangent_singular} show that  the derivative 
of $g$ acting on $X$ in the direction of the branch of $C$ at $\pi(\infty)$  
is $1/a_g^q$ for some $q\geq 1$. So~\eqref{eq:moment} implies that $\ee( \abs{\log \abs{a_g}})<\infty$. 
Since $\nu$ is uniformly expanding,  this direction  
is repelling on average: by   Lemma~\ref{lem:tangent_singular}, we get
$\ee( \log \abs{a_j\inv})>0$. 
To estimate $\abs{b_g}$, we note that  
$\dist_X(\pi(s), p) \asymp \abs{s}^{-q}$ when $s\in \C$ approaches $\infty$. Changing the affine coordinate $s$ if necessary, 
we may assume that 
$\pi(0)\neq p$. We get 
\begin{equation}
  \unsur{\abs{b_g}^q} \asymp \dist_X (\pi(\hat g(0)), \pi(\infty)) =  \dist_X(g(\pi(0)), g(p))\leq \norm{g}_{C^1}  \dist_X(\pi(0), p).
  \end{equation} 
From this and~\eqref{eq:moment} it follows that $\ee(\log^+\abs{b_g})<\infty$.
\end{proof}

The integrability provided by Lemma~\ref{lem:integrability} now allows us to construct a stationary 
 measure with full mass in the affine chart $\C\subset C$ with non-positive Lyapunov exponent 
 (relative to the affine metric). 
This is classical, we briefly recall the argument for completeness (see~\cite{brandt}). 
 For   $\omega = (g_n)_{n\geq 0}$, write   $g_n(s) = a_ns+b_n$,  and
   consider the sequence of right products  
 $r_n(\omega)  = g_0\cdots g_{n-1}$. One easily checks that 
\begin{equation}\label{eq:series}
r_n(\omega) (s) = a_{0}\cdots a_{n-1}s + \sum_{j=0}^{n-1} a_{0}\cdots a_{j-1} b_{j}. 
\end{equation}
For $\nu^\N$-almost every $\omega$, $\unsur{n} \log \abs{a_{0}\cdots a_{n-1}}$ converges to  $\lambda := \ee(\log\abs{a_g})<0$. 
Fix $\e < \abs{\lambda}$.  Since $\ee(\log^+\abs{b_g})<\infty$, $\sum_{j=0}^\infty \nu\{ \abs{b_g} > e^{\e j} \} <\infty$. 
By the Borel-Cantelli Lemma, $\abs{b_{j}}\leq e^{\e j} $ for $\nu^\N$-almost every $\omega$ and for large $j$; hence, the series on 
the right hand side of~\eqref{eq:series} converges. It follows that $r_n(\omega) (s) $ converges almost surely to a limit  $e_\omega$
that does not depend on $s\in \C$. The distribution of  $e_\omega$ is the desired stationary measure $\mu_\C$. 
If $\mu$ is any  stationary measure with $\mu(\C) = 1$, then  $r_n(\omega)_\varstar \mu$ converges to 
$\delta_{e_\omega}$ almost surely: this shows that $\mu_\C$ is the  unique stationary measure 
with $\mu(\C) = 1$; in particular, $\mu_\C$ is ergodic. 
 Since the affine derivative of $g$ is the constant $a_g$, 
the Lyapunov exponent of $\mu_\C$,  relative to  the affine metric, is equal to $\lambda$.  

To conclude the proof, note that $\mu:=\pi_\varstar (\mu_\C)$  is an ergodic  
$\nu$-stationary measure on $X$ which has  a well-defined Lyapunov exponent,  thanks to the moment condition~\eqref{eq:moment}. 
If $\mu$ gives positive mass to the singular set of $C$, then it must be   concentrated on a 
single singular point of $C$ (and likewise $\mu_\C$ is a single atom in $\hat C$). 
By Lemma \ref{lem:tangent_singular}  the corresponding branch is 
attracting on  average, which  contradicts uniform expansion. Therefore $\mu$ gives no mass to 
$\Sing(C)$, and we claim that its Lyapunov exponent $\lambda(\mu)\rest{TC}$ 
in the direction of $C$ equals $\lambda$ (even if the ratio between the 
 ambient and affine metrics on $\C\subset C$ is unbounded). Indeed, for 
 $\mu\times \nu^\N$-almost every $(x,\omega)$ and $v\in T_x^1C$, we can fix a subsequence $n_j$ such that $f^{n_j}_\omega(x)$ is far 
 from the singularities of $C$ (hence from $p=\pi(\infty)$).  If $j$ is large,
 $\unsur{n_j}\log\norm{D_xf^{n_j}_\omega}$ is both close to $\lambda$ and to
 $\lambda (\pi_\varstar \mu)\rest{TC}$. 
We conclude that  $\lambda (\pi_\varstar \mu)\rest{TC}<0$, which again is contradictory. The proof is complete. \end{proof}
 
\subsubsection{Zariski diffuse measures} 
From now on we focus on the case of a minimal Kähler surface  
 $X$  of Kodaira dimension zero, that is, 
 a torus, a K3 surface, or an Enriques surface. In this case 
 $\Aut(X)$ preserves a canonical volume form $\vol_X$ (see Example~\ref{eg:K3_intro}).

From  Corollary~\ref{cor:criterion_chung}, the obstruction 
to uniform expansion is the existence of a non-expanding stationary measure $\mu$.  Moreover, in the first case of Definition~\ref{defi:non_expanding_measure}, both 
exponents must vanish because we are in a volume preserving setting. In this situation, 
Theorems~\ref{thm:ledrappier_invariance_principle} 
and~\ref{thm:classification_proj_invariant}  give a precise description of $\mu$.

\begin{thm}\label{thm:rigidity}  
Let $X$ be a torus, a K3 surface, or an Enriques surface. Let $\nu$ be a probability measure on 
$\Aut(X)$ satisfying~\eqref{eq:moment} such that $\Gamma_\nu$ is non-elementary. 
If $\mu$ is a  Zariski diffuse  $\nu$-stationary  measure, the following properties are equivalent
\begin{enumerate}
\item[(a)]  $\mu$ is non-expanding;
\item[(b)]  the fiber entropy $h_\mu(X, \nu)$ vanishes.
\end{enumerate}
Morover under these assumptions, $\mu$ is invariant and   $h_\mu(f)=0$ for every $f\in \Gamma_\nu$.
\end{thm}

\begin{proof}  
As a preliminary step, observe that almost every ergodic component of $\mu$  is
Zariski diffuse: this  follows from the fact that there are  
only finitely many invariant curves and countably many isolated  periodic points. In addition, by linearity of the entropy, if $h_\mu(X, \nu)=0$ then   almost every ergodic component of $\mu$ has zero fiber entropy as well. 
Thus for both implications we may further assume that $\mu$ is ergodic as a stationary measure. 

 Since there is an invariant volume form, either 
  both  Lyapunov exponents of $\mu$ vanish or $\mu$ is hyperbolic.  In the first case, 
   the invariance principle guarantees that  $\mu$ is $\Gamma_\nu$-invariant and 
 the fibered version of the Ruelle inequality (see e.g.~\cite[\S 7]{stiffness}) implies that its fiber entropy 
 vanishes. If   $\mu$ is hyperbolic, the invariance of $\mu$ and the vanishing of the entropy follow from 
 \cite[Thm. 9.1]{stiffness}. 
 Thus, Property (a) implies Property (b), together with the invariance of $\mu$.
 
Consider the converse implication. Again, if $\mu$ has zero Lyapunov exponents then it is non-expanding and 
 invariant. Otherwise it is hyperbolic and 
by applying the whole argument  of~\cite{br} in the complex case, we infer that if the stable directions of 
$\mu$ depend on the itinerary, its   conditionals along Pesin unstable manifolds admit a non-trivial 
 translation invariance; in particular they are non-atomic.  It follows that $h_\mu(X, \nu)>0$ 
(see also~\cite[Rmk 9.2]{stiffness}). So under   assumption (b) the stable directions are non-random and, as 
already explained, $\mu$ is invariant by~\cite[Thm. 9.1]{stiffness}. 
 
The fact that $h_\mu(f)=0$ for all $f\in \Gamma_\nu$ will be shown in 
 Theorem~\ref{thm:entropy_every_f}.
\end{proof}

\subsubsection{Refined criterion}

The discussion of the previous paragraphs leads to a version of 
Theorem~\ref{thm:criterion_uniform_expansion_parabolic} that does not require $\Gamma_\nu$ to contain parabolic elements: 

\begin{thm}\label{thm:criterion_uniform_expansion}  
Let $X$ be a compact Kähler surface which is not rational. Let $\nu$ be a probability measure on 
$\Aut(X)$ satisfying~\eqref{eq:moment} such that $\Gamma_\nu$ is non-elementary. 
Then $\nu$ is uniformly expanding if and only if  the three following conditions 
hold:
\begin{enumerate}[ \em (1)]
\item every finite $\Gamma_\nu$-orbit is uniformly expanding;  
\item  there is no $\Gamma_\nu$-invariant algebraic curve; 
\item there is no  Zariski diffuse invariant  measure $\mu$ with zero fiber entropy. 
\end{enumerate}
\end{thm}

\begin{proof} If a compact Kähler surface $X$ is ruled (over a curve of positive genus) or has a positive Kodaira dimension, then $\Aut(X)$ is elementary (in the first case, it preserves the ruling; in the second case, it preserves the Kodaira-Iitaka fibration, acting as a finite group on the base). Thus, the Kodaira dimension of $X$ vanishes. If $X$ is not minimal, the uniqueness of the minimal model shows that there is a  $\Aut(X)$-invariant curve, and we know this is incompatible with uniform expansion (Proposition~\ref{pro:invariant_curve}). Now if ${\mathsf{kod}}(X)=0$, $X$ is minimal, and $\Aut(X)$ is non-elementary, then $X$ is a torus, a K3 surface, or an Enriques surface; hence, we can assume that $X$ is such a surface.

If $\nu$ is uniformly expanding, Property~(1) is obvious, Property~(2) follows from Proposition~\ref{pro:invariant_curve}, and
Property~(3) follows from Corollary~\ref{cor:criterion_chung} and Theorem~\ref{thm:rigidity}.    

Conversely, if these  properties hold, and if $\mu$ is an ergodic $\nu$-stationary measure then by Property~(2)
$\mu$ is either Zariski diffuse or finitely supported. Then,  Theorem~\ref{thm:rigidity} and Property~(1) imply that $\mu$ is not non-expanding, and we conclude with Corollary~\ref{cor:criterion_chung}.  
\end{proof}

\begin{proof}[Proof of Theorem~\ref{thm:criterion_uniform_expansion_parabolic}]
This follows directly from Theorem~\ref{thm:hyperbolic} and Theorem~\ref{thm:criterion_uniform_expansion}.
\end{proof}
  
%
  
\begin{rem}\label{rmk:conditional} 
The proof of  (b)$\Rightarrow$(a) in Theorem~\ref{thm:rigidity} relies on the following fact: {\textit{for a hyperbolic stationary measure, if the stable directions of $\mu$ depend on the itinerary, then its  unstable conditionals  satisfy some  non-trivial translation invariance}}. This is the
``easy part'' of the adaptation of~\cite{br} to complex surfaces; the ``difficult part'' would be to obtain some SRB property from this invariance (either on $X$ or on some totally real surface associated to the stationary measure). We did not provide a proof for this fact because the arguments of~\cite{br} can be applied directly. As a consequence, this fact is also used in the implication ``uniformly expanding implies (3)'' in 
Theorem~\ref{thm:criterion_uniform_expansion}. On the other hand, it is not used in Theorem~\ref{thm:criterion_uniform_expansion_parabolic} because in this case the condition~(3)  of 
Theorem~\ref{thm:criterion_uniform_expansion} is automatically
 satisfied, thanks to Theorem~\ref{thm:hyperbolic}; it is not used either for the part of Theorem~\ref{thm:criterion_uniform_expansion} asserting that the assumptions (1), (2) and (3)  imply uniform expansion.
\end{rem} 
  
\begin{rem}
Using Theorem~\ref{thm:hyperbolic_kummer} instead of Theorem~\ref{thm:hyperbolic} gives a version of 
Theorem~\ref{thm:criterion_uniform_expansion_parabolic} where the existence of a parabolic element in $\Gamma$ is replaced by the existence of a Kummer element. The details of the adaptation are left to the interested reader. 
\end{rem}

\subsection{Uniform expansion along finite orbits}\label{subs:finite_orbits}
Using classical results on random products of matrices, it is easy
 to characterize when a fixed point under $\Gamma_\nu$ is uniformly expanding. We say that a subgroup of $\GL_2(\C)$ is 
 \textbf{strictly triangular} if it is reducible with   exactly one invariant direction.  
 
 \begin{pro}\label{pro:fixed_UE}
Let $X$ be a torus, a K3 surface, or an Enriques surface. Let 
 $\nu$ be a probability measure on $\Aut(X)$ satisfying~\eqref{eq:moment}, and let $x_0$  be a fixed point of 
  $\Gamma_\nu$.  Then $\nu$ is uniformly expanding on $T_{x_0}X$ if and only if one of the following holds
  \begin{enumerate}[(a)]
  \item the 
 induced action of $\Gamma_\nu$ on $ T_{x_0}X$ is non-elementary; 
 \item  this   action  is strictly triangular and its invariant direction 
  is expanding.
\end{enumerate}
If $\nu$ is symmetric, it is  uniformly expanding on $T_{x_0}X$ if and only if (a) holds. 
 \end{pro}

In case (b) there exists $ u\in T_{x_0}X$ such that  $f_\varstar   u = \lambda_f  u$ for every $f\in \Gamma_\nu$, and the expansion means   that $\int \log \abs{\lambda_f} d\nu(f) >0$. 

\begin{proof}[Proof (see also~\cite{Prohaska-Sert:TAMS})]
By Lemma~\ref{lem:weak_UE2}, to prove uniform expansion it is enough to show that for every $ v\in T_{x_0}X$, $\liminf_{n\to\infty}\unsur{n}\log \norm{(f_\omega^n)_\varstar   v}  >0$. 
The proof is based on  the work of Furstenberg and Kifer~\cite{furstenberg-kifer} (see also~\cite[\S 3.7]{bougerol-lacroix}).  These references deal with general random products of matrices in $\GL_d(\R)$; in our volume preserving situation the Lyapunov exponents $\lambda_2\leq \lambda_1$ of the random 
product in $\GL_2(\C)$ 
satisfy  $\lambda_1+\lambda_2=0$, so they can be read off directly from the   action 
on 
$\P T_{x_0}X$. According to Theorems~3.5 and~3.9 of~\cite{furstenberg-kifer}, there are two possibilities:
\begin{enumerate}[(i)]
\item for every $ v\in T_{x_0}X$ and $\nu^\N$-almost every $\omega$, $\unsur{n}\log \norm{(f_\omega^n)_\varstar   v} \to \lambda_1$;
\item  there 
 exists a 
 non-random, $\Gamma_\nu$-invariant filtration $\set{0}=L_2< L_1<  L_0 = T_{x_0}X$
 and $\beta_1<\beta_0$ 
  such that for $i=0,1$ for any  $ v\in L_{i}\setminus L_{i+1}$, for $\nu^\N$-almost every $\omega$, 
 $\unsur{n}\log \norm{(f_\omega^n)_\varstar   v} \to \beta_i$. Furthermore $\beta_0= \lambda_1$.
 \end{enumerate}
 We now compare this dichotomy with the classification of subgroups of $\PGL_2(\C)$ (with a slight abuse of notation, we also denote by  $\Gamma_\nu$ the induced subgroup of $\PGL_2(\C)$).
 
$-$ If $\Gamma_\nu$ is strongly irreducible,  we are in case (i) and there are two possibilities. 
If $\Gamma_\nu$ is proximal (hence non-elementary)
then $\lambda_1>0$ and $\nu$ is uniformly expanding. If $\Gamma_\nu$ is not proximal,  it is contained in a
 compact subgroup and $\nu$ is not uniformly expanding.

$-$  If $\Gamma_\nu$ is irreducible but not strongly irreducible,  we are in case (i)  and there are two lines which are permuted by $\Gamma$. In some affine coordinate $z$ on $\P T_{x_0}X$, $\Gamma_\nu$ is conjugate to a subgroup of $\set{z\mapsto \lambda z^{\e}\; ; \ \lambda\in \C^\times, \ \e = \pm 1}$ and $\e=- 1$ with positive probability. In this case $\lambda_1 = 0$ (see e.g.~\cite[Prop. 5.3]{degenerations_SL2}), so $\nu$ is not uniformly expanding.

$-$  If $\Gamma_\nu$ is  reducible it preserves one or two directions in $T_{x_0}X$. If $\Gamma$ preserves a   direction with exponent $\leq 0$, then  $\nu$ is not uniformly expanding. So, we can assume that $\Gamma$ preserves a unique direction, and that the corresponding exponent $\beta$ is positive.  By (i) and (ii)  we see that $\lim_{n\to\infty}\unsur{n}\log \norm{(f^n_\omega)_\varstar   v}\geq \beta$ for any $ v\in T_{x_0}X$ and almost every $\omega$; so  
  $\nu$ is   uniformly expanding. 

This covers all possible cases and the proof is complete.
\end{proof}
 
Let $F$ be a finite set, viewed as a $0$-dimensional manifold, and  $V$ be a  real or complex vector bundle of dimension $d$ over $F$; identify
$V$ with $ F\times \mathbf{K}^d$, for ${\mathbf{K}} = \R$ or $\C$. Let $\GL(V)$ be the group of bijections of $V$ acting  linearly on   fibers: it is   a semidirect product $\GL(V) \simeq \mathfrak{S}(F) \ltimes \GL_d({\mathbf{K}})^F$ where  $\mathfrak{S}(F)$ acts on 
$\GL_d({\mathbf{K}})^F$   by permuting the factors. We say that a subgroup of $\GL(V)$ is {\bf{strongly irreducible}} if it acts transitively on $F$
and the stabilizer of any $x\in F$ acts strongly irreducibly on the fiber $\{ x\} \times {\mathbf{K}}^d$ of $V$; equivalently, if there is no invariant and finite  collection of subspaces of  dimension $\neq 0, d$ in some fibers of $V$. Similar notations and notions are defined for $\PGL(V)$.

Assume now that $F$ is a finite $\Gamma$-orbit on $X$,
and consider the induced action of $\Gamma$ on  $TX\rest{F}:=\bigcup_{x\in F} T_x X$. 
We say that this action is \textbf{non-elementary} 
if its image in $\PGL\lrpar{TX\rest{F}}$  is strongly irreducible and unbounded. When 
$\Gamma$ preserves a volume form on $X$, its  image  in $\GL\lrpar{TX\rest{F}}$  is unbounded if and only if it is unbounded in $\PGL\lrpar{TX}$. 
We say that it is \textbf{strictly triangular} if the only proper $\Gamma$-invariant subbundle in $T X\rest{F}$ is given by a $1$-dimensional subbundle $L\subset TX\rest{F}$. 

Pick a point $x$ in $F$ and set $\Gamma_x = \Stab_{\Gamma}(\set{x})$. Since $F$ is an orbit, $[\Gamma:\Gamma_x]=\abs{F}$ and  the image of $\Gamma$ in $\PGL\lrpar{TX\rest{F}}$ is unbounded if and only if the image of $\Gamma_x$ in $\PGL\lrpar{T_xX}$ is unbounded. Thus, one easily gets the following lemma.

\begin{lem}\label{lem:stab_NE}
If $F$ is a finite $\Gamma$-orbit, the action of 
$\Gamma$ on $TX\rest{F}$ is non-elementary (resp. strictly triangular) if and only if for some, hence any,  $x\in F$
 the action of $\Stab_{\Gamma}(\set{x})$ on $T_{x}X$ is non-elementary (resp. strictly triangular).  
\end{lem}

  
\begin{thm}\label{thm:finite_UE}
Let $X$ be a torus, a K3 surface, or an Enriques surface. Let 
 $\nu$ be a probability measure on $\Aut(X)$ satisfying~\eqref{eq:moment+}, and $F$ be a finite $\Gamma_\nu$-orbit.  
 Then $\nu$ is uniformly expanding on $F$ if and only if the 
 induced action of $\Gamma_\nu$ on $TF$ is 
 \begin{enumerate}[(a)]
 \item either non-elementary;
 \item or strictly triangular and the field of invariant directions $L\subset TX\rest{F}$ is uniformly expanding. 
 \end{enumerate}
If $\nu$ is symmetric,  it is uniformly expanding on $F$ if and only if (a) holds. 
\end{thm}

\begin{proof} 
Let $\Gamma_F$ be the finite index subgroup fixing every point of $F$. 
Assume that $\nu$ is uniformly expanding. Then by Proposition~\ref{pro:induced_UE}, for some 
$n_0$, the induced measure $(\nu^{(n_0)})_{\Gamma_F}$ is uniformly expanding. Therefore, 
 by Proposition~\ref{pro:fixed_UE}  $\Gamma_F$ satisfies Property~(a) or~(b) at every point of $F$, and 
  we conclude by Lemma~\ref{lem:stab_NE}. 
  Conversely, assume that~(a) or~(b) holds. Note that by Theorem~\ref{thm:M$_2$}, 
  $\nu_{\Gamma_F}$ satisfies~\eqref{eq:moment+}. 
   By Lemma~\ref{lem:stab_NE} and Proposition~\ref{pro:fixed_UE}, 
  $\nu_{\Gamma_F}$ is uniformly expanding on $F$, 
  hence by Proposition~\ref{pro:induced_UE2} $\nu$ is uniformly expanding on $F$, as desired. 
    \end{proof}

This theorem shows that  when $\nu$ is symmetric
all conditions in Theorem~\ref{thm:criterion_uniform_expansion} depend only on 
$\Gamma_\nu$, and not on $\nu$. Thus  we obtain:

\begin{cor}\label{cor:independence_nu}
Let $X$ be a torus, a K3 surface, or an Enriques surface. Let $\Gamma$ be a non-elementary subgroup of $\Aut(X)$. Let $\nu$ and $\nu'$  be  symmetric probability measures 
on $\Aut(X)$ satisfying~\eqref{eq:moment+} such that  $\Gamma_\nu  = \Gamma_{\nu'}=\Gamma$. Then $\nu$ is uniformly expanding if and only if $\nu'$ is 
uniformly expanding.
\end{cor}

In the following if $X$ is a torus, K3, or Enriques surface,  we will    
 say that the action of a non-elementary subgroup $\Gamma\subset\Aut(X)$ is {\bf{uniformly expanding}} 
 if this property holds for some (hence any)
  symmetric probability measure $\nu$ satisfying~\eqref{eq:moment+}  and 
 generating~$\Gamma$.

\section{Examples of uniformly expanding actions}

\subsection{A finitary version of Theorem~\ref{thm:criterion_uniform_expansion_parabolic} and application to Wehler surfaces}\label{subs:finitary}
In~\cite[\S~\S 7-8]{chung}, Chung uses computer assistance to prove the uniform expansion of 
some concrete algebraic actions on real surfaces. 
In our situation 
 Theorem~\ref{thm:criterion_uniform_expansion_parabolic} can be used to
 check uniform expansion, but this requires a description of all invariant Zariski-closed subsets.
As already explained, invariant curves can be determined by cohomological computations; for instance, if $X$ is a generic Wehler surface,  there is no $\Aut(X)$-invariant curve. Thus the main problem is to study finite orbits. 
 
 If the group $\Gamma$ is non-elementary, contains parabolic elements, and has no invariant curve, the main 
 result of~\cite{finite_orbits} says that $\Gamma$ admits only finitely many finite orbits, except when $(X,
 \Gamma)$ is a Kummer example. However, the proof given in~\cite{finite_orbits} does not provide any bound
  on the number or the lengths of such  orbits; so,  there is \textit{a priori} no hope of numerically checking 
  uniform expansion along all of them, nor proving that there are no finite orbits. The next result explains how to overcome this issue. 
 
 \begin{thm}\label{thm:effective}
  Let $X$ be a smooth projective surface and 
  $\Gamma$ be a non-elementary subgroup of $\Aut(X)$ containing parabolic elements, which does not  
preserve any algebraic curve. Assume that we are given:
\begin{enumerate}
\item[(i)] algebraic equations for $X$, and the formulas defining a generating subset $S$ of $\Gamma$;
\item[(ii)] a basis of $\NS(X;\R)$ and the matrices of $s^\varstar\colon \NS(X;\R)\to \NS(X;\R)$, for $s$ in $S$;
\item[(iii)] a parabolic element $g\in \Gamma$, given as a word in the generators $s\in S$, 
and its invariant fibration $\pi\colon X\to B$.
\end{enumerate}  
Then, there is an analytically computable integer $N(X, \Gamma)$ 
such that  the action of $\Gamma$ on any 
 finite orbit of  length greater than $N(X, \Gamma)$ is non-elementary. 
  \end{thm}

By {\bf{analytically computable}}, we mean 
computable by a computer able to solve real analytic equations; by {\bf{algebraically computable}}, we mean 
computable by a computer able to solve algebraic equations. The proof will provide  an analytically  computable subset containing all possible non-expanding finite orbits.

\begin{eg} Let $h\in \Gamma$ be a conjugate of $g$ with a distinct invariant fibration. 
Denote by $\mathrm{Tor}_N(g)$  the finite set of 
fibers of the $g$-invariant fibration in which $g$ is a periodic translation of period $\leq N$. 
Then, the set of finite orbits of $\Gamma$ of length $\leq N$ is algebraically computable since it is contained in 
\begin{equation} 
\mathrm{Tor}_N(g)\cap \mathrm{Tor}_N(h) =\{x\in X\; ; \; g^N(x)=h^N(x)=x \}.
\end{equation} 
\end{eg}
 
A typical application of Theorem~\ref{thm:effective}  is to the Wehler family. 
 Recall from  \S~\ref{subs:wehler}
   that ${ \cW}_0$ is the family of Wehler surfaces which are smooth and do not contain any 
   fiber of   the three natural projections 
 $  (\pp^1)^3\to (\pp^1)^2$.
 Under these assumptions the group $\Gamma$ generated by the three basic involutions 
 $\sigma_1$, $\sigma_2$ and $\sigma_3$ is non-elementary and 
 has no invariant curve (see~\cite[Prop. 2.2]{finite_orbits}).
It turns out that in this case $N(X, \Gamma)$ is   constant on a Zariski 
 dense open subset  (see Proposition~\ref{pro:N(X)_bounded} below). 
This leads to:

\begin{thm}\label{thm:wehler_zariski} 
There is a dense Zariski open subset of  ${ \cW}_0$
 (resp. of  the family ${ \cW}_0(\R)$ of real Wehler surfaces),  
in which the action of $\Gamma=\langle \sigma_1, \sigma_2, \sigma_3\rangle$ is 
  uniformly expanding on $X$. 
\end{thm}


\subsubsection{Preliminaries on Halphen twists}\label{par:preliminaries2_on_Halphen} Let us resume the discussion from \S~\ref{subs:halphen} and add a few preliminaries on Betti foliations and the non-twisting locus. Let $h$ be a Halphen twist with associated fibration $\pi\colon X\to B$. 
Consider a  simply connected open subset $U$  of $B^\circ$ together with a section $\sigma\colon U\to X$ of $\pi$  
and a continuous frame for the homology of the fibers above $U$. 
For $w\in U$, one can identify the fiber $X_w$ to $\C/\Lat(w)$ ($\sigma(w)$ corresponding to the zero of $\C/\Lat(w)$), as in \S~\ref{subs:halphen}.
Then, above $U$, there is   a unique real-analytic diffeomorphism $\Psi: \pi \inv(U)\to U\times \R^2/\Z^2$ 
such that 
\begin{enumerate}
\item[(a)] $\pi\circ \Psi=\Psi\circ \pi_U$, where $\pi_U$ is the projection onto $U$;
\item[(b)] $\Psi$ maps $\sigma$ to the zero section $w\mapsto (w, (0,0))$ of $\pi_U$, and  maps the basis of $H_1(X_w;\Z)$ to the standard basis of $H_1(\R^2/\Z^2; \Z)=\Z^2$;
\item[(c)]  on each fiber, $\Psi$ is a real analytic isomorphism of real Lie group.
\end{enumerate}
Above $U$, the {\bf{Betti foliation}} is  the foliation by 
  submanifolds of the form $\Psi\inv(U\times \set{(x,y)})$; these leaves are local holomorphic sections of $\pi$, with $\sigma$ corresponding to $\Psi\inv(U\times \set{(0,0)})$.
  Conjugating by 
$\Psi$,  we get
\begin{equation}\label{eq:definition_T}
\Psi\circ h\circ \Psi^{-1}\colon (w, (x,y))\mapsto (w, (x,y)+T(w)),
\end{equation} 
where $T\colon U\to \R^2/\Z^2$ is real analytic. By~\cite[Lem. 3.9]{invariant}, the map $T$ 
is an (orientation preserving)
branched covering,  so it behaves topologically like $w\mapsto w^k$.  In $U$, the 
{\bf{non-twisting locus}} $\mathrm{NT}_{h}$ 
 is the set $\{w\in U\; ; \; D_wT = 0\}$; equivalently, $\mathrm{NT}_{h }\cap U = \pi(\set{t_1, \ldots , t_q})$, where 
 $\set{t_1, \ldots , t_q}$ is the set of tangencies between   
the Betti foliation and the section $h \circ \sigma$. These definitions do not depend on the  above choices and $\mathrm{NT}_{h}$ can indeed be defined globally on $B^\circ$. 
A key  fact is that $\mathrm{NT}_{h}$  is  a finite  subset of $B^\circ$
(see~\cite[Prop. 3.14]{invariant} or~\cite[Cor. 7.7.10]{duistermaat}).
We denote by $\abs{\NT_h}$ its cardinality, and by $\mult(\NT_h)$ 
its cardinality counted with multiplicity, that is, taking into account the 
degree of the local branched covering $T$.  

Note that, once $h$ and $\pi$ are given, {\textit{the set $\NT_h\subset B^\circ$ is analytically computable}}: one has to compute the periods of $X_w$ 
to get $\Lat(w)$, then $\Psi$ is $\R$-linear from $\C/\Lat(w)$ to $\R^2/\Z^2$, and $T$ is then obtain from $h$ by conjugacy.

\subsubsection{Proof of Theorem~\ref{thm:effective}}   
As in \cite{invariant}, for $(g,h)\in \Hal(\Gamma)^2$
we  set 
\begin{equation}
\STang(\pi_g, \pi_h) = 
\Sing(\pi_g)\cup\Sing (\pi_h)\cup \Tang^{\mathrm{tt}}(\pi_g, \pi_h), 
\end{equation} 
where $\Sing(\pi_g)$ is the union of all singular and multiple fibers, and 
$\Tang^{\mathrm{tt}}(\pi_g, \pi_h)$ is the part of the tangency locus of $\pi_g$ and $\pi_h$ which is not contained in $\Sing(\pi_g)\cup\Sing (\pi_h)$.
Put $\NT_g^X=\pi_g^{-1}(\NT_g)$ (so that $\NT_g^X$ is a curve in $X$) and
likewise $\NT_h^X=\pi_h^{-1}(\NT_h)$.  

\begin{lem}\label{lem:NT1}
Let  $g,h$  be a pair of Halphen twists in $\Gamma$ with distinct invariant fibrations, and let  $x\in X$ be a point with a finite $\Gamma$-orbit.  
If this orbit is not uniformly expanding, then
it is   contained in 
$ \STang(\pi_g, \pi_h)\cup \NT_g^X\cup \NT_h^X$. 
\end{lem}

\begin{proof}
We argue by contraposition: replacing $x$ by another point in its orbit if necessary, 
 we   assume that  $x\notin  \STang(\pi_g, \pi_h)\cup \NT_g^X\cup \NT_h^X$, and we want to show that its orbit is 
 uniformly expanding. 
Since $\Gamma(x)$ is finite, there are positive integers $k$ and $\ell$ 
such that $g^k$ and $h^\ell$ are in  $\mathrm{Stab}_\Gamma(x)$. By definition of the non-twisting locus, 
$g^k$ and $h^\ell$
induce parabolic homographies on $\pp(T_xX)$; and since $x\notin \Tang^{tt}(\pi_g,\pi_h)$, the
  fixed points of these homographies are distinct; thus, the action of 
$\langle g^k, h^\ell\rangle$ on $\pp(T_xX)$ is non-elementary. By Proposition~\ref{pro:fixed_UE} and Lemma~\ref{lem:stab_NE}, the orbit of $x$ is uniformly expanding. \end{proof}


The intersection number of $\NT_g^X \cup \STang(\pi_g, \pi_h)$ (resp. $\NT_h^X \cup \STang(\pi_g, \pi_h)$) with a smooth fiber $X_w^h$ (resp. $X_w^g$) does not depend on the fiber. Let $n_0(g,h)$ be the maximum of these intersection numbers:
\begin{equation}\label{eq:definition_n0}
n_0=\max\{ [\NT_g^X \cup \STang(\pi_g, \pi_h)] \cdot [X_w^h]\;  ; \; [\NT_h^X \cup \STang(\pi_g, \pi_h]\cdot [X_w^g])\}.
\end{equation}
The set $\STang(\pi_g, \pi_h)$ can be computed algebraically, thus 
\begin{equation} 
n_0(g,h)\leq A(g,h) \max(\vert \NT_g\vert \, ; \, \vert \NT_h\vert) + B(g,h)
\end{equation}
 where $A(g,h)$ and $B(g,h)$ can be computed algebraically (by computing the tangency loci and intersection numbers).  
Then, we set
\begin{equation}\label{eq:definition_n(g,h)}
n(g,h)=n_0(g,h) !
\end{equation}

\begin{lem}\label{lem:definition_of_bad(g,h)}
Let  $g,h$  be a pair of Halphen twists in $\Gamma$ with distinct invariant fibrations. Let
  $x\in X$  be  such that $\Gamma(x)$ is finite and not 
uniformly expanding. Then $\Gamma(x)$ is contained in 
$$
  \STang(\pi_g, \pi_h) \cup
(\NT_g^X\cap \NT_h^X)\cup (\NT_g^X\cap \NT_{h^n gh^{-n}}^X) \cup (\NT_h^X\cap \NT_{g^nhg^{-n}}^X),
$$
where $n=n(g,h)$ is defined by~\eqref{eq:definition_n(g,h)}.
\end{lem}

\begin{proof} 
The statement of the lemma concerns the orbit $\Gamma(x)$, but we only have to prove it for $x$ itself. 
If $x\in  \STang(\pi_g, \pi_h) \cup (\NT_g^X \cap \NT_h^X)$ we are done. 
Otherwise by Lemma~\ref{lem:NT1}, $x$ belongs to 
$\NT_g^X\setminus  \NT_h^X$ or  $\NT_h^X\setminus \NT_g^X$. 
Assume that $x\in  \NT_g^X\setminus \NT_h^X$.
 The $h$-orbit of $x$ is finite and by Lemma~\ref{lem:NT1} again,
for every $q$, $h^q(x)$ is   contained  in  
$X^h_{x}\cap (\NT_g^X \cup \STang(\pi_g, \pi_h))$ (here we abuse notation and 
write $X^h_{x}$ for $X^h_{\pi_h(x)}$). Thus, 
$h^{n}(x) = x$, where $n=n(g,h)$. Set  $f = h^{n} g h^{-n}$. The fiber   $X^f_x$ associated to $f$ through $x$
 is   $h^{n}( X^g_x)$, and since $x\notin \NT_h^X$, $X^f_x$ is transverse to   
 $X^g_x$ at $x$, as well as to  $X^h_x$. Moreover, $x$ belongs to $\NT_f^X$, because $x$ belongs to $\NT_g^X$ and $h^n(x)=x$. Hence   $x\in \NT_g^X\cap \NT_{h^ngh^{-n}}^X$. Doing the same 
 in the case where $x\in \NT_h^X\setminus  \NT_g^X$ completes the proof. 
\end{proof}

The set $\NT_g$ is analytically computable (by \S~\ref{par:preliminaries2_on_Halphen}), and $\Crit(\pi_g)$ is algebraically computable. Similarly, if $h$ is in $\Hal(\Gamma)$,  $\Tang^{tt}(\pi_g,\pi_h)\cap \NT_f^X$ is analytically computable.  
The  previous lemma 
 shows that all  non uniformly expanding finite orbits are contained in 
$$\mathrm{Bad}(g,h):=\STang(\pi_g, \pi_h) \cup
(\NT_g^X\cap \NT_h^X)\cup (\NT_g^X\cap \NT_{h^n gh^{-n}}^X) \cup (\NT_h^X\cap \NT_{g^nhg^{-n}}^X)
$$
for every pair $(g,h)\in \Hal(\Gamma)^2$ with distinct invariant fibrations, where $n = n(g,h)$ as in Equation~\eqref{eq:definition_n(g,h)}.
Intersecting these sets for various choices of  $(g,h)$, we expect to get a finite analytically computable set. 
Observe that  $\mathrm{Bad}(g,h)$ 
is the union of $\STang(\pi_g, \pi_h)$ and a finite set, because $\NT_g^X\cap \NT_h^X$ is finite when $\pi_g$ and $\pi_h$ are distinct. So, what remains to do is to
  \textit{exhibit an explicit finite set of  pairs $(g,h)$ such that the 
  intersection of the $\STang(\pi_g, \pi_h)$ is finite}. 
We first treat the case of Wehler surfaces, 
which is sufficient to proceed with Theorem~\ref{thm:wehler_zariski}.
 
 \smallskip
 
\noindent \textit{Conclusion of the proof of Theorem~\ref{thm:effective} in the Wehler case.}
Fix a  Wehler surface $X\in  {\cW}_0$ and consider the three pairs $(g_1, g_2)$, $(g_2,g_3)$, $(g_3,g_1)$, where $g_1 = \sigma_2\circ \sigma_3$, $g_2 = \sigma_3\circ\sigma_1$ and $g_3 =\sigma_1\circ\sigma_2$. Note that the $g_i$-invariant fibration is the $i$-th projection $\pi_i$.
 
Assume that the intersection of the divisors $\STang(\pi_i, \pi_j)$ contains an irreducible curve $D\subset X$. If $D$ is contained in $\Sing(\pi_i)\cap \Sing(\pi_j)$ with  $i\neq j$, then $(\pi_i, \pi_j)$ maps $D$ onto a point and this contradicts the fact that $X\in { \cW}_0$. If $D$ is contained in, say, $\Tang^{tt}(\pi_1,\pi_2)$ 
 and $\Tang^{tt}(\pi_2, \pi_3)$,  the three fibrations are pairwise tangent  along $D$, and we obtain a contradiction because there is no tangent vector $v\neq 0$ to $(\pp^1)^3$ which is mapped to $0$ by each $D\pi_i$. The last possibility is that $D$ is contained in, say,  $\Tang^{tt}(\pi_1,\pi_2)$ and $\Sing(\pi_3)$.    In this case, there is a point $p$ on $D$ at which $D_p\pi_3\colon T_pX\to T_{\pi_3(p)}\P^1$ is equal to $0$, and 
 at such a point, the same contradiction applies. This shows that 
 \begin{equation}\label{eq:finite-bad-set}
 F(g_1,g_2,g_3):=\mathrm{Bad}(g_1,g_2)\cap \mathrm{Bad}(g_2,g_3)\cap \mathrm{Bad}(g_3,g_1) 
 \end{equation}
 is finite, with an analytically computable cardinality, and   the proof is complete. \qed

\begin{rem}\label{rem:N(X)}
The above proof provides a computation of the integer $N(X)$  involving:
 \begin{enumerate}
 \item algebraic quantities that are constant on $\cW_0$, like $\STang(\pi_i,\pi_j)\cdot \STang(\pi_j,\pi_k)$, 
 \item  $\vert \mathrm{NT}_{g_i}\vert$ for 
$i=1, 2, 3$.
\end{enumerate} 
{\textit{Therefore, if $\vert \mathrm{NT}_{g_i}\vert \leq B$, then $N(X)\leq N(B)$ for some $N(B)$ depending only on $B$.}} 

Indeed, the number $n$ in Lemma~\ref{lem:definition_of_bad(g,h)} depends only on $n_0$ (see Equations~\eqref{eq:definition_n(g,h)}) and by Equation~\eqref{eq:definition_n0}) $n_0$ is bounded by a function of $B$. 
Then, because the norm of $(g_i^{n_0})^*\colon \NS(X)\to \NS(X)$ is bounded by $Cn_0^2$ for some uniform constant $C$, we obtain 
$\NT_{g_i}^X \cap \NT_{g_j^ng_ig_j^{-n}}^X\leq C' n_0^2 B^2$  
for some constant $C'$ and the result follows. \qed
\end{rem}
 
  \noindent \textit{Conclusion of the proof of Theorem~\ref{thm:effective} in the general case.} \label{conclusion_proof_thm_effective} 
By assumption, $\Gamma$ is non-elementary and has no invariant   curve.  Let $\Gamma^*$ be its image in $\GL(\NS(X;\Z))$. 

If $g$ is the parabolic element given in assumption~(iii) of the theorem, up to sign, there is a unique integral primitive class $c(g)\in \NS(X)$ such that $g^\varstar c(g)=c(g)$ and $c(g)\cdot c(g)=0$. By the  assumptions~(ii) and~(iii), this class can be computed explicitly. An element $f$ of $\Aut(X)$ preserves the $g$-invariant fibration $\pi$ (permuting its fibers) if and only if it fixes $c(g)$. Since $\Gamma$ is non-elementary, it is not contained in the stabilizer of $c(g)$. Thus, according to Proposition~3.2 of~\cite{eskin-mozes-oh}, there is a computable integer $N$, and a composition $f$ of length $N$ in the generators $s\in S$ that does not preserve $c(g)$. Then, $h:= f\circ g \circ f^{-1}$ is a parabolic element of $\Gamma$ with invariant fibration $\pi\circ f\neq \pi$.

Since  $g$  and its  invariant fibration  $\pi$, as well as $f$,  are explicit, we can compute the degree of $\STang(\pi_g, \pi_h)$ (for the embedding $X\subset \P^m(\C)$ given by assumption (i)). Denote by $(C_i)_{i\in I}$ the irreducible components of $\STang(\pi_g, \pi_h)$; we have $\vert I\vert \leq \deg(\STang(\pi_g, \pi_h))$. 

Suppose that for each $C_i$, one can exhibit some $f_i\in \Gamma$ for which  $f_i(C_i)\not \subset \STang(\pi_g, \pi_h)$. Then the set of pairs 
\begin{equation}
\set{(g,h)} \cup \set{ (f_igf_i\inv, f_i h f_i\inv)\; ; \; \ i\in I} 
\end{equation}
satisfies $\vert \bigcap_i\STang(\pi_{g_i}, \pi_{h_i})\vert < +\infty$, and we are done because the cardinality of this finite set is algebraically computable. 
So, we now fix such an irreducible component $C_i$, and we construct such an $f_i$. 

First, assume that $C_i$ is an irreducible component of $\Tang^{\mathrm{tt}}(\pi_g, \pi_h)$. Then $C_i$ is generically transverse to $\pi$, hence $\deg(g^n(C_i))$ tends to infinity. We can thus set $f_i=g^{n_i}$ for some large enough $n_i$ (explicitely computable from  the action on $\NS(X)$).

The second case is when $C_i$ is an irreducible component of $\Sing(\pi_g)\cup\Sing(\pi_h)$; 
in particular, its self-intersection $C_i^2$ is $\leq 0$. 
By~\cite[Thm D]{finite_orbits}, there exists a loxodromic element $f_0\in \Gamma$ 
 without invariant curve; in particular  $f_0^{\abs{I}!}(C_i) \neq C_i$. Since $f_0$ is loxodromic, this inequation is equivalent to 
$(f_0^{\abs{I}!})_\varstar [C_i] \neq  [C_i]$.  Indeed, either $C_i^2=0$ and we readily get a contradiction 
since a loxodromic element does not fix any non-zero isotropic class, or $C_i^2<0$
 and this follows from $f_0^{\abs{I}!}(C_i) \neq C_i$ since 
$[C_i]$ determines $C_i$ when the self-intersection is negative.
Thus, if we set 
\begin{equation}
 W_i:=\set{f\in \GL(\NS(X, \R)): f^{\abs{I}!}_\varstar [C_i] = [C_i]},
 \end{equation}
 we see that $\Gamma^*$ is not contained in $W_i$. Proposition 3.2 of \cite{eskin-mozes-oh} then provides a computable element $f\in \Gamma$ 
such that $f\notin W_i$. Now, if  $f^q(C_i)$ were contained in $\STang(\pi_g, \pi_h)$ for  $0\leq q\leq \vert I\vert$, we would find two integers $q_1< q_2\leq \vert I\vert$ such that $f^{q_2-q_1}(C_i)=C_i$; in particular, $f^{\vert I\vert !}(C_i)$ would be equal to $C_i$, a contradiction. Thus, there is an iterate $f_i:=f^{q_i}$, with $q_i \leq \vert I\vert$, such that $f_i(C_i)\not \subset \STang(\pi_g, \pi_h)$, and the proof is complete.
\qed
 
\subsubsection{Proof of Theorem~\ref{thm:wehler_zariski} } Recall that, for Wehler surfaces, $\Gamma = \bra{\sigma_1, \sigma_2, \sigma_3}$.  

\begin{pro}\label{pro:N(X)_bounded}
There exists an analytically computable integer $N$ such that for  
any Wehler  surface  $X\in { \cW}_0$,   any 
finite $\Gamma$-orbit of  length $>N$ is non-elementary. 
\end{pro}
 
This uniform bound is the main step towards Theorem~\ref{thm:wehler_zariski}.
In view of Remark~\ref{rem:N(X)}, this proposition 
follows from Theorem~\ref{thm:effective} and the following uniformity result. 
 
\begin{pro}\label{pro:NT_bounded}
For any $g\in \set{g_1, g_2, g_3}$, the cardinality of $\NT_g$ is uniformly bounded 
in $\cW_0$. 
\end{pro}

Let  $\cX\subset  \cW_0\times (\P^1\times \P^1\times \P^1)$ be 
 the universal family of Wehler surfaces, as in \cite[\S 2]{finite_orbits}.   
 As $X$ varies in $ \cW_0$, the automorphisms $g_i$ and their 
invariant fibrations $\pi_i$ depend on $X$, but for notational simplicity we drop the dependence in $X$. 

From now on, we fix $g\in \set{g_1, g_2, g_3}$; its invariant fibration $\pi\colon X\to \P^1$ is the restriction of one of the projections $\pi_i$ to $X$; its base does not depend on $X$. 

\begin{lem}\label{lem:NT_local}
Let $X_0\in \cW_0$,   $w_0\in \NT_g$, and   $k$ be the multiplicity of $w_0$ in $\NT_g$.
Let $U\subset \P^1$ be a topological disk  such that $\overline U\cap \NT_g = \set{w_0}$ and $\overline U\cap \Crit(\pi)=\emptyset$.
 Then, there exists a neighborhood $V$
of $X_0$ in  $\cW_0$ such that for any $X$ in $V$, the total multiplicity of 
$\NT_g$ in $U$ is equal to  $k$. 
\end{lem}

\begin{proof}
Fix an open connected neighborhood $V$ of $X_0$ such that for $X$ in $V$,  
\begin{itemize}
\item $U$ does not intersect any of the sets $\Crit(\pi)$;
\item  there is a section $w\mapsto \varsigma_X(w)$
 of $\cX\to  \cW_0\times \P^1$ above $V\times U$, together    with a continuous choice of basis for the homology of the fibers of $\pi$ above $U$.  
\end{itemize}
 Then the sections, the Betti foliations (above $U$), and their lifts to $U\times \C$   all   
 depend  continuously  on $X$ in $V$. In particular, we can find a disk $U'\subset \P^1$, with 
 $w_0\in U'\Subset U$, whose boundary is a 
 smooth Jordan curve $\gamma$, and such that 
  for any $X\in V$, the Betti foliation  is transverse to $g\circ \varsigma_X$ above $\gamma$. In particular, 
  $\NT_g$ is disjoint from $\gamma$. 
   
 Now, recall that the map $T$ defined in Equation~\eqref{eq:definition_T} 
behaves topologically like $w\mapsto w^{k+1}$; in such a local coordinate, 
$k+1$ is  the winding number of the curve $T\circ \gamma$ around~$T(0)$.
Since $\NT_g$ stays disjoint from $\gamma$ for $X$ in $V$, 
this winding number is constant in $V$; thus,  the   number of points of $\NT_g$    enclosed by $\gamma$ (counted with multiplicity) stays constant on $V$. The  
lemma follows. 
\end{proof}

\begin{lem}\label{lem:effective_analytic}
There exists a   proper
semi-algebraic  subvariety $\mathcal Z_g \subset   \cW_0$ of positive codimension
such that $\mult(\NT_g)$
is locally constant in $ \cW_0\setminus \mathcal Z_g$.
\end{lem}

\begin{proof} ~

\noindent{\textbf{Step 1: }}\textit{Keeping away from the singular fibers. }
Fix $X_0\in \cW_0$ and $w_0\in X_0$  a critical value of $\pi$.
It is shown in \cite[Lem. 3.11]{invariant} that  ${\NT}_{g}$ does not accumulate $w_0$. 
Here, we show that outside a semi-algebraic subvariety $\mathcal Z_g\subset  \cW_0$  this non-accumulation   holds uniformly with respect to~$X$: we shall construct  a neighborhood $V\times U$ of $(X_0, w_0)$ such that $U$ is disjoint from $\mathrm{NT}_{g}$ for every $X$ in $V$. 
For this, we   review the proof of \cite[Lem. 3.11]{invariant}  
and make it locally uniform in $X$ under  appropriate hypotheses on $X_0$.  

Define $\cW_1$ to be the dense, Zariski open subset (\footnote{To show that $\cW_1$ is dense, we only have to show that it is non-empty. 
This is a consequence of the following fact. {\textit{Let $X$ be in $\cW_0$, let $\pi_1\colon X\to \P^1$ be the first projection, and let $m$ be a critical point of $\pi_1$.  Let $F$ be the fiber of $\pi_1$ containing $m$. Then, each of the conditions
\begin{enumerate}
\item the singularity of $F$ at $m$ is degenerate (in the sense of Morse, i.e.\ it is not a $A_1$-singularity);
\item $F$ contains a second singular point $m'$
\end{enumerate}
defines a proper subset of $\cW_0$.}} In other words, these properties (1) and (2) disappear after a generic small perturbation of $X$ in $\cW_0$, which can be checked directly. }) of $ \cW_1$ 
such that for any $X\in  \cW_1$ and any 
$i\in \set{1,2, 3}$, all singular fibers of $\pi_i$ are of type $I_1$.  In this case 
 there are $24$ such fibers (the Euler characteristic of a K3 surface is $24$, the contribution to the Euler characteristic of a smooth fiber is $0$, and the contribution of an $I_1$ fiber is $1$).  Suppose that $X_0\in \cW_1$.

Fix a small disk $U\subset \P^1$ centered at $w_0$  and containing no other singular value of $\pi\colon X_0\to \P^1$.   
Fix a neighborhood $V$ of $X_0$ in $ \cW_1$, and local coordinates on $U$ (depending on $X$), so that (i) this property persists for $X\in V$ and (ii) the unique singular value of $\pi$ in $U$ is  $w_0=0$. Let
  $X^\#_U $ be the complement in $X_U^g:=\pi\inv(U)$ of the unique singular point of $X_{w_0}^g$. 
We fix a reference section $\varsigma_X:U\to  X^\#_U$ depending holomorphically on~$X\in V$ and~$w\in U$.  

For $X\in V$ and $w\in U\setminus\{w_0\}$ we can write 
$X_w^g\simeq \C/\Z\oplus \Z\tau_X(w)$, as in \S~\ref{subs:halphen}. Since the singular fiber $X_{w_0}^g$ is of type $I_1$ and $w_0=0$, the monodromy along a simple loop around~$0$ maps the basis $(1, \tau_X(w))$ to $(1, \tau_X(w)+1)$. Moreover,
$X^\#_U$ is biholomorphic to the quotient of $U\times \C$ by the family of lattices 
$\Z\oplus \Z\tau_X(w)$, where $\tau_X(w) = \unsur{2\ii \pi} \log (k_X(w))$ for a function $k_X\colon U\to \C$ which has a single zero at the origin and depends holomorphically on $X\in V$ and $w\in U$.
Since  
$g\circ \varsigma_X$ is  another section of $\pi$ above $U$, there is a holomorphic function $t_X(w)$ 
of $X$ and $w$ such that the lift of $g$ to  $U\times \C$ is given by $(w, z)\mapsto (w, z+ t_X(w))$. 
The calculations of \cite{invariant} 
(see  \S 3.3.2  and Lemma 3.11 there) show that the equation for $\mathrm{NT}_g$ in $U$  is 
\begin{equation}\label{eq:NT}
- \ii \log(\abs{k_X(w)}) k_X(w) \, t_X'(w) = k_X'(w) \Ima(t_X(w)).
\end{equation}
We claim that if $\Ima(t_X(0)) \neq  0$,  then by reducing $V$ and $U$ if necessary,   
$\mathrm{NT}_g\cap X_U^g = \emptyset$. Indeed if  $U$ is small enough, 
there exist  positive   constants  $\e$, $c$ such that for any $X\in V$, 
\begin{equation}
\abs{k_X(w)}\leq \e, \;\, \abs{k_X'(w)}\geq c, \;\,
\abs{\Ima(t_X(w))}\geq c, \;\, {\text{and}} \;\, \abs{t_X'(w)}\leq c\inv.
\end{equation}
Reducing $U$ further, 
 $\e$  can be chosen arbitrary small    while 
$c$ remains bounded away from~$0$. If $\e\log\e < c^3$, this is not 
compatible with the equality~\eqref{eq:NT}, so $\mathrm{NT}_g\cap U = \emptyset$.  
 
\begin{lem}\label{lem:twisting_singular}
The locus 
\begin{equation}\label{eq:analytic_equation}
\set{X\in V: \  \Ima(t_X(w_0)) = 0}
\end{equation}
is a  semi-algebraic subset  of positive codimension. 
\end{lem}

\begin{proof}[Proof of Lemma~\ref{lem:twisting_singular}]
 Consider the Wehler surfaces $X\subset V\subset \cW_1$, and their equations 
 \begin{equation}
   A_{222} x^2y^2z^2+ A_{221}x^2y^2z+\cdots + A_{100}x + A_{010}y+A_{001}z + A_{000}   = 0. 
\end{equation}
Permuting coordinates if necessary, we suppose that  $\pi\colon X\to \P^1$ is the projection onto the first coordinate.  
As $X$ varies near $X_0$, the critical value of $\pi$ near $w_0$  and the corresponding critical point in $X$ can be computed algebraically in terms of the $A_{ijk}$. Using  the action of $\PGL(2, \C)^3$ on $\P^1\times \P^1\times \P^1$, we may 
assume that  $w_0=0$ (as above) and the unique singular point of the fiber $X_{w_0}^g:=X \cap \set{x= 0}$ is  $(0,0)$. So, the equation of $X_{w_0}^g$ in $\P^1\times \P^1$ is    
\begin{equation}
ay^2z^2 + b y^2z+cy z^2+ dyz+ey^2+f z^2  = 0,
\end{equation}
for some coefficients $a, \ldots , f$ given by algebraic expressions in the $A_{ijk}$. 
Since $X\in \cW_1$,  $X_{w_0}^g$ has two transverse branches at  $(0,0)$: their tangent directions are given 
by the solutions of $dyz+ey^2+f z^2  = 0$ in $\P^1$. 

One can also write $X_0^g\setminus\{0,0\}$ as the quotient of $\set{0}\times \C$  by the lattice $\Lat(0)=\Z$; in this coordinate, 
$g$ acts as multiplication by $\exp(2\ii \pi t_X(0))$. Thus, $\Ima(t_X(0)) = 0$ means that $g$  induces a rotation, instead of a loxodromic homography, on the rational curve $X_0^g$. Writing down $g = \sigma_y\circ \sigma_z$ in coordinates, we obtain 
\begin{equation}
g(y,z) = \begin{pmatrix}-1 - \frac{d^2}{ef}& \frac{d}{e} \\ -\frac{d}{f} & -1
\end{pmatrix}\begin{pmatrix} y\\z\end{pmatrix} + O\lrpar{\norm{(y,z)}^2}
\end{equation}
for $(y,z)\in X_0^g$.
Thus,  $D_{(0, 0)}g\in \GL(T_{0,0}X)$ has determinant $1$ and  trace $ - 2- \frac{d^2}{ef}$.
As a consequence,  $g$ acts as a rotation on $X_0^g$ if and only if $ 2 + \frac{d^2}{ef} \in[-2, 2]$: this is a semi-algebraic condition. 
\end{proof}

To conclude, we   let $ \cW_g\subset  \cW_0$ be  
 the intersection of $ \cW_1$ with the complement of the subsets 
 $\set{X\in  \cW_1; \ \Ima(t_X(w_i) = 0}$ for each of the 24 singular values $w_i$ of $\pi$.  
 We finally define $\mathcal Z_g$ to be the complement of $ \cW_g$; by Lemma~\ref{lem:twisting_singular}, it is a proper 
 semi-algebraic set of positive codimension.

  \smallskip 
  
 \noindent{\textbf{Step 2: }}\textit{Conclusion. } 
 Pick $X_0\in \cW_0\setminus \mathcal Z_g$ and  cover $\P^1$ by a finite family $F$ of   
topological disks, such that for every $U\in F$, 
$\overline U$ contains  at most one point of $\Crit(\pi)\cup \NT_g$.  
If $U\in F$ contains a critical value of $\pi$ (and no point of $\NT_g$), 
then, as already explained, this property persists in a neighborhood of $X_0$. By Step~1, for
  $X$ sufficiently 
  close to $X_0$, $\overline U$ is disjoint from $\NT_g$ as well. For the remaining disks, the local constancy of $\mult(\NT_g)$ follows from Lemma~\ref{lem:NT_local}. The proof is complete. 
\end{proof}

\begin{proof}[Proof of Proposition~\ref{pro:NT_bounded}]
We use a semi-continuity argument.
Since the exceptional set $\mathcal Z_g$ defined in Lemma~\ref{lem:effective_analytic} is semi-algebraic,
the open set  $\cW_0\setminus \mathcal Z_g$ is also  semi-algebraic,  so it 
admits finitely many connected components (see~\cite[Cor. 2.7]{bierstone-milman:ihes} for instance). 
Thus, by Lemma~\ref{lem:effective_analytic},   $\mult(\NT_g)$ and therefore $\abs{\NT_g}$   
are uniformly bounded 
on  $\cW_0\setminus \mathcal Z_g$, say $\abs{\NT_g}\leq B$. Now, pick $X_0\in \mathcal Z_g$ (thus $X_0\in \cW_0$) and assume 
that for $X_0$ one has $\abs{\NT_g}>B$. We can then consider a finite number of small topological 
disks $U_i$ with disjoint closures  in $\P^1$, such that $\abs{\NT_g}\cap \bigcup U_i>B$. By Lemma~\ref{lem:NT_local}, these non-twisting points persist for  
 $X$ close enough to $X_0$. Since  $\cW_0\setminus \mathcal Z_g$ is dense in $\cW_0$, this
  contradicts the definition of $B$ and the proof is complete. 
\end{proof}

 
\begin{proof}[Proof of Theorem~\ref{thm:wehler_zariski}]
The main point  of~\cite[Thm A]{finite_orbits} is 
that the set of $X\in  \cW_0$ possessing  a finite orbit of length $\leq B$  is a 
proper Zariski closed subset $\mathcal Z_B$ of $ \cW_0$.  For  $N$ as in 
  Proposition~\ref{pro:N(X)_bounded},   for 
any $X\in \cW_0\setminus \mathcal Z_N$, all finite orbits of $\Gamma$ are   uniformly expanding.  We
conclude by applying Theorem~\ref{thm:criterion_uniform_expansion_parabolic} (with $\nu = \unsur{3}(\delta_{\sigma_1}+ \delta_{\sigma_2}+\delta_{\sigma_3})$). The proof of the corresponding statement in $ \cW_0(\R)$ is identical. 
 \end{proof}

 \begin{rem} 
 We expect that  an analogue of Theorem~\ref{thm:wehler_zariski} holds for other families 
   with large automorphism groups containing parabolic elements, like Enriques surfaces, or the family associated to pentagon folding (see~\cite{Simons}). 
 \end{rem}

\begin{rem}  
The  proof of Proposition~\ref{pro:NT_bounded} suggests that 
there should exist a notion of multiplicity, {\textit{including singular fibers}},
for which $\mult(\NT_g)$ would be constant on $\cW_0$ and would be an algebraically 
computable invariant of the parabolic automorphism $g$. A variant of this question  is mentioned in \cite[Rmk 7.7.4]{duistermaat}. 
\end{rem}


\subsection{Thin subgroups}\label{subs:thin}
In this section we consider the total
space $ \cW$ of all Wehler surfaces and the universal family 
$\mathcal X\subset   \cW\times (\pu\times \pu\times \pu)$. We change a 
little bit the notation: $\Gamma$ will be a subgroup of $\Z/2\Z\ast\Z/2\Z\ast\Z/2\Z$, and  
$\Gamma_{X}$ will be  the corresponding subgroup of $\Aut(X)$. 
 
Let $E$ be an elliptic curve. Consider the following classical Kummer construction (see~\cite[\S 4]{finite_orbits}): let $\eta$ be  
the involution $\eta (x,y)\mapsto (-x,-y)$ on $A:=E\times E$; the associated Kummer surface is 
the desingularization $\widehat{E\times E/\eta}$; the natural $\GL(2, \Z)$ action on $E\times E$ descends 
to $E\times E/\eta$ and
induces a non-elementary automorphism group of $\widehat{E\times E/\eta}$.
The surface $E\times E/\eta$ can be realized  as a singular Wehler example 
(see~\cite[\S 8.2]{Cantat:Panorama-Synthese}); in addition the action of 
 $\Z/2\Z\ast\Z/2\Z\ast\Z/2\Z$ is induced by a finite index subgroup of $\GL(2, \Z)$. 
 Let us briefly recall the construction: write 
 $E$ in Weierstrass form $y^2 = 4x^3 - g_2x - g_3$, with the neutral element of the group law on 
 $E$ located at infinity. 
 To  $(m_1, m_2)\in E\times E$, $m_i  = (x_i, y_i)$, we associate $m_3 = - (m_1+m_2)$ and 
 $\phi(m_1, m_2 )  = (x_1, x_2, x_3)$, where $m_3 = (x_3, y_3)$. Then, $\phi$ is $\eta$-invariant and determines  
 a biregular map 
 $\phi: E\times E/\eta\to X_E$ onto a singular Wehler surface $X_E$ with $16$ nodal singularities.
  
 Assume that 
 $\Gamma\subset \Z/2\Z\ast\Z/2\Z\ast\Z/2\Z$ is not virtually cyclic. Then for $X\in  \cW_0$, $\Gamma_X$ is non-elementary
 (see~\cite[\S 3]{stiffness}). 
  
 \begin{thm}\label{thm:thin}
 Let $\Gamma$ be a subgroup of $\Z/2\Z\ast\Z/2\Z\ast\Z/2\Z$ which is not virtually cyclic.
 For $X\in  \cW_0$ sufficiently close to $X_E$, the subgroup $\Gamma_X$ is uniformly expanding on $X$. 
 \end{thm}

Thus for every  ``abstract'' non-elementary subgroup   $\Gamma$ of $\Z/2\Z\ast\Z/2\Z\ast\Z/2\Z$, the open 
subset $\cW_{\exp}(\Gamma)$ of those $X\in  \cW_0$ for which the action of $\Gamma_X$ is uniformly expanding is non-empty.  
The  group $\Gamma$ can be arbitrarily thin, in particular it is not assumed to contain parabolic elements.
  In view of Theorem~\ref{thm:criterion_uniform_expansion}, it is natural to expect that $\cW_{\exp}(\Gamma)$ is actually dense. 

\begin{proof}
 The  difficulty  is that we cannot directly argue that uniform expansion is an open property, because
 $X_E$ is singular.
 
\begin{lem}\label{lem:regular}
Fix $f\in \Z/2\Z\ast\Z/2\Z\ast\Z/2\Z$, and denote also by $f$ the induced fibered map on 
the universal family of $(2,2,2)$-surfaces in $(\P^1)^3$. 
Then $f$ is regular on a neighborhood of $X_E$. 
\end{lem}

\begin{proof}[Proof of Lemma~\ref{lem:regular}]  Pick a $(2,2,2)$ surface $X$. If $X$ does not contain any fiber 
of the projection $\pi_{12}=(\pi_1,\pi_2)\colon (\P^1)^3\to (\P^1)^2$, then the same property holds   in a neighborhood $\mathcal V$ of $X$ of the universal family of $(2,2,2)$-surfaces, and $\sigma_3$ determines an automorphism of $\mathcal V$. 
Thus, we only have to prove that $X_E$ does not contain any fiber of the projections $\pi_{ij}$. 
Let us show that $X_E$ does not contain any vertical line $\{x=x_0, y=y_0\}$. Such a line would provide a family of points $(m_1,m_2)$ on $E\times E$ with fixed first coordinates $x_1=x_0$, $x_2=y_0$, for which the first coordinate of $m_3:=-(m_1+m_2)$ takes arbitrary values. This is impossible. The same argument applies to the lines  $\{y=y_0, z=z_0\}$ and $\{z=z_0, x=x_0\}$ because the relation $m_1+m_2+m_3=0$ is symmetric (equivalently, the equation of $X_E$ given in~\cite[\S 8.2]{Cantat:Panorama-Synthese} is symmetric in $(x,y,z)$). 
\end{proof}

There is a finite index subgroup of $\Gamma$  that 
  fixes each singularity of $X_E$. By Proposition~\ref{pro:induced_UE2} and the fact that uniform expansion does not depend on the measure, we can replace $\Gamma$ with this finite index subgroup, endow $\Gamma$ with a finitely supported, symmetric measure $\nu$ with $\Gamma=\Gamma_\nu$, and then we have to prove that $(\Gamma_X, \nu_X)$ is uniformly expanding for $X\in  \cW_0$ near $X_E$; here, $\nu_X$ is the measure induced by $\nu$ on $\Gamma_X$. 
  
Endow $\P^1\times \P^1\times \P^1$ with the Fubini study metric, and the Wehler surfaces $X$ with the induced metric. Recall that $T^1X$ denotes the unit tangent bundle.

Assume, by way of   contradiction, that there is a   sequence    
  $X_n\to X_E$ along which $\nu_{X_n}$ is not uniformly expanding. 
  For each $n$, let $\varpi$ denote the natural  projection $T^1X_n\to X_n$.
 For $X_E$, denote by  $T^1X_E$ the subset of  $T^1(\pu\times \pu\times \pu)$ which coincides with $T^1 \mathrm{Reg}(X_E)$  above the regular part of $X_E$ and coincides with $T^1_x(\pu\times \pu\times \pu)$ above each singularity $x\in \Sing(X_E)$; again, we denote by $\varpi$ the projection $T^1X_E\to X_E$. 
  The proof of
  Theorem~\ref{thm:limit_measure_top_lyapunov} 
   provides a sequence of stationary measures $\hat \mu_{X_n}$ on $T^1X_n$ (with projections $\mu_{X_n}:=\varpi_*\hat\mu_{X_n}$) such that   
   \begin{equation}
  \int \log\norm{f_\varstar u} d\nu_{X_n}(f) \, d\hat \mu_{X_n} (u)\leq 0.
  \end{equation}

From Lemma~\ref{lem:regular}, we can extract a subsequence,   still denoted by  $(X_n)$,  
such that $\lrpar{\hat \mu_{t_n}}$ converges to a stationary measure $\hat \mu_{X_E}$ on  $T^1X_E$ satisfying
\begin{equation}
\int \log\norm{f_\varstar u} d\nu_{X_E}(f) \, d\hat \mu_{X_E} (u)\leq 0.
\end{equation}
By iterating and using the stationarity of $\hat \mu_{X_E}$, the same inequality holds with $\nu_{X_E}^{(m)}$ instead of $\nu_{X_E}$ for every $m>0$. To get the desired contradiction, we shall show that no such measure exists. 

\smallskip

\noindent{\bf{Step 1:}} {\textit{near the singularities}}.-- Here we show that there exists $n_0\in \N$, $c_0>0$,
and an open neighborhood $U$ of $\Sing (X_E)$ such that if $u\in T^1X_E$ and
$\varpi(u)\in U$, then 
\begin{equation}\label{eq:UE_sing}
\int \log\norm{f_\varstar u}\,  d\nu_{X_E}^{(n_0)}(f)  \geq c_0.
\end{equation} 
By continuity it is enough to prove this when 
$\varpi(u)\in \Sing(X_E)$. 
Recall that $\Gamma_{X_E}$ fixes $\Sing(X_E)$ pointwise. Around each of its singularities, $X_E$ is locally isomorphic to the  quotient
$\C^2/\eta$, $\eta(u,v)=(-u,-v)$, standardly embedded in $\C^3$ by 
\begin{equation}
\phi: (u,v)\mapsto (u^2, uv, v^2) = (x,y,z),
\end{equation}
whose image is the quadratic cone $\{xz-y^2 = 0\}\subset \C^3$.  The level-$2$ congruence 
subgroup $G$ of $\GL_2(\Z)$ 
fixes each torsion point of $A:=E\times E$ of order $\leq 2$, and $\Gamma_{X_E}$ is induced by  
a non-elementary subgroup $G_0$ of $G$. 
The standard linear action of $G$ on $\C^2$ (or more precisely on a neighborhood of any $2$-torsion point of $A$) commutes to $\eta$ and induces a linear action on $\C^3$ via the homomorphism 
\begin{equation}
\phi_\varstar: \begin{pmatrix} a&b\\ c&d\end{pmatrix} \longmapsto 
\begin{pmatrix} a^2 & 2ab &b ^2 \\ ac & ad+bc & bc \\ c^2 & 2cd & d^2\end{pmatrix}.
\end{equation}
Thus, the action of $\Gamma_{X_E}$ on the tangent cone $T_0X_E$ is, up to a linear conjugacy, given by $\phi_\varstar(G_0)$.
This  is a subgroup 
of $O(q;\R)\simeq O_{2,1}(\R)$, where $q$ is the quadratic form $q(x,y,z) = xz-y^2$. By assumption, it is a non-elementary group of isometries of $q$, hence it acts strongly irreducibly and proximally on $\R^3\subset \C^3$ (loxodromic elements of $\GL_2(\Z)$ are mapped to loxodromic elements in $O(q;\R)$). It preserves the real decomposition $ \C^3=\R^3\oplus_\R \ii \R^3$ and the action on $\R^3$ and $\ii \R^3$ are linearly conjugate (by multiplication by $\ii$). Therefore, as in \S~\ref{subs:finite_orbits}, the  Inequality~\eqref{eq:UE_sing} follows from~\cite{furstenberg-kifer}    (see also~\cite{bougerol-lacroix}, Chap. III, Cor. 3.4(iii)).  

\noindent{\bf{Step 2:}} {\textit{away from the singularities}}.--  
We shall show that there exists a neighborhood $U'\subset U$  of $\Sing(X_E)$ 
and $c>0$ such that 
for any fixed $u\in T^1X_E$   such that $\varpi(u)\notin U'$, 
\begin{equation}\label{eq:UE_proba}
  \pp\lrpar{\unsur{m}\log \norm{(f_\omega^m)_\varstar u}  \geq c} \underset{m\to\infty}{\longrightarrow} 1. 
\end{equation}
By Lemma~\ref{lem:weak_UE2} (see also Remark~\ref{rem:given_v}), this implies that 
$\ee\lrpar{\log \norm{(f_\omega^m)_\varstar u}}\geq mc/2$ for large $m$. Then,    the first step and 
a compactness argument identical to that of 
Lemma~\ref{lem:weak_UE} show  that uniform expansion holds on $T^1X_E$, 
which is the desired contradiction.  
  
Let $U'$ be an open neighborhood of $\Sing(X_E)$ which will be specified later. 
There is a constant $\delta  = \delta(U')$  such that 
\begin{equation}\label{eq:flat1}
 \text{if } \varpi(u)\notin U'\text{ and } \varpi(f_\varstar u)\notin U', \text{ then }
 \log\norm{f_\varstar u} \geq \log\norm{f_\varstar u}_{\mathrm{flat}}  - \delta,
 \end{equation} 
 where $\norm{\cdot}_{\mathrm{flat}}$ is the Riemannian metric on $\mathrm{Reg}(X_E)$ induced by the flat metric of 
$E\times E$. 

The pull-back of 
  $\nu$ to $\GL(2, \Z)$ generates $G_0$ and its support is finite. Since $G_0\subset \GL(2, \Z)$ is non-elementary, we have uniform expansion with respect to the flat metric.  By Lemma~\ref{lem:weak_UE2},
   there exists a constant $c_1>0$ and sets of trajectories 
  $\Omega_m^1\subset \Omega$ such that $\pp(\Omega_m^1)\to 1$ as $m\to \infty$
  and
  \begin{equation}\label{eq:flat2}
  \text{if }\omega\in \Omega_m^1, \quad
 \unsur{m}\log \norm{(f_\omega^m)_\varstar u}_{\mathrm{flat}} \geq c_1  .
  \end{equation}
Fix $\e>0$ and $0<c<c_1$. By the first step of the proof, $\Sing(X_E)$ is repelling  on average in $\pu\times \pu\times \pu$;
so there is a Margulis function on $X_E$ with poles at 
$\Sing(X_E)$. Thus, there is an open set $U' = U'(\e)\subset U$ with the following property:
for large enough $m$, 
 the set  $\Omega_m^2$ of trajectories $\omega\in \Omega$ such that 
 $(f_\omega^m)(\varpi(u)) \notin U'$ satisfies $\pp(\Omega_m^2)\geq 1-\e/2$. Now, 
 $U'$ being fixed, for $m$ large enough we have that $\pp(\Omega_m^1\cap \Omega_m^2)\geq 1-\e$ and 
 by~\eqref{eq:flat1} and~\eqref{eq:flat2}, if $\omega \in  \Omega_m^1\cap \Omega_m^2$,
\begin{equation}
  \unsur{m}\log \norm{(f_\omega^m)_\varstar u}\geq  c_1  - \frac{\delta(U')}{m}\geq c.
  \end{equation} 
  Thus, the convergence~\eqref{eq:UE_proba} holds and the proof is complete. 
\end{proof}

\section{Applications}\label{sec:applications}

\subsection{Orbit closures}
The following is a version of  the orbit closure Theorem~E of~\cite{invariant} in which  
periodic orbits are allowed. 
Combined with Theorem~\ref{thm:wehler_zariski}, it gives 
 Theorem~\ref{thm:orbits_wehler}.  

\begin{thm}\label{thm:orbit_closures}
Let $X$ be a torus,  a K3 surface, or an Enriques surface. 
Let $\Gamma\subset \Aut(X)$ be a non-elementary subgroup which contains parabolic elements  and does not
preserve any algebraic curve.  Assume 
that for any finite orbit $O$, the induced action of $\Gamma$ on $TX\rest{O}$ is non-elementary. 
 Then there exists a finite set $F$ and a real analytic,  
 totally real, and $\Gamma$-invariant  surface $Y \subset X$ with $\Sing(Y)\subset F$ 
 such that for every $x\in X$:
 \begin{enumerate}[{\em (a)}]
 \item either $x$ belongs to $F$ (and its orbit is finite);
 \item  or $x$ belongs to $Y\setminus F$ and $\overline{ \Gamma (x)}$ is a union of components of $Y$;
 \item or $\overline{ \Gamma (x)} =X$. 
 \end{enumerate}
\end{thm}

\begin{proof} 
First observe that under these assumptions,   \cite[Thm C]{finite_orbits} implies 
that there exists a maximal finite invariant subset $F$.   
Fix a symmetric measure $\nu$ such that $\Gamma_\nu = \Gamma$ and satisfying  the  
  moment condition~\eqref{eq:moment+}.
By Theorems~\ref{thm:criterion_uniform_expansion_parabolic} and~\ref{thm:finite_UE}, 
$\nu$ is uniformly expanding. We now resume the discussion   from~\cite[ \S 8]{invariant}, in particular Remark~8.6 there: $\STang_\Gamma$ is a finite invariant set and if $x\in X$ is such that $\Gamma(x)$ is infinite 
but not dense, then there are two possible situations:
\begin{enumerate}
\item either $\overline{\Gamma(x)}\setminus \STang_\Gamma$ is discrete outside $\STang_\Gamma$;
\item or $\overline{\Gamma(x)}\setminus \STang_\Gamma =: Y(x)$ is a totally real analytic surface, whose singular locus $\Sing(Y)$ is discrete outside $\STang_\Gamma$. 
\end{enumerate}

In Case~(1),   $\Gamma(x)$ is finite. 
Indeed $\overline {\Gamma(x)}$ is at most countable, 
so if $\mu$ is any cluster value of 
$\unsur{n}\sum_{k=0}^{n-1} \nu^k\star\delta_x$, then $\mu$ is a purely atomic 
stationary measure. In this case it follows from Theorem~\ref{thm:margulis_UE} that  
 the orbit of $x$ must be finite, hence contained in $F$.

If Case~(2) holds, we first claim that $\Sing(Y(x))$ is finite. 
Indeed, $\overline  {\Sing(Y(x))}$ is a $\Gamma$-invariant countable set, which clusters only at 
$\STang_\Gamma$. By the previous argument, every orbit $\Gamma(y)$ in $\overline  {\Sing(Y(x))}$ is finite, so by the 
finiteness of the set of finite orbits~\cite[Thm C]{finite_orbits} we conclude that 
 $\Sing(Y(x))$ itself is finite. 
  
 Now, let $\mu'$ be a cluster value of 
 $\unsur{n}\sum_{k=0}^{n-1} \nu^k\star\delta_x$. By Theorem~\ref{thm:margulis_UE}, $\mu'$
 is an atomless  stationary measure 
 supported on $Y(x)$ such that $\mu'(\mathrm{Reg}(Y(x)))=1$. Since $\Gamma$ has no invariant 
 curve,  $\mu'$ is Zariski diffuse. Let $\mu$ be any ergodic component of $\mu'$. 
  Theorems~\ref{thm:rigidity} and~\ref{thm:hyperbolic} imply that $\mu$ is hyperbolic and its stable directions  depend genuinely on the itinerary. Then the 
 argument of \cite[Thm 3.1]{br} adapts immediately to show that $\mu$ is 
 SRB(\footnote{We are \emph{not} claiming that we can extend \cite{br} to non-compact surfaces here.  All the necessary 
  estimates  on the Lyapunov  norms and Pesin charts   
  hold by viewing  $\mu$ as a hyperbolic stationary measure on the 
 \emph{compact} complex manifold $X$. The only issue appears when considering the size and intersection 
 properties of    real stable and unstable manifolds in $Y(x)$, starting from \S 9.7 of \cite{br}. 
 At this stage Brown and Rodriguez-Hertz 
 already discard a set of small measure of points with bad properties (see the definition of $\Lambda(\gamma_1)$ on p. 1087); so it is enough to remove from this $\Lambda(\gamma_1)$ the set of 
 small measure of points too close to $\Sing(Y(x))$, and proceed with their argument.}).  The 
 canonical invariant 2-form of $X$ induces a $\Gamma$-invariant measure  $\vol_{Y(x)}$ 
 on $Y(x)$ (see Lemma~\ref{lem:volume_Y}).  Since $\mathrm{Reg}(Y(x))$ admits a Margulis function, we conclude from 
 Proposition~\ref{pro:finite_mass} that the volume $\vol_{Y(x)}$  is finite. Therefore we can copy verbatim the 
 argument of \cite[Thm 3.4]{br} to conclude that $\mu$ is $\Gamma$-invariant. 
Since \cite[Thm C]{invariant} says that there are only finitely many $\Gamma$-invariant measures, 
  there are only finitely many possible  surfaces $Y(x)$. Taking $Y$ to be their union, 
  the proof is complete.  
\end{proof}

\subsection{Ergodicity} \label{subs:ergodicity}
In~\cite{dolgopyat-krikorian}, the original motivation to introduce uniform 
expansion was a criterion for ergodicity. The same holds  in our setting, with a few caveats which will be explained below. 

\begin{thm}[{Dolgopyat-Krikorian~\cite[Cor. 2]{dolgopyat-krikorian}}, see also~\cite{xliu, chung}]
\label{thm:dolgopyat-krikorian}
Let $X$ be a torus, a K3 surface, or an Enriques surface. Let $\Gamma\subset \Aut(X)$ be a 
non-elementary subgroup with a uniformly expanding action on $X$. Then $\vol_X$ is $\Gamma$-ergodic. 

Likewise, if $Y\subset X$ is a $\Gamma$-invariant totally real 
analytic subset such that $\Gamma$  acts transitively on the 
set of irreducible components of $Y$, then $\vol_Y$ is $\Gamma$-ergodic.   
\end{thm}

Note   that the notion of irreducible component in real analytic geometry is not well-behaved in general 
(see~\cite[\S 5.1]{invariant} for a short discussion). 
Here we content ourselves with saying that $Y$ is irreducible when $\mathrm{Reg}(Y)$ is connected. 
 Observe also that the ergodicity of 
$\vol_X$ follows directly from Theorem~\ref{thm:classification_invariant} when $\Gamma$ 
contains a parabolic element.

\begin{proof}[Proof (sketch)] 
The proof in~\cite{dolgopyat-krikorian} is a bit sketchy, but it was already 
expanded  in~\cite{xliu, chung} (see also~\cite{zzhang}). 
Here we just make a few comments on (1) the extension to the holomorphic case for the action on $X$, 
and (2) how to deal with the possible  singularities for the action on totally real surfaces $Y$. 

Regarding the action on $X$, let us recall that the proof of~\cite{dolgopyat-krikorian} is a 
variation on the Hopf  argument in which the asymptotic behavior of the Birkhoff 
sums $  \unsur{n}\sum_{k=0}^{n-1} \delta_{f_\omega^n(x)}$ is propagated along chains of 
local stable manifolds (associated to different $\omega$'s), to ultimately conclude that there 
is a uniform $r$ so that almost every point $x$ 
is located at distance at least $r$ from the boundary of its ergodic component.  
The key technical ingredients are     the facts that under the   uniform expansion assumption:
\begin{itemize}
\item  stable directions at a given point do not concentrate, more precisely there exists 
$\alpha>0$ such that for any $x\in X$ and any $[v]\in \pp(T_x X)$, the probability that 
$d_{\P^1}\lrpar{[v], [E^s(x,\omega) ]} <\alpha$ is smaller than 1/100: this follows from a compactness argument 
(see~\cite[Prop. 4.4.4]{xliu});
\item  the Pesin local stable manifolds have uniformly bounded geometry 
(e.g. uniformly bounded size in the sense of 
  \cite[\S 7.4]{stiffness}): this follows from the usual proof of the local stable manifold theorem;
\item the absolute continuity of the local stable foliation in Pesin charts: we can copy the usual proof 
or notice that in the holomorphic case this follows from the fact that the holonomy of a holomorphic 
motion is quasiconformal. 
\end{itemize}
Given these facts, we can copy the proof of~\cite{dolgopyat-krikorian}  by plugging in \S 10.4 the 
 the following elementary geometric property, whose proof is left to the reader: 
let $w = (w_1, w_2)\in \C^2$ with $\norm{w}<1$ (possibly close to 1)
and $E_w$ be the direction  perpendicular to the line $(0w)$; then if  
$L$ is a complex line containing $w$, such that the angle in $\P^1$ 
between the direction of 
$L$ and $E_w$ is greater than $\alpha$, then $L\cap B(0, 1)$ contains a disk of radius $r\geq r(\alpha)$.  

For the second statement of the theorem we can directly resort to~\cite{xliu, chung}, except that 
we have to take into account the possibility of singular points on $Y$, which affect the size and geometry of local stable manifolds on $Y$.  For this, we may argue exactly as in Theorem~\ref{thm:orbit_closures}: 
first, the existence of a Margulis function guarantees that $\vol_Y$ is finite. Next, since 
uniform expansion holds on $X$, the size and angle change of local complex stable manifolds is uniformly controlled. Thus, when restricting to $Y$, we also have a uniform control of this geometry outside any 
$\delta$-neighborhood of $\mathrm{Sing}(Y)$. Since the Hopf argument is local, we get that there is  a single 
ergodic component outside a $\delta$-neighborhood of $\mathrm{Sing}(Y)$, for every $\delta>0$, and we conclude by letting $\delta$ tend to zero. 
\end{proof}

\begin{rem}
The argument  of~\cite{dolgopyat-krikorian} works for a random dynamical system on a  (real) compact
 $2d$-dimensional manifold enjoying a uniform expansion property 
 along $d$ dimensional tangent subspaces. This assumption 
 does \emph{not} hold in our setting since along a totally real subspace one  
  may witness both expansion and 
 contraction. In particular the complex uniform expansion condition is not stable under 
  $C^1$ perturbations by (real) volume preserving diffeomorphisms of $X$. 
  Still, the philosophy of the above proof  is that 
 the argument is robust enough so that uniform expansion along complex 1-dimensional 
 tangent subspaces in a 2-dimensional complex surface guarantees ergodicity. 
\end{rem}

\subsection{Equidistribution}

In the following results, given an action of $(\Gamma, \nu)$ on $M$
 we say that  
 that random trajectories from $x$ equidistributes towards 
 $\mu$ if for 
 $\nu^\N$-almost every $\omega$
$\unsur{n}\sum_{k=1}^n \delta_{f^k_\omega} \to \mu$, where the convergence is in the 
weak$^\varstar$ topology. By averaging with respect to $\nu^\N$ and applying 
the Dominated Convergence Theorem, this implies that 
 $\unsur{n}\sum_{k=1}^n \nu^k\ast\delta_x\to\mu$ as well. 

The following theorem already appears under stronger moment assumptions in \cite[Thm D]{chung}. 

\begin{thm} 
Let $X_\R$ be a smooth 
real projective surface and $\nu$ a probability measure on $\aut(X_\R)$ satisfying \eqref{eq:moment+}. 
Assume that $\Gamma_\nu$ preserves a smooth volume form $\vol$ on $X(\R)$ and that $\nu$ is uniformly 
expanding on $X(\R)$. Then for any $x\in X$ one of the following alternatives holds:
\begin{enumerate}[(a)]
\item either $\Gamma_\nu\cdot x$ is finite;
\item or the random trajectories from $x$  equidistribute 
 towards $\vol_{X'(\R)}$, the normalized induced volume on a 
 union of components of $X(\R)$. 
\end{enumerate}
\end{thm}

Recall from  Theorem~\ref{thm:criterion_uniform_expansion_parabolic} that the uniform expansion assumption 
holds when $\Gamma_\nu$ 
contains parabolic elements, has no invariant curve, and that the induced 
action of $\Gamma_\nu$ on finite orbits is uniformly expanding. In this case  by \cite[Thm C]{finite_orbits}, 
the number of finite orbits is finite. By Theorem~\ref{thm:wehler_zariski}  this applies to generic real Wehler 
surfaces and yields Theorem~\ref{thm:equidistribution_wehler}. 

 \begin{proof}
 The  result follows directly from the classification of stationary measures in~\cite{stiffness}, the ergodicity Theorem~\ref{thm:dolgopyat-krikorian} , 
 and the existence of a Margulis function (Theorem~\ref{thm:margulis}). 
 \end{proof}

The next result is conditional to the $\nu$-stiffness property of 
complex non-elementary uniformly expanding actions. We expect  that 
  it will be established in the near future. 

\begin{thm}\label{thm:equidistribution_complex}
Let $X$ be a K3 or Enriques surface and $\nu$ be a finitely supported measure on $\aut(X)$. Assume that 
\begin{enumerate}
\item $\Gamma_\nu$ is non-elementary, contains parabolic elements, has no invariant curve, 
and every finite $\Gamma$-orbit is uniformly expanding;
\item $\nu$-stiffness holds, that is, every $\nu$-stationary measure is invariant. 
\end{enumerate}
Then there exists a finite set $F$ and a (possibly singular) totally real analytic surface $Y$ such that  for every 
$x\in X$:
\begin{enumerate}[(a)]
\item either $x$ belongs to $F$;
\item or $x$ belongs to $Y\setminus F$ and its orbit equidistributes towards $\vol_{Y'}$, where $Y'$ is a union of components of $Y$;
\item or $x\notin F\cup Y$ and its orbit equidistributes towards $\vol_{X}$. 
\end{enumerate}
\end{thm}

\begin{proof}
The sets $F$ and $Y$ were already constructed in Theorem~\ref{thm:orbit_closures}, whose proof 
also implies property~(b).   
The classification of invariant measures (Theorem~\ref{thm:classification_invariant}) and the 
stiffness property show  that the only $\nu$-stationary measure giving no mass to 
$Y\cup F$ is  $\vol_X$. Therefore the equidistribution property~(c) follows from Breiman's ergodic theorem 
and the existence of a Margulis function associated to  finite orbits and totally real surfaces (Theorems~\ref{thm:margulis} and~\ref{thm:margulis_totalement_reel}).
\end{proof}

\begin{rem}
Note that if the maximal invariant totally real surface 
 $Y$ is known to be smooth (e.g. if  $Y = X(\R)$) 
 then we may relax the finiteness assumption on  $\nu$ 
to the moment condition \eqref{eq:moment+}. 
\end{rem}

\appendix

\section{Rigidity of zero entropy measures}\label{app:zero_entropy}

Here we complete the proof of Theorem~\ref{thm:rigidity} by establishing the following result
 of independent interest. 

\begin{thm}\label{thm:entropy_every_f}
Let $X$ be a torus or a K3 or Enriques surface,     and $\nu$ be a probability measure on 
$\Aut(X)$ such that $\Gamma_\nu$ is non-elementary. Assume that $\mu$ is a Zariski diffuse 
  $\nu$-stationary measure such that $h_\mu(X, \nu) = 0$.  
Then $\mu$ is $\Gamma_\nu$-invariant and for every 
$f\in \Gamma_\nu$, $h_\mu(f)=0$. 
\end{thm}

\begin{proof}
As in Theorem~\ref{thm:rigidity}, we may assume that $\mu$ is ergodic as a stationary measure, and 
its $\Gamma_\nu$-invariance was already established there. 
 Pick   $f\in \Gamma_\nu$. 
Assume by way of contradiction that $h_\mu(f)>0$, 
in particular $f$ must be loxodromic. 
If $\mu$ is ergodic for $f$, then the result follows rather immediately from the 
measure rigidity theorem 11.1 in~\cite{stiffness}. Indeed in that theorem 
we consider  an ergodic measure $\mu$
of positive entropy for $f$ and study the group of automorphisms of $X$ preserving $\mu$, 
under the additional assumption that $\mu$ is supported on a real surface. 
We reduce the argument to the case of $\Gamma = \langle f, g\rangle$ for some $g$, and divide the proof into 3 cases: (1) either there is a $\Gamma$-invariant measurable line field, or (2) 
there is a $\Gamma$-invariant pair of  measurable line fields, or (3) none of the above. In cases (1) and (2)
we conclude that $\Gamma$ is elementary by adapting the argument of~\cite[Thm. 9.1]{stiffness}: this 
does not rely on the additional real structure. In case (3), since $\mu$ is hyperbolic for $f$, Theorems~\ref{thm:ledrappier_invariance_principle} and~\ref{thm:classification_proj_invariant} imply that 
$\mu$ is hyperbolic as a stationary measure and as in the proof of Theorem~\ref{thm:rigidity} we deduce 
 that $h_\mu(X, \nu)>0$, which 
is contradictory. Thus,  case (3) does not happen, and we deduce that $\Gamma$ is elementary for every 
$g\in \Gamma_\nu$, which is a contradiction. Therefore $h_\mu(f) = 0$. 
 
What remains to do is to adapt this argument to the case where $\mu$ is not ergodic under $f$. 
 So  consider $f\in \Gamma_\nu$ and assume that $h_\mu(f)>0$ so that $f$ is loxodromic. 
As before there are 3 cases: either (1) there is a $\Gamma_\nu$-invariant measurable line field, or (2) 
there is a $\Gamma_\nu$-invariant pair of  measurable line fields, or (3) none of (1) and (2). 
We first observe that as before Case 3 does not happen: indeed if there is no  invariant line field or pair of invariant line fields, by Theorems~\ref{thm:ledrappier_invariance_principle} and~\ref{thm:classification_proj_invariant}, either $\mu$ is hyperbolic as a $\nu$-stationary measure, or the 
projectivized tangent action of $\Gamma_\nu$ reduces to a compact subgroup. But since 
$h_\mu(f)>0$, $f$ admits non-zero Lyapunov exponents on a set of positive measure 
so the latter is impossible. Hence $\mu$ is hyperbolic as a $\nu$-stationary measure, and since 
there is no invariant line field, stable directions depend on the itinerary and as before we conclude that 
$h_\mu(X, \nu) > 0$, a contradiction. So one of Cases (1) or (2) holds. 

So assume there exists a measurable $\Gamma_\nu$-invariant line field 
$x\mapsto [E(x)]\in \pp(T_xX)$ and   pick  $g\in \Gamma_\nu$.  Assume further that $g$ is loxodromic. 
We will derive a contradiction by 
 showing  that $\langle f, g\rangle$ must be  elementary:   this is a contradiction because 
 any non-elementary subgroup of $\Aut(X)$ contains a purely loxodromic non-elementary subgroup. 
 Let $\mathcal P$ be the measurable partition  into ergodic components (under $f$) and 
denote by $\mu_P$ the conditional 
measure on $P\in \mathcal P$, so that  
that $\mu   = \int \mu_{\mathcal P(x)} d\mu(x)$ is the ergodic decomposition of $\mu$. 
Since the entropy function is affine, there exists  a $f$-invariant 
 set $B$ of positive measure such that 
for any $x\in B$, $h_{\mu_{\mathcal P(x)}}(f)>0$. In particular $f$ is non-uniformly hyperbolic along $B$, so  
along $B$, $E$ must coincide almost everywhere with one of $E^s_f$ or $E^u_f$. Reducing $B$ to a smaller 
invariant subset we may 
assume that $E = E^s_f$ almost everywhere along $B$. 
For every $n\in \Z$, the automorphism $g^{-n} f g^n$ is loxodromic, preserves $\mu$,   is 
non-uniformly hyperbolic along $g^{-n}(B)$, and $E$ coincides with $E^s_{g^{-n}fg^n}$ almost everywhere
By measure preservation 
there exists $m\neq n$ such that $\mu(g^{-n}(B)\cap g^{-m}(B))>0$, so 
$\mu(B\cap g^{m-n}(B))>0$. Letting $h = g^{m-n}fg^{-(m-n)}$ and 
$A  = B\cap g^{m-n}(B)$ we are exactly in the situation of Lemma 11.2 of~\cite{stiffness}, and 
we conclude that $W^s(f, x) = W^s(h, x)$ for $\mu$-almost every $x\in A$, from which it follows that 
$T^+_f= T^+_h$ and finally $(g^{m-n})^*T_f^+  = c T^+_f$. Since $g$ is loxodromic, this implies 
that $T^+_f = T^+_g$ or  $T^+_f = T^-_g$, and finally that $\langle f, g\rangle$ is elementary, which is 
the sought-after contradiction. 

Finally, if there is a measurable pair  $\set{E_1, E_2}$ of    line fields which is $\nu$-a.s. 
invariant, we get a $f$-invariant set $B$ of positive measure along which 
$\set{E_1(x), E_2(x) } = \set{E^s_f(x) = E^u_f(x)}$, and  a set 
$A =  B\cap g^{m-n}(B)$ of positive measure along which 
$\set{E^s_f(x) = E^u_f(x)} = {E^s_h(x) = E^u_h(x)}$, where $h = g^{m-n}fg^{-(m-n)}$, and we conclude as before. 
\end{proof}


\bibliographystyle{acm}
\bibliography{biblio-hyperbolic}

\end{document}